\documentclass[11pt]{preprint}
\usepackage{difftrees} 
\usepackage[full]{textcomp}
\usepackage[osf]{newtxtext} 
\usepackage[cal=boondoxo]{mathalfa}
\usepackage{colortbl}
\usepackage{tikz,tikz-cd}
\usetikzlibrary{cd} 
\usepackage[all,cmtip]{xy}
\usepackage{comment}

\usepackage{amssymb}
\usepackage{mathtools}
\usepackage{mhenvs}
\usepackage{mhequ} 
\usepackage{mhsymb}
\usepackage{mathtools}
\usepackage{hyperref}
\usepackage{breakurl}
\usepackage{bm}

\usepackage{booktabs}
\usepackage{tikz}
\usepackage{tcolorbox}
\usepackage{mathrsfs}
\usepackage[utf8]{inputenc}
\usepackage{longtable}
\usepackage{wrapfig}
\usepackage{rotating} 
\usepackage{subcaption}
\usepackage{mathrsfs}
\usepackage{epsfig}
\usepackage{microtype}
\usepackage{comment}
\usepackage{wasysym}
\usepackage{centernot}
\usepackage{enumitem}
\usepackage{bm}
\usepackage{stackrel}

\usepackage{graphicx}
\usetikzlibrary{calc}
\usetikzlibrary{decorations}
\usetikzlibrary{positioning}
\usetikzlibrary{shapes}
\usetikzlibrary{external}
\usetikzlibrary{decorations.pathmorphing}
\usetikzlibrary{decorations.text}
\usetikzlibrary{snakes}

\def\geo{{\text{\rm\tiny geo}}}
\def\ngeo{{\text{\rm\tiny non-geo}}}
\def\id{\mathrm{id}}

\tikzstyle{tinydots}=[dash pattern=on \pgflinewidth off \pgflinewidth]
\tikzstyle{superdense}=[dash pattern=on 4pt off 1pt]

\newif\ifdark
\darkfalse

\ifdark

\definecolor{darkred}{rgb}{0.9,0.2,0.2}
\definecolor{darkblue}{rgb}{0.7,0.3,1}
\definecolor{darkgreen}{rgb}{0.1,0.9,0.1}
\definecolor{franck}{rgb}{0,0.8,1}
\definecolor{pagebackground}{rgb}{.15,.21,.18}
\definecolor{pageforeground}{rgb}{.84,.84,.85}
\pagecolor{pagebackground}
\AtBeginDocument{\globalcolor{pageforeground}}
\definecolor{symbols}{rgb}{0,0.7,1}
\colorlet{connection}{red!80!black}
\colorlet{boxcolor}{blue!50}

\else

\definecolor{darkred}{rgb}{0.7,0.1,0.1}
\definecolor{darkblue}{rgb}{0.4,0.1,0.8}
\definecolor{darkgreen}{rgb}{0.1,0.7,0.1}
\definecolor{franck}{rgb}{0,0,1}
\definecolor{pagebackground}{rgb}{1,1,1}
\definecolor{pageforeground}{rgb}{0,0,0}
\colorlet{symbols}{blue!90!black}
\colorlet{connection}{red!30!black}
\colorlet{boxcolor}{blue!50!black}

\fi

\makeatletter
\pgfdeclareshape{crosscircle}
{
	\inheritsavedanchors[from=circle] 
	\inheritanchorborder[from=circle]
	\inheritanchor[from=circle]{north}
	\inheritanchor[from=circle]{north west}
	\inheritanchor[from=circle]{north east}
	\inheritanchor[from=circle]{center}
	\inheritanchor[from=circle]{west}
	\inheritanchor[from=circle]{east}
	\inheritanchor[from=circle]{mid}
	\inheritanchor[from=circle]{mid west}
	\inheritanchor[from=circle]{mid east}
	\inheritanchor[from=circle]{base}
	\inheritanchor[from=circle]{base west}
	\inheritanchor[from=circle]{base east}
	\inheritanchor[from=circle]{south}
	\inheritanchor[from=circle]{south west}
	\inheritanchor[from=circle]{south east}
	\inheritbackgroundpath[from=circle]
	\foregroundpath{
		\centerpoint%
		\pgf@xc=\pgf@x%
		\pgf@yc=\pgf@y%
		\pgfutil@tempdima=\radius%
		\pgfmathsetlength{\pgf@xb}{\pgfkeysvalueof{/pgf/outer xsep}}%
		\pgfmathsetlength{\pgf@yb}{\pgfkeysvalueof{/pgf/outer ysep}}%
		\ifdim\pgf@xb<\pgf@yb%
		\advance\pgfutil@tempdima by-\pgf@yb%
		\else%
		\advance\pgfutil@tempdima by-\pgf@xb%
		\fi%
		\pgfpathmoveto{\pgfpointadd{\pgfqpoint{\pgf@xc}{\pgf@yc}}{\pgfqpoint{-0.707107\pgfutil@tempdima}{-0.707107\pgfutil@tempdima}}}
		\pgfpathlineto{\pgfpointadd{\pgfqpoint{\pgf@xc}{\pgf@yc}}{\pgfqpoint{0.707107\pgfutil@tempdima}{0.707107\pgfutil@tempdima}}}
		\pgfpathmoveto{\pgfpointadd{\pgfqpoint{\pgf@xc}{\pgf@yc}}{\pgfqpoint{-0.707107\pgfutil@tempdima}{0.707107\pgfutil@tempdima}}}
		\pgfpathlineto{\pgfpointadd{\pgfqpoint{\pgf@xc}{\pgf@yc}}{\pgfqpoint{0.707107\pgfutil@tempdima}{-0.707107\pgfutil@tempdima}}}
	}
}
\makeatother

\DeclareMathOperator{\NT}{\nabla{}Trees}
\DeclareMathOperator{\Tw}{Tw}

\DeclareMathOperator{\LA}{LieAdm}
\DeclareMathOperator{\ComMag}{ComMag}
\DeclareMathOperator{\PL}{PreLie}
\DeclareMathOperator{\MC}{MC}
\DeclareMathOperator{\Lie}{Lie}

\tikzset{
	cross/.style={path picture={ 
			\draw[symbols]
			(path picture bounding box.south east) -- (path picture bounding box.north west) (path picture bounding box.south west) -- (path picture bounding box.north east);
	}},
	root/.style={circle,fill=green!50!black,inner sep=0pt, minimum size=1.2mm},
	dot/.style={circle,fill=pageforeground,inner sep=0pt, minimum size=1mm},
	blank/.style={circle,fill=white,inner sep=0pt, minimum size=1mm},
	dotred/.style={circle,fill=pageforeground!50!pagebackground,inner sep=0pt, minimum size=2mm},
	var1/.style={circle,fill=pagebackground,draw=pageforeground,inner sep=0pt, minimum size=2mm},
	var/.style={circle,fill=pageforeground!10!pagebackground,draw=pageforeground,inner sep=0pt, minimum size=3mm},
	sqvar/.style={rectangle,fill=pageforeground!10!pagebackground,draw=pageforeground,inner sep=0pt, minimum size=3mm},
	kernel/.style={semithick,shorten >=2pt,shorten <=2pt},
	kernels/.style={snake=zigzag,shorten >=2pt,shorten <=2pt,segment amplitude=1pt,segment length=4pt,line before snake=2pt,line after snake=5pt,},
	kernelsbl/.style={snake=zigzag,shorten >=2pt,shorten <=2pt,segment amplitude=1pt,segment length=4pt,line before snake=2pt,line after snake=5pt,bend left=60},
	rho/.style={densely dashed,semithick,shorten >=2pt,shorten <=2pt},
	testfcn/.style={dotted,semithick,shorten >=2pt,shorten <=2pt},
	renorm/.style={shape=circle,fill=pagebackground,inner sep=1pt},
	labl/.style={shape=rectangle,fill=pagebackground,inner sep=1pt},
	xic/.style={very thin,circle,draw=symbols,fill=symbols,inner sep=0pt,minimum size=1.2mm},
	g/.style={very thin,rectangle,draw=symbols,fill=symbols!10!pagebackground,inner sep=0pt,minimum width=2.5mm,minimum height=1.2mm},
	xi/.style={very thin,circle,draw=symbols,fill=symbols!10!pagebackground,inner sep=0pt,minimum size=1.2mm},
	xies/.style={very thin,rectangle,fill=green!50!black!25,draw=symbols,inner sep=0pt,minimum size=1.1mm},
	xiesf/.style={very thin,rectangle,fill=green!50!black,draw=symbols,inner sep=0pt,minimum size=1.1mm},
	xix/.style={very thin,crosscircle,fill=symbols!10!pagebackground,draw=symbols,inner sep=0pt,minimum size=1.2mm},
	X/.style={very thin,cross,rectangle,fill=pagebackground,draw=symbols,inner sep=0pt,minimum size=1.2mm},
	xib/.style={thin,circle,fill=symbols!10!pagebackground,draw=symbols,inner sep=0pt,minimum size=1.6mm},
	xie/.style={thin,circle,fill=green!50!black,draw=symbols,inner sep=0pt,minimum size=1.6mm},
	xid/.style={thin,circle,fill=symbols,draw=symbols,inner sep=0pt,minimum size=1.6mm},
	xibx/.style={thin,crosscircle,fill=symbols!10!pagebackground,draw=symbols,inner sep=0pt,minimum size=1.6mm},
	kernels2/.style={very thick,draw=connection,segment length=12pt},
	keps/.style={thin,draw=symbols,->},
	kepspr/.style={thick,draw=connection,->},
	krho/.style={thin,draw=symbols,superdense,->},
	krhopr/.style={thick,draw=connection,superdense,->},
	triangle/.style = { regular polygon, regular polygon sides=3},
	not/.style={thin,circle,draw=connection,fill=connection,inner sep=0pt,minimum size=0.5mm},
	diff/.style = {very thin,draw=symbols,triangle,fill=red!50!black,inner sep=0pt,minimum size=1.6mm},
	diff1/.style = {very thin,dectriangle={1}{0},fill=red!50!black,draw=symbols,inner sep=0pt,minimum size=1.6mm},
	diff2/.style = {very thin,dectriangle={1}{1},fill=red!50!black,draw=symbols,inner sep=0pt,minimum size=1.6mm},
	diffmini/.style = {very thin,rectangle,fill=black,draw=black,inner sep=0pt,minimum size=0.75mm},
	kernelsmod/.style={very thick,draw=connection,segment length=12pt},
	rec/.style = {very thin,rectangle,fill=black,draw=black,inner sep=0pt,minimum size=2mm},
	cerc/.style={very thin,circle,draw=black,fill=symbols,inner sep=0pt,minimum size=2mm},
	stars/.style={very thin,star,star points=6,star point ratio=0.5, draw=black,fill=red,inner sep=0pt,minimum size=0.7mm},
	>=stealth,
}

\newcommand{\derivatives}{
	\begin{tikzpicture}[scale=0.1,baseline=-2]
		\coordinate (root) at (0,-0.4);
		\coordinate (t1) at (-.8,1.3);
		\coordinate (t2) at (.8,1.3);
		\draw[kernels2] (t1) -- (root);
		\draw[kernels2] (t2) -- (root);
		\node[not] (rootnode) at (root) {};
	\end{tikzpicture}
}

\newcommand{\diffh}{\begin{tikzpicture}[scale=0.2,baseline=-2]
		\coordinate (root) at (0,0);
		\node[diff] (rootnode) at (root) {};
	}
	
	\newcommand{\generic}{\begin{tikzpicture}
			\draw (0,0.6) node[xi] {};
		\end{tikzpicture}
	}
	
	\newcommand{\genericb}{\begin{tikzpicture}
			\draw (0,0.6) node[xic] {};
		\end{tikzpicture}
	}
	
	\newcommand{\IXi}{ \begin{tikzpicture}[scale=0.2,baseline=-2] \draw (0,-0.25) node[not] {} -- (0,1) node[xi] {};
		\end{tikzpicture}
	}
	
	\newcommand{\IiXi}{ \begin{tikzpicture}[scale=0.2,baseline=-2] \draw[kernels2] (0,-0.25) node[not] {} -- (0,1) node[xi] {};
		\end{tikzpicture}
	}

	\newcommand{\Xii}{ \begin{tikzpicture}[scale=0.2,baseline=-2] \draw (-1,-0.25) node[xi] {} -- (0,1) node[xi] {};
		\end{tikzpicture}
	}
	
	\newcommand{\IXitwo}{
		\begin{tikzpicture}[scale=0.2,baseline=-2]
			\draw[kernels2] (0,0) node[not] {} -- (-1,1) node[xi] {};
			\draw[kernels2] (0,0) -- (1,1) node[xi] {};
		\end{tikzpicture}
	}
	
	\newcommand{\Xiii}{\begin{tikzpicture}[scale=0.2,baseline=-2] \draw (0,0) node[xi] {} -- (-1,1) node[xi] {} -- (0,2) node[xi] {};
	\end{tikzpicture}	}
	
	\newcommand{\Xiiia}{\begin{tikzpicture}[scale=0.2,baseline=-2]
			\draw (-1,1)  node[xi] {} -- (0,0); 
			\draw (0,0) node[xi] {}  -- (1,1) node[xi] {};
		\end{tikzpicture}
	}
	
	\newcommand{\IiIXiiib}{
		\begin{tikzpicture}[scale=0.2,baseline=-2]
			\draw[kernels2] (0,0) node[not] {} -- (-1,1) ; \draw[kernels2] (0,0)   -- (1,1) node[xi] {} ;
			\draw (-1,1) node[xi] {} -- (0,2) node[xi] {};
		\end{tikzpicture}
	}
	
	\newcommand{\IiIXiiic}{	\begin{tikzpicture}[scale=0.2,baseline=-2]\draw[kernels2] (0,0) node[not] {} -- (-1,1) ; \draw[kernels2] (0,0)   -- (1,1) node[xi] {} ;
			\draw[kernels2] (-1,1) node[not] {} -- (0,2) node[xi] {};
			\draw[kernels2] (-1,1) node[not] {} -- (-2,2) node[xi] {};
	\end{tikzpicture}}
	
	\newcommand{\Xiiii}{\begin{tikzpicture}[scale=0.2,baseline=-2] \draw (0,0) node[xi] {} -- (-1,1) node[xi] {} -- (0,2) node[xi] {} -- (-1,3) node[xi] {};\end{tikzpicture}}
	
	\newcommand{\Xiiiib}{\begin{tikzpicture}[scale=0.2,baseline=-2] \draw(0,1.5) node[xi] {} -- (0,0); \draw (-1,1) node[xi] {} -- (0,0) node[xi] {} -- (1,1) node[xi] {};\end{tikzpicture}}
	
	\newcommand{\Xiiiic}{\begin{tikzpicture}[scale=0.2,baseline=-2] \draw (0,1) -- (0.8,2.2) node[xi] {};\draw (0,-0.25) node[xi] {} -- (0,1) node[xi] {} -- (-0.8,2.2) node[xi] {};\end{tikzpicture}}
	
	\newcommand{\Xiiiie}{\begin{tikzpicture}[scale=0.2,baseline=-2] \draw (0,2) node[xi] {} -- (-1,1) node[xi] {} -- (0,0) node[xi] {} -- (1,1) node[xi] {};\end{tikzpicture}}
	
	\newcommand{\IiIXiii}{\begin{tikzpicture}[scale=0.2,baseline=-2]\draw (0,0) node[xi] {} -- (-1,1) ; 
			\draw[kernels2] (-1,1) node[not] {} -- (0,2) node[xi] {};
			\draw[kernels2] (-1,1) node[not] {} -- (-2,2) node[xi] {};\end{tikzpicture}}
	
	\newcommand{\IiXiiic}{\begin{tikzpicture}[scale=0.2,baseline=-2]\draw[kernels2](0,1.5) node[xi] {} -- (0,0) node[not] {}; \draw (-1,1) node[xi] {} -- (0,0) ; \draw[kernels2] (0,0) -- (1,1) node[xi] {};\end{tikzpicture}}
	
	\newcommand{\IiXiiiia}{\begin{tikzpicture}[scale=0.2,baseline=-2]\draw[kernels2] (0,0) node[not] {} -- (-1,1) ; \draw[kernels2] (0,0) node[not] {} -- (1,1) node[xi] {} ;
			\draw (-1,1) node[xi] {} -- (0,2) node[xi] {} -- (-1,3) node[xi] {};\end{tikzpicture}}
	
	\newcommand{\IiXiiiiacc}{\begin{tikzpicture}[scale=0.2,baseline=-2]\draw[kernels2] (0,0) node[not] {} -- (-1,1) ; \draw[kernels2] (0,0) node[not] {} -- (1,1) node[xic] {} ;
			\draw (-1,1) node[xic] {} -- (0,2) node[xi] {} -- (-1,3) node[xi] {};\end{tikzpicture}}
	
	\newcommand{\IiXiiiib}{\begin{tikzpicture}[scale=0.2,baseline=-2]\draw (0,0) node[xi] {} -- (-1,1) node[xi] {} -- (0,2) ; \draw[kernels2] (0,2) node[not] {} -- (-1,3) node[xi] {};\draw[kernels2] (0,2)  -- (1,3) node[xi] {};
	\end{tikzpicture}	}
	
	\newcommand{\IiXiiiic}{\begin{tikzpicture}[scale=0.2,baseline=-2]\draw (0,0) node[xi] {} -- (-1,1) node[not] {}; \draw[kernels2] (-1,1) -- (0,2) ; 
			\draw[kernels2] (-1,1) -- (-2,2) node[xi] {} ;
			\draw (0,2) node[xi] {} -- (-1,3) node[xi] {};\end{tikzpicture}}

	\newcommand{\Xiiiieabis}{\begin{tikzpicture}[scale=0.2,baseline=-2] \draw (-1,2.5) node[xi] {} -- (-1,1) ; \draw[kernels2] (-1,1) node[xi] {} -- (0,0); 
			\draw (0,0)  -- (1,1) node[xi] {};
			\draw[kernels2] (0,0) node[not] {} -- (0,1.5) node[xi] {}; \end{tikzpicture} }
	
	\newcommand{\Xiiiiea}{\begin{tikzpicture}[scale=0.2,baseline=-2] \draw (-1,2.5) node[xi] {} -- (-1,1) node[xi] {} -- (0,0); 
			\draw[kernels2] (0,0)  -- (1,1) node[xi] {};
			\draw[kernels2] (0,0) node[not] {} -- (0,1.5) node[xi] {}; \end{tikzpicture} }
	
	\newcommand{\Xiiiica}{\begin{tikzpicture}[scale=0.2,baseline=-2] \draw (0,1) -- (-1,2.2) node[xi] {};\draw (0,-0.25) node[xi] {} -- (0,1) ; \draw[kernels2] (0,1) node[not] {} -- (1,2.2) node[xi] {};
			\draw[kernels2] (0,1) {} -- (0,2.7) node[xi] {}; \end{tikzpicture}
	}

	\newcommand{\Xiiiiba}{\begin{tikzpicture}[scale=0.2,baseline=-2] \draw(-0.5,1.5) node[xi] {} -- (0,0); \draw (-1.5,1) node[xi] {} -- (0,0) node[not] {}; \draw[kernels2] (0,0) -- (1.5,1) node[xi] {};
			\draw[kernels2] (0,0) -- (0.5,1.5) node[xi] {} ; \end{tikzpicture}}
	
	\newcommand{\Xiiiicab}{\begin{tikzpicture}[scale=0.2,baseline=-2] \draw (-1,1) -- (-2,2) node[xi] {};\draw[kernels2] (0,0)  -- (-1,1); \draw[kernels2] (0,0) node[not] {} -- (1,1) node[xi] {} ; 
			\draw[kernels2] (-1,1)  {} -- (0,2) node[xi] {};
			\draw[kernels2] (-1,1) node[not] {} -- (-1,2.5) node[xi] {};
		\end{tikzpicture}
	}
	
	\newcommand{\Xiiiieabbis}{\begin{tikzpicture}[scale=0.2,baseline=-2] \draw[kernels2] (-1,2.5) node[xi] {} -- (-1,1) ; \draw[kernels2] (-2,2)  node[xi] {} -- (-1,1) ; \draw[kernels2] (-1,1)  node[not] {} -- (0,0); 
			\draw (0,0)  -- (1,1) node[xi] {};
			\draw[kernels2] (0,0) node[not] {} -- (0,1.5) node[xi] {};
		\end{tikzpicture} 
	}

	\newcommand{\Xiiiieab}{\begin{tikzpicture}[scale=0.2,baseline=-2] \draw[kernels2] (-1,2.5) node[xi] {} -- (-1,1) ; \draw[kernels2] (-2,2)  node[xi] {} -- (-1,1) ; \draw (-1,1)  node[not] {} -- (0,0); 
			\draw[kernels2] (0,0)  -- (1,1) node[xi] {};
			\draw[kernels2] (0,0) node[not] {} -- (0,1.5) node[xi] {}; 
		\end{tikzpicture}
	}

	\newcommand{\iiIiXiiii}{\begin{tikzpicture}[scale=0.2,baseline=-2] \draw[kernels2] (0,0) node[not] {} -- (-1,1) node[not] {};
			\draw[kernels2] (0,0) -- (1,1) node[not] {};
			\draw[kernels2] (-1,1) -- (-1.5,2.5) node[xi] {};
			\draw[kernels2] (-1,1) -- (-0.5,2.5) node[xi] {};
			\draw[kernels2] (1,1) -- (0.5,2.5) node[xi] {};
			\draw[kernels2] (1,1) -- (1.5,2.5) node[xi] {};
		\end{tikzpicture}
	}
	
	\newcommand{\iiIiXiiiiaa}{\begin{tikzpicture}[scale=0.2,baseline=-2] \draw[kernels2] (0,0) node[not] {} -- (-1,1) node[not] {};
			\draw[kernels2] (0,0) -- (1,1) node[not] {};
			\draw[kernels2] (-1,1) -- (-1.5,2.5) node[xic] {};
			\draw[kernels2] (-1,1) -- (-0.5,2.5) node[xic] {};
			\draw[kernels2] (1,1) -- (0.5,2.5) node[xi] {};
			\draw[kernels2] (1,1) -- (1.5,2.5) node[xi] {};
		\end{tikzpicture}
	}
	
	\newcommand{\iiIiXiiiibb}{\begin{tikzpicture}[scale=0.2,baseline=-2] \draw[kernels2] (0,0) node[not] {} -- (-1,1) node[not] {};
			\draw[kernels2] (0,0) -- (1,1) node[not] {};
			\draw[kernels2] (-1,1) -- (-1.5,2.5) node[xic] {};
			\draw[kernels2] (-1,1) -- (-0.5,2.5) node[xi] {};
			\draw[kernels2] (1,1) -- (0.5,2.5) node[xic] {};
			\draw[kernels2] (1,1) -- (1.5,2.5) node[xi] {};
		\end{tikzpicture}
	}

	\newcommand{\Xitwo}
	{\begin{tikzpicture}[scale=0.2,baseline=-2] \draw[kernels2] (0,0) node[not] {} -- (-1,1) node[not] {}
			-- (-2,2) node[not]{} -- (-3,3) node[xi]  {};
			\draw[kernels2] (0,0) -- (1,1) node[xi] {};
			\draw[kernels2] (-1,1) -- (0,2) node[xi] {};
			\draw[kernels2] (-2,2) -- (-1,3) node[xi] {};
	\end{tikzpicture} }
	
	\newcommand{\Xitwoo}
	{\begin{tikzpicture}[scale=0.2,baseline=-2] \draw[kernels2] (0,0) node[not] {} -- (-1,1) node[not] {}
			-- (-2,2) node[not]{} -- (-3,3) node[xic]  {};
			\draw[kernels2] (0,0) -- (1,1) node[xi] {};
			\draw[kernels2] (-1,1) -- (0,2) node[xic] {};
			\draw[kernels2] (-2,2) -- (-1,3) node[xi] {};
	\end{tikzpicture} }
	
	\newcommand{\Xitwooo}
	{\begin{tikzpicture}[scale=0.2,baseline=-2] \draw[kernels2] (0,0) node[not] {} -- (-1,1) node[not] {}
			-- (-2,2) node[not]{} -- (-3,3) node[xic]  {};
			\draw[kernels2] (0,0) -- (1,1) node[xi] {};
			\draw[kernels2] (-1,1) -- (0,2) node[xi] {};
			\draw[kernels2] (-2,2) -- (-1,3) node[xic] {};
	\end{tikzpicture} }
	
	\newcommand{\iiIiXiiiic}{\begin{tikzpicture}[scale=0.2,baseline=-2] \draw[kernels2] (0,0) node[not] {} -- (-1,1);
			\draw[kernels2] (0,0) -- (1,1) node[not] {};
			\draw (-1,1)  node[xi] {} -- (-1,2.5) node[xi] {};
			\draw[kernels2] (1,1) -- (0.4,2.5) node[xi] {};
			\draw[kernels2] (1,1) -- (1.6,2.5) node[xi] {};
		\end{tikzpicture}
	}

	\newcommand{\IiXiiiiac}{ \begin{tikzpicture}[scale=0.2,baseline=-2]\draw[kernels2] (0,0) node[not] {} -- (-1,1) ; \draw[kernels2] (0,0) node[not] {} -- (1,1) node[xi] {}; 
			\draw[kernels2] (-1,1) node[not] {} -- (0,2) ; 
			\draw[kernels2] (-1,1) -- (-2,2) node[xi] {} ;
			\draw (0,2) node[xi] {} -- (-1,3) node[xi] {};
		\end{tikzpicture}
	}

	\newcommand{\IiXiiiibc}{\begin{tikzpicture}[scale=0.2,baseline=-2] \draw (0,0) node[xi] {} -- (-1,1) node[not] {}; \draw[kernels2] (-1,1) -- (0,2) ; 
			\draw[kernels2] (-1,1) -- (-2,2) node[xi] {} ; \draw[kernels2] (0,2) node[not] {} -- (-1,3) node[xi] {};\draw[kernels2] (0,2)  -- (1,3) node[xi] {};
		\end{tikzpicture}
	}

	\newcommand{\IiXiiiiab}{\begin{tikzpicture}[scale=0.2,baseline=-2] \draw[kernels2] (0,0) node[not] {} -- (-1,1) ; \draw[kernels2] (0,0) node[not] {} -- (1,1) node[xi] {};\draw (-1,1) node[xi] {} -- (0,2) ; \draw[kernels2] (0,2) node[not] {} -- (-1,3) node[xi] {};\draw[kernels2] (0,2)  -- (1,3) node[xi] {}; 
	\end{tikzpicture}}

	\newcommand{\Xiiiicb}{\begin{tikzpicture}[scale=0.2,baseline=-2] \draw (-1,1) -- (-2,2) node[xi] {};\draw[kernels2] (0,0)  -- (-1,1) node[xi] {} ; \draw[kernels2] (0,0) node[not] {} -- (1,1) node[xi] {} ; 
			\draw (-1,1) node[xi] {} -- (0,2) node[xi] {}; \end{tikzpicture}}
	
	\newcommand{\Xiiiieb}{
		\begin{tikzpicture}[scale=0.2,baseline=-2] 	\draw[kernels2] (0,2) node[xi] {} -- (-1,1) ; \draw[kernels2] (-2,2)  node[xi] {} -- (-1,1) ; \draw (-1,1)  node[not] {} -- (0,0); 
			\draw (0,0) node[xi] {}  -- (1,1) node[xi] {};
		\end{tikzpicture}
	}
	
	\newcommand{\iiIiXiiiib}{\begin{tikzpicture}[scale=0.2,baseline=-2] \draw[kernels2] (0,0) node[not] {} -- (-1,1) ;
			\draw[kernels2] (0,0) -- (1,1);
			\draw (-1,1) node[xi] {} -- (-1,2.5) node[xi] {};
			\draw (1,1)  node[xi] {} -- (1,2.5) node[xi] {};
		\end{tikzpicture}
	}

\begin{document}
	\def\st{\mathsf{fgt}}
	\def\mail#1{\burlalt{#1}{mailto:#1}}
	\title{Chain rule symmetry for singular SPDEs from multi-indices to decorated trees}
	\author{Yvain Bruned}
	\institute{ 
		Université de Lorraine, CNRS, IECL, F-54000 Nancy, France
		\\
		Email:\ \begin{minipage}[t]{\linewidth}
			\mail{yvain.bruned@univ-lorraine.fr}
	\end{minipage}}
	\def\dsqcup{\sqcup\mathchoice{\mkern-7mu}{\mkern-7mu}{\mkern-3.2mu}{\mkern-3.8mu}\sqcup}

	\maketitle

	\begin{abstract}
		
		\ \ \ \ In these lecture notes, we review the recent results around the chain rule for the generalised KPZ equation and the geometric KPZ equation. We introduce the main ideas of Regularity Structures via a B-series type expansion and the two main combinatorial sets for encoding this expansion: decorated trees and multi-indices. Then, we explain how the chain rule was first understood via integration by parts and diagrammatic computations. 
	For the generalised KPZ equation,  a scalar-valued equation, we use multi-indices to derive a characterisation in terms of iterations of covariant derivatives. The same characterisation for decorated trees is more involved and requires the use of operadic and homological tools. We summarise the main arguments of this proof. 
	{\scriptsize  
		} 
	\end{abstract}

	\setcounter{tocdepth}{1}
	\tableofcontents

	\section{ Introduction }
		

		

		In these lecture notes, we focus on the following singular stochastic partial differential equation (SPDE):
		\begin{equation} \label{equ_gKPZ}
			\partial_t u = \partial_x^2 u + f(u) (\partial_x u)^2 + g(u) \xi, \quad (t,x) \in \mathbb{R}_+ \times \mathbb{T}
		\end{equation}
		where $ \mathbb{T} $ is the one-dimensional torus, $ \xi $ is a  space-time noise, a space-time random distribution. We will consider noises that belong to the space-time Hölder space $\mathcal{C}^{-2+ \kappa}$ with $\kappa > 0$. Here, one uses the parabolic scaling $\mathfrak{s} = (2,1)$ and work with the anisotropic norm: 
		\begin{equation*}
			\Vert (t,x) \Vert_{\mathfrak{s}} = |x| + \sqrt{|t|}.
		\end{equation*}
		The exponent $-2$ is critical and one cannot work with noises with regularity below this threshold. A main example of a subcritical noise is the space-time white noise which is a Gaussian process whose covariance is formally described by
		\begin{equation*}
			\mathbb{E}(\xi(t,x) \xi(s,y)) = \delta(t-s) \delta(x-y)
		\end{equation*}
		where $ \delta(x) $ is equal to one when $x=0$ and zero otherwise. This means that one has independence between space-time points. The trajectories of this noise belong to $\mathcal{C}^{-\frac{3}{2}- \kappa}$ for $\kappa > 0$.
		Equation \eqref{equ_gKPZ} is called the generalised KPZ equation as if one takes $f=g=1$, one recovers the KPZ equation. This equation models one dimensional random growing interfaces.  This generalisation is built such that one expects to have the chain rule in the sense that if $u$ solves \eqref{equ_gKPZ} then for some diffeomorphism $\varphi$, $\varphi(u)$ solves the same type of equation but with different coefficients $f$ and $g$.
		
		One can write a vector-valued extension of \eqref{equ_gKPZ} called geometric KPZ equation given below
		\begin{equation}
			\label{geo_KPZ}
			\partial_t u^{\alpha} = \partial_x^2 u^{\alpha} + \Gamma_{\beta \gamma}^{\alpha}(u) \partial_x u^{\beta} \partial_x u^{\gamma} + \sigma_{i}^{\alpha}(u) \xi_i, \quad (t,x) \in \mathbb{R}_+ \times \mathbb{T}
		\end{equation}
		where $ \alpha, \beta, \gamma \in \{  1,...,d \} $, $ i\in \{  1,...,n \} $. We have used Einstein notation for the indices omitting the sum over the indices $ \beta, \gamma $ and $i$. 
		The $\xi_i$ are independent space-time noises.
		We suppose that $\Gamma_{\beta \gamma}^{\alpha}(u)$ and  $ \sigma_{i}^{\alpha}(u)$ are smooth functions in $u$ such that $ \Gamma^{\alpha}_{\beta \gamma} = \Gamma^{\alpha}_{ \gamma \beta} $. This system has a geometric interpretation as it corresponds in local coordinates to a natural stochastic evolution taking values in loops on a compact Riemannian manifold. It was first considered for coloured noises in \cite{Fun92} and singular noises in \cite{proc,BGHZ,BD24}. The $ \Gamma$ are the Christoffel symbols and the $\sigma_i$ are the smooth vector fields connected to the metric $ g$. One has $ g^{\alpha \beta} = \sigma_i^{\alpha} \sigma_i^{\beta} $. One main property of the equation \eqref{geo_KPZ} is the invariance under change of coordinates that we call the chain rule symmetry.
		
		Both equations \eqref{equ_gKPZ} and \eqref{geo_KPZ} contain singular distributional products and one has to modify such equations by subtracting counter-terms. One wants to guarantee that the chain rule symmetry will not break with such a renormalisation procedure. The main aim of this course is to study this symmetry. The course is divided into three parts:
		\begin{enumerate}
			\item Regularity Structures, decorated trees and multi-indices. In this section, we recall the main ideas behind the different methods for solving singular SPDEs. Then, we focus on Regularity Structures by presenting the B-series perspective and the main combinatorial sets for encoding this expansion are decorated trees and multi-indices. 
			\item Cancellations via integration by parts. We introduce graphical rules fully justified by an important identity on the heat kernel $K$. Then, using integration by parts one observes compensations between renormalisation constants. It was the first attempt at obtaining the chain rule.
			\item Characterisation of the chain rule symmetry for multi-indices and decorated trees. In this section, we review the proofs of the characterisation of the chain rule given in \cite{BD24,CBB}. This characterisation is based on iterated covariant derivatives. The multi-indices case in \cite{BD24} relies on elementary arguments whereas the high-dimensional case uses operadic and homological tools such as the operadic twisting.
		\end{enumerate}
		\subsection*{Acknowledgments}
		
		{\small
			These lecture notes are based on a series of lectures given in the YMCN Spring School: Recent Advances in SPDEs, 24-28 March 2025 at the
			University of Münster in Germany. Y.B. thanks the organisers of this school, Fabian Höfer,
			Sarah-Jean Meyer and
			Xiaohao Ji for the opportunity to give this lecture in front of such a large audience. 
			Y.B. acknowledges funding support from the European Research Council (ERC) through the ERC Starting Grant Low Regularity Dynamics via Decorated Trees (LoRDeT), grant agreement No.\ 101075208 is acknowledged. Views and opinions expressed are however those of the author only and do not necessarily reflect those of the European Union or the European Research Council. Neither the European Union nor the granting authority can be held responsible for them. 
		}

		\section{Regularity Structures, decorated trees and multi-indices.}
		
		Let us summarise the main ideas for solving the singular equations \eqref{equ_gKPZ} and \eqref{geo_KPZ}. One starts with a \textbf{regularisation step} by considering a mollifier $  \varrho$, a smooth compactly supported function symmetric in space 
		\begin{equation} \label{sym_rho}
			\varrho(t,-x) = \varrho(t,x) 
		\end{equation}
		and rescaled according to the parabolic scaling for every $ \varepsilon > 0 $
		$$  \varrho_{\varepsilon}(t,x) = \varepsilon^{-3} \varrho(\varepsilon^{-1} t, \varepsilon^{-2}x).  $$
		Then, one replaces the rough noise $ \xi $ by a smooth one $ \xi_{\varepsilon} = \varrho_{\varepsilon} * \xi $ where $ * $ is the space-time convolution and is such that $ \xi_{\varepsilon} $ converges to $ \xi $ when $\varepsilon$ tends to zero.
		The symmetric condition \eqref{sym_rho} is crucial in the sequel for studying the chain rule symmetry as it allows one to disregard terms that break the symmetry.
		One considers $u_{\varepsilon}$ the solution of the mollified equation
		\begin{equation}
			\partial_t u_{\varepsilon} = \partial_x^2 u_{\varepsilon} + f(u_{\varepsilon}) (\partial_x u_{\varepsilon})^2 + g(u_{\varepsilon}) \xi_{\varepsilon} + F(u_{\varepsilon}, \partial_x u_{\varepsilon})
		\end{equation}
		where $ F(u_{\varepsilon}, \partial_x u_{\varepsilon}) $ is a counter-term to renormalise the ill-defined distributional products in the equation. Here, we assume locality of the counter-terms in the solution and its derivatives.
		Let us briefly explain  the \textbf{distributional products} of \eqref{equ_gKPZ} which are $ f(u) (\partial_x u)^2 $ and $ g(u) \xi $.
		By  a scaling argument one can show that
		\begin{equation*}
			\text{Regularity of } u    = \text{Regularity of } v
		\end{equation*}
		where $v$ solves the linear stochastic heat equation. One has
		\begin{equation*}
			\partial_t v = \partial_x^2 v + \xi, \quad v= K * \xi.
		\end{equation*}
		Here, one has an explicit well-defined expression for $v$ given by the space-time convolution of the heat kernel with the noise $\xi$. One can make sense of this non-local product called stochastic convolution and show that $ v \in \mathcal{C}^{\frac{1}{2} - \kappa} $ when $\xi$ is the space-time white noise. In general, one gets a $+2$ gain in Hölder regularity provided by the Schauder estimate for the convolution with the heat kernel. Then, one expects $u$  not to be differentiable and therefore $ \partial_x u $ is a distribution and one cannot define the square of this object. The product $g(u) \xi$ is also ill-defined as the regularity of the noise is $ -\frac{3}{2} - \kappa $ and therefore the sum of the Hölder exponents of both $u$ and $\xi$ is negative. 
		
		With the regularisation, one gets a solution map that depends on the smooth noise $ \xi_{\varepsilon}  $. Unfortunately, this map is not continuous and one has to add extra stochastic data to obtain continuity. This extra data is an ansatz for the solution.  One has
		\begin{equation*}
			u_{\varepsilon} = \Phi( \xi_{\varepsilon}, \text{Ansatz} ), \quad \Phi \, \, \text{  continuous}.
		\end{equation*}
		Several  ansatz have been proposed for constructing this continuous solution map: 
		\begin{itemize}
			\item Regularity Structures (see \cite{reg}), inspired by rough paths (see \cite{Lyons, Gub04, Gub10}), proposes a local expansion of the solution via new monomials. Then, from this local description, one gets a global distribution via a reconstruction theorem.  A large class of subcritical equations is covered in a series of works \cite{reg,BHZ,CH,BCCH} where the renormalisation used is inspired by the BPHZ renormalisation for Feynman diagrams given in \cite{BP57,KH69,WZ69}. One can find nice overviews on Regularity Structures in \cite{FrizHai,BH20}.
			\item Paracontrolled calculus uses the Littlewood-Paley decomposition  to single out the singular part of distributional products. The ansatz is written via iterated paraproducts. It was initiated in \cite{GIP} with higher-order computations given in \cite{BB19}. A general ansatz is given in \cite{BM26} with the convergence of the stochastic data reaching the degree of generality given by Regularity Structures.
			\item The flow approach by introducing a scale $ \mu$ on the kernel $K$ replaced by $  K_{\mu}$. Then, the idea is to perform a flow on this scale and solve the equation scale by scale. One has two options: discrete scales (Renormalisation group see \cite{K16}) or continuous scales (Polchinski's flow see \cite{P84,Duc21}). The most successful approach is the continuous one which covers a family of subcritical equations in \cite{Duc21} and the generalised KPZ equation in \cite{CF24a}. A general ansatz is provided in \cite{BM25} where one also proves that the renormalisation in the flow approach corresponds to the BPHZ renormalisation.
		\end{itemize}
		All these ansatz can be described in the same way using a B-series formalism.
		This formalism was introduced by Butcher for encoding Runge-Kutta schemes via tree series.
		One has
		\begin{equation*}
			u_{\varepsilon} = \text{Finite sum of stochastic iterated integrals} + \text{Remainder}.
		\end{equation*}
		One has to perform analytical constructions on the stochastic iterated integrals depending on the theory: recentering around a base point for Regularity Structures, Fourier decomposition for Paracontrolled calculus and insertion of the scale $\mu$ for the flow approach. The B-series formalism corresponds to what these various approaches have in common that is the combinatorial set and the Hopf algebraic structures for encoding the stochastic iterated integrals. In this lecture, we will briefly review  the Regularity Structures ansatz.

		The ansatz is given by a Taylor-type expansion around the base point $z$
		\begin{equation}
			u_{\varepsilon}(\bar z) = \sum_{\tau \in \hat{\mathcal{T}}} \frac{\hat{F}[\tau]}{S(\tau)}(u_{\varepsilon}(z), \partial_x u_{\varepsilon}(z))
			(\Pi_{z} \tau )(\bar{z}) + R_{\hat{\mathcal{T}},z}(\bar{z})
		\end{equation}
		where 
		\begin{itemize}
			\item $\hat{\mathcal{T}}$ is a combinatorial set. We will use multi-indices for the generalised KPZ equation and decorated trees for the geometric equation.
			\item $ S(\tau) $ is a symmetry factor and $ F[\tau](u_{\varepsilon}(z), \partial_x u_{\varepsilon}(z)) $ are elementary differentials, which are the coefficients of the Taylor expansion.
			\item The $ (\Pi_z \tau)(\bar{z}) $ are the new monomials that satisfy
			\begin{equation*}
				(\Pi_z \tau)(\bar{z}) \lesssim \Vert \bar{z} - z  \Vert_{\mathfrak{s}}^{\mathrm{deg}(\tau)}
			\end{equation*}
			where $ \mathrm{deg} : \hat{\mathcal{T}} \rightarrow \mathbb{R} $ associates to each tree a scalar not necessarily an integer. We will give a precise definition below.
			\item $R_{\hat{ \mathcal{T}},z}(\bar{z})$ is a Taylor remainder such that there exists $ \gamma $ with $ \gamma \geq \mathrm{deg}(\tau) $ for all $ \tau \in \hat{\mathcal{T}} $ and 
			\begin{equation*}
				R_{\hat{\mathcal{T}},z}(\bar{z}) \lesssim 
				\Vert \bar{z} - z \Vert_{\mathfrak{s}}^{\gamma}.
			\end{equation*}
		\end{itemize}	
		One can write a similar ansatz for the non-linearity of the equation:
		\begin{equation} \label{B_series_nonlinearity}
			(f(u_{\varepsilon}) (\partial_x u_{\varepsilon})^2 + g(u_{\varepsilon}) \xi_{\varepsilon} )(\bar{z}) = \sum_{\tau \in \mathcal{T}} \frac{F[\tau]}{S(\tau)}(u_{\varepsilon}(z), \partial_x u_{\varepsilon}(z))
			(\Pi_{z} \tau )(\bar{z}) + R_{\mathcal{T},z}(\bar{z})
		\end{equation}
		The main difference is the combinatorial set $\mathcal{T}$ that replaces $ \hat{\mathcal{T}} $. The idea of two different B-series within Regularity Structures is highlighted in \cite{B23}.
		
		The idea behind such an ansatz is to start a perturbative expansion.	One uses decorated trees to encode the first stochastic iterated integrals
		\begin{equation*}
			\xi_{\varepsilon} \equiv \generic, \quad K * \xi_{\varepsilon} \equiv \IXi, \quad \partial_x K * \xi_{\varepsilon} \equiv \IiXi 
		\end{equation*}
		where the node $ \generic $ encodes the space-time noise, a thin edge denotes a convolution with $K$ and a thick edged denotes a convolution with $\partial_x K$. Then, the merging root product denoted by $\cdot_r$, which merges the roots of two decorated trees, corresponds to the pointwise product of the integrals. We illustrate this below with the following examples:
		\begin{equation*}
			\xi_{\varepsilon} \, K * \xi_{\varepsilon}  \equiv \generic \cdot_r   \IXi =  \Xii, \quad  (\partial_x K * \xi_{\varepsilon})^{2} \equiv   \IiXi \cdot_r \IiXi =  \IXitwo.
		\end{equation*}
		As one cannot multiply noises, the merging root product is zero when such a configuration occurs. Let us explain how to recenter. We first define the degree map $ \mathrm{deg}$ that takes into account the different regularities of the objects:
		\begin{equation*}
			\mathrm{deg}(\generic) = \alpha, \quad  \mathrm{deg}( \begin{tikzpicture}[scale=0.2,baseline=-2] \draw (0,-0.25) node[] {} -- (0,1) node[] {};
			\end{tikzpicture}) = 2, \quad  \mathrm{deg}( \begin{tikzpicture}[scale=0.2,baseline=-2] \draw[kernels2] (0,-0.25) node[] {} -- (0,1) node[] {};
			\end{tikzpicture}) =1, 
		\end{equation*}
		where we assume that 
		$ \xi \in \mathcal{C}^\alpha $ with $ \alpha > -2 $, and the degree of the decorated edges corresponds to the Schauder estimates ($+2$ for $K$ and $+1$ for $\partial_x K$). Then, the degree of a decorated tree is the sum of the degrees of its edges and nodes. As an example, one has
		\begin{equation*}
			\begin{aligned}
				\mathrm{deg}( \Xii ) & = 2 \mathrm{deg}(\generic)   +  \mathrm{deg}( \begin{tikzpicture}[scale=0.2,baseline=-2] \draw (0,-0.25) node[] {} -- (0,1) node[] {};
				\end{tikzpicture}) = 2 \alpha + 2, 
				\\  
				\mathrm{deg}(\IXitwo) &= 2 \mathrm{deg}(\generic) + 2 \mathrm{deg}( \begin{tikzpicture}[scale=0.2,baseline=-2] \draw[kernels2] (0,-0.25) node[] {} -- (0,1) node[] {};
				\end{tikzpicture})  = 2 \alpha + 2.
			\end{aligned} 
		\end{equation*}
		We introduce monomials of the form $ X^{k} = X_1^{k_1} X_2^{k_2} $ with $ k = (k_1,k_2) \in  \mathbb{N}^2 $ and we set
		\begin{equation*}
			|k|_{\mathfrak{s}} = 2 k_1 + k_2.
		\end{equation*}
		Then, the set of decorated trees $ \mathcal{T} $ is described by the following inductive set
		\begin{equation*}
			\begin{aligned}
				\mathcal{T} & = \{ 
				\begin{tikzpicture}[scale=0.4,baseline=-2]
					\draw (0,0) node[not,label= {[label distance=-0.2em]below: \tiny  $  k   $} ] {};
				\end{tikzpicture}, 
				\begin{tikzpicture}[scale=0.4,baseline=-2]
					\draw (0,0) node[xi,label= {[label distance=-0.2em]below: \tiny  $  k   $} ] {};
				\end{tikzpicture},  \begin{tikzpicture}[scale=0.4,baseline=-2]
					\coordinate (root) at (0,-0.4);
					\coordinate (t2) at (-0.8,0.5);
					\coordinate (t3) at (0.8,0.5);
					\draw[] (t2) -- (root);
					\draw[] (t3) -- (root);
					\draw (0,-0.4) node[xi,label= {[label distance=-0.2em]below: \tiny  $  k   $} ] {};
					\draw (-0.9,0.7) node[] {\tiny$\tau_1$};
					\draw (0.9,0.7) node[] {\tiny$\tau_n$};
					\draw (0,0.7) node[] {\tiny$\cdots$};
				\end{tikzpicture}, \begin{tikzpicture}[scale=0.4,baseline=-2]
					\coordinate (root) at (0,-0.4);
					\coordinate (t2) at (-0.8,0.5);
					\coordinate (t3) at (0.8,0.5);
					\draw[] (t2) -- (root);
					\draw[] (t3) -- (root);
					\draw (0,-0.4) node[not,label= {[label distance=-0.2em]below: \tiny  $  k   $} ] {};
					\draw (-0.9,0.7) node[] {\tiny$\tau_1$};
					\draw (0.9,0.7) node[] {\tiny$\tau_n$};
					\draw (0,0.7) node[] {\tiny$\cdots$};
				\end{tikzpicture},   \begin{tikzpicture}[scale=0.4,baseline=-2]
					\coordinate (root) at (0,-0.4);
					\coordinate (t1) at (-0.3,0.5);
					\coordinate (t2) at (-1.1,0.5);
					\coordinate (t3) at (1.1,0.5);
					\coordinate (t4) at (1.1,0.5);
					\coordinate (tau1) at (1,0.6);
					\draw[] (t1) -- (root);
					\draw[kernels2] (t2) -- (root);
					\draw[] (t3) -- (root);
					\node[not,label= {[label distance=-0.2em]below: \tiny  $  k   $}] (rootnode) at (root) {};
					\draw (-1.2,0.7) node[] {\tiny$\tau_1$};
					\node[not] (rootnode) at (root) {};
					\draw (1.3,0.7) node[] {\tiny$\tau_n$};
					\draw (-0.4,0.7) node[] {\tiny$\tau_{\tiny{2}}$};
					\draw (0.5,0.7) node[] {\tiny$\cdots$};
				\end{tikzpicture},  \begin{tikzpicture}[scale=0.4,baseline=-2]
					\coordinate (root) at (0,-0.4);
					\coordinate (t1) at (-0.3,0.5);
					\coordinate (t2) at (-1.1,0.5);
					\coordinate (t3) at (1.1,0.5);
					\coordinate (t4) at (1.1,0.5);
					\coordinate (tau1) at (1,0.6);
					\draw[kernels2] (t1) -- (root);
					\draw[kernels2] (t2) -- (root);
					\draw[] (t3) -- (root);
					\node[not,label= {[label distance=-0.2em]below: \tiny  $  k   $}] (rootnode) at (root) {};
					\draw (-1.2,0.7) node[] {\tiny$\tau_1$};
					\node[not] (rootnode) at (root) {};
					\draw (1.3,0.7) node[] {\tiny$\tau_m$};
					\draw (-0.4,0.7) node[] {\tiny$\tau_{\tiny{2}}$};
					\draw (0.5,0.7) node[] {\tiny$\cdots$};
				\end{tikzpicture}, \\ & m,n \in \mathbb{N}^*, m \geq 2,  \tau_i \in \mathcal{T}, k \in \mathbb{N}^{2} \}
			\end{aligned}
		\end{equation*}
		In practice, one uses the identification $ X^k = \begin{tikzpicture}[scale=0.4,baseline=-2]
			\draw (0,0) node[not,label= {[label distance=-0.2em]below: \tiny  $  k   $} ] {};
		\end{tikzpicture} $.
		We have added new decorations on the nodes in $\mathbb{N}^2$ to take into account the monomials that will be part of the local expansion and play a role in the computation of the degree. One has for $ z=(t,x) $
		\begin{equation*}
			\begin{tikzpicture}[scale=0.2,baseline=-2] \draw (-1,-0.25) node[xi,label= {[label distance=-0.2em]below: \tiny  $  k   $}] {} -- (0,1) node[xi] {};
			\end{tikzpicture} \equiv t^{k_1} x^{k_2}(\xi_{\varepsilon} \, K * \xi_{\varepsilon})(z), \quad \mathrm{deg}(	\begin{tikzpicture}[scale=0.2,baseline=-2] \draw (-1,-0.25) node[xi,label= {[label distance=-0.2em]below: \tiny  $  k   $}] {} -- (0,1) node[xi] {};
			\end{tikzpicture}) = \mathrm{deg}( \Xii )	+ |k|_{\mathfrak{s}}, \quad \begin{tikzpicture}[scale=0.4,baseline=-2]
				\draw (0,0) node[not,label= {[label distance=-0.2em]below: \tiny  $  m   $} ] {};
			\end{tikzpicture}	\cdot_r	\begin{tikzpicture}[scale=0.2,baseline=-2] \draw (-1,-0.25) node[xi,label= {[label distance=-0.2em]below: \tiny  $  k   $}] {} -- (0,1) node[xi] {};
			\end{tikzpicture}  = \begin{tikzpicture}[scale=0.2,baseline=-2] \draw (-1,-0.25) node[xi,label= {[label distance=-0.2em]below: \tiny  $  m+k   $}] {} -- (0,1) node[xi] {};
			\end{tikzpicture}.
		\end{equation*}
		One adds the degrees of the monomials to the previous degree. The merging root product is extended by adding the root decorations. One can see $\mathcal{T}$ as defined by some rules which correspond to the products appearing on the right-hand side  of the singular SPDE. One can draw the following parallel
		\begin{equation*}
			f(u_{\varepsilon}) (\partial_x u_{\varepsilon})^2 \equiv \begin{tikzpicture}[scale=0.4,baseline=-2]
				\coordinate (root) at (0,-0.4);
				\coordinate (t1) at (-0.3,0.5);
				\coordinate (t2) at (-1.1,0.5);
				\coordinate (t3) at (1.1,0.5);
				\coordinate (t4) at (1.1,0.5);
				\coordinate (tau1) at (1,0.6);
				\draw[kernels2] (t1) -- (root);
				\draw[kernels2] (t2) -- (root);
				\draw[] (t3) -- (root);
				\node[not,label= {[label distance=-0.2em]below: \tiny  $  k   $}] (rootnode) at (root) {};
				\draw (-1.2,0.7) node[] {\tiny$\tau_1$};
				\node[not] (rootnode) at (root) {};
				\draw (1.3,0.7) node[] {\tiny$\tau_m$};
				\draw (-0.4,0.7) node[] {\tiny$\tau_{\tiny{2}}$};
				\draw (0.5,0.7) node[] {\tiny$\cdots$};
			\end{tikzpicture}, \quad g(u_{\varepsilon}) \xi_{\varepsilon} \equiv \begin{tikzpicture}[scale=0.4,baseline=-2]
				\coordinate (root) at (0,-0.4);
				\coordinate (t2) at (-0.8,0.5);
				\coordinate (t3) at (0.8,0.5);
				\draw[] (t2) -- (root);
				\draw[] (t3) -- (root);
				\draw (0,-0.4) node[xi,label= {[label distance=-0.2em]below: \tiny  $  k   $} ] {};
				\draw (-0.9,0.7) node[] {\tiny$\tau_1$};
				\draw (0.9,0.7) node[] {\tiny$\tau_n$};
				\draw (0,0.7) node[] {\tiny$\cdots$};
			\end{tikzpicture}.
		\end{equation*}
		Then, we define another set of decorated trees used for describing the expansion of the solution $u_{\varepsilon}$
		\begin{equation*}
			\hat{\mathcal{T}} = \{ X^{k}, \, \begin{tikzpicture}[scale=0.4,baseline=-2]
				\coordinate (root) at (0,-0.4);
				\coordinate (t2) at (0,0.5);
				\draw[] (t2) -- (root);
				\draw (root) node[not] {};
				\draw (0,0.7) node[] {\tiny$\tau$};
			\end{tikzpicture}, \, k \in \mathbb{N}^2, \tau \in \mathcal{T} \}.
		\end{equation*}
		The fact that the ansatz solves the equation determines the definition of the coefficients $ F[\tau] $. This property is called coherence in the literature (see \cite[Section 4.2]{BCCH}).
		One has
		\begin{equation*}
			\begin{aligned}
				F [\begin{tikzpicture}[scale=0.4,baseline=-2]
					\coordinate (root) at (0,-0.4);
					\coordinate (t1) at (-0.3,0.5);
					\coordinate (t2) at (-1.1,0.5);
					\coordinate (t3) at (1.1,0.5);
					\coordinate (t4) at (1.1,0.5);
					\coordinate (tau1) at (1,0.6);
					\draw[kernels2] (t1) -- (root);
					\draw[kernels2] (t2) -- (root);
					\draw[] (t3) -- (root);
					\node[not,label= {[label distance=-0.2em]below: \tiny  $  k   $}] (rootnode) at (root) {};
					\draw (-1.2,0.7) node[] {\tiny$\tau_1$};
					\node[not] (rootnode) at (root) {};
					\draw (1.3,0.7) node[] {\tiny$\tau_m$};
					\draw (-0.4,0.7) node[] {\tiny$\tau_{\tiny{2}}$};
					\draw (0.5,0.7) node[] {\tiny$\cdots$};
				\end{tikzpicture}] & = \prod_{i=1}^m F[\tau_i]
				\partial^{k}	D_{\partial_x u}^2 D_u^{m-2} f(u) (\partial_x u)^2, \\
				F[ \begin{tikzpicture}[scale=0.4,baseline=-2]
					\coordinate (root) at (0,-0.4);
					\coordinate (t2) at (-0.8,0.5);
					\coordinate (t3) at (0.8,0.5);
					\draw[] (t2) -- (root);
					\draw[] (t3) -- (root);
					\draw (0,-0.4) node[xi,label= {[label distance=-0.2em]below: \tiny  $  k   $} ] {};
					\draw (-0.9,0.7) node[] {\tiny$\tau_1$};
					\draw (0.9,0.7) node[] {\tiny$\tau_n$};
					\draw (0,0.7) node[] {\tiny$\cdots$};
				\end{tikzpicture}] & = \prod_{i=1}^m F[\tau_i]
				\partial^{k}	 D_u^{n} g(u), \quad \hat{F}[ \begin{tikzpicture}[scale=0.4,baseline=-2]
					\coordinate (root) at (0,-0.4);
					\coordinate (t2) at (0,0.5);
					\draw[] (t2) -- (root);
					\draw (root) node[not] {};
					\draw (0,0.7) node[] {\tiny$\tau$};
				\end{tikzpicture}  ] = F[\tau],
			\end{aligned}
		\end{equation*}
		where $ \partial^{k} $ are the partial derivatives for the multi-index $k \in \mathbb{N}^2$, $ D_u$ is the derivative in the $u$ variable and $ D_{\partial_x u} $ is the derivative in the $\partial_x u$ variable. Below, we provide some examples
		\begin{equation*}
			\begin{aligned}
				F[\generic](u) & = g(u), \quad   F[\Xii](u) = 	F[\generic](u) D_{u} g(u) = g(u) g'(u),
				\\
				F[\IXitwo](u)	& = F[\generic](u)^2  D_{\partial_x u}^2 f(u) (\partial_x u)^2 = 2 g(u)^2 f(u).  
			\end{aligned}
		\end{equation*}
		Before performing the local perturbative expansion, we introduce the recentered stochastic integrals.
		One has
		\begin{equation*}
			(\Pi_z	\generic)(\bar{z}) = \xi_{\varepsilon}(\bar{z}), \quad
			(\Pi_z \IXi)(\bar{z}) = (K * \xi_{\varepsilon})(\bar{z}) - 
			(K * \xi_{\varepsilon})(z).
		\end{equation*}
		One does not need to recenter the noise as it is a distribution in the limit. For the second term, one has
		\begin{equation*}
			(\Pi_z \IXi)(\bar{z}) \lesssim \Vert z- \bar{z} \Vert^{ \alpha + 2}_{\mathfrak{s}}.
		\end{equation*}	
		As $\alpha > -2$, the exponent $ \alpha +2 $ is positive. The recentering corresponds to introducing Taylor expansions up to a certain order depending on the degree of the decorated tree considered. We refrain from providing a general formula as it will not be used in these lecture notes.	Also, we do not give an explicit expressions of the symmetry factors.		
		Let us run a local perturbative expansion to see how these coefficients appear for the first terms.
		We start by rewriting the regularised equation in the mild form
		\begin{equation*}
			u_{\varepsilon} = K * ( f(u_{\varepsilon}) (\partial_x u_{\varepsilon})^2 + g(u_{\varepsilon}) \xi_{\varepsilon} ).
		\end{equation*}
		We consider the following expansion
		\begin{equation*}
			u_{\varepsilon}  = u_{\varepsilon}(z) + \mathcal{R}_{\mathcal{T}_0,z}
		\end{equation*}
		where $ \mathcal{T}_0 $ contains the nodes with no decoration that is $ X^{0} $. Then, we get the following Taylor expansion
		\begin{equation*}
			g(u_{\varepsilon}) = g(u_{\varepsilon}(z)) + \cdots
		\end{equation*} 
		We plug this ansatz into the non-linearity $g$ to get
		\begin{equation*}
			u_{\varepsilon} = K * ( g(u_\varepsilon(z)) \xi_{\varepsilon} + \cdots ) =   g(u_\varepsilon(z)) \,  K * \xi_{\varepsilon} + \cdots
		\end{equation*}
		This justifies the definitions of $F[\generic]$ and $\hat{F}[\IXi]$. In the sequel, we will work with a reduced set of decorated trees that appears in the renormalised equation. It is given by
		\begin{equation}
			\begin{aligned}
				\mathcal{T}_{\text{\tiny{Chr}}} & = \{  
				\begin{tikzpicture}[scale=0.4,baseline=-2]
					\draw (0,0) node[xi,label= {[label distance=-0.2em]below: \tiny  $  $} ] {};
				\end{tikzpicture},  \begin{tikzpicture}[scale=0.4,baseline=-2]
					\coordinate (root) at (0,-0.4);
					\coordinate (t2) at (-0.8,0.5);
					\coordinate (t3) at (0.8,0.5);
					\draw[] (t2) -- (root);
					\draw[] (t3) -- (root);
					\draw (0,-0.4) node[xi,label= {[label distance=-0.2em]below: \tiny  $     $} ] {};
					\draw (-0.9,0.7) node[] {\tiny$\tau_1$};
					\draw (0.9,0.7) node[] {\tiny$\tau_n$};
					\draw (0,0.7) node[] {\tiny$\cdots$};
				\end{tikzpicture},     \begin{tikzpicture}[scale=0.4,baseline=-2]
					\coordinate (root) at (0,-0.4);
					\coordinate (t1) at (-0.3,0.5);
					\coordinate (t2) at (-1.1,0.5);
					\coordinate (t3) at (1.1,0.5);
					\coordinate (t4) at (1.1,0.5);
					\coordinate (tau1) at (1,0.6);
					\draw[kernels2] (t1) -- (root);
					\draw[kernels2] (t2) -- (root);
					\draw[] (t3) -- (root);
					\node[not,label= {[label distance=-0.2em]below: \tiny  $     $}] (rootnode) at (root) {};
					\draw (-1.2,0.7) node[] {\tiny$\tau_1$};
					\node[not] (rootnode) at (root) {};
					\draw (1.3,0.7) node[] {\tiny$\tau_m$};
					\draw (-0.4,0.7) node[] {\tiny$\tau_{\tiny{2}}$};
					\draw (0.5,0.7) node[] {\tiny$\cdots$};
				\end{tikzpicture},  m,n \in \mathbb{N}^*, m \geq 2,  \tau_i \in \mathcal{T}_{\text{\tiny{Chr}}} \}.
			\end{aligned}
		\end{equation}
		We call this set the Christoffel trees. We will justify in the next section why one can consider only this set. The next proposition narrows down the set of decorated trees with a negative degree to a set very close to $\mathcal{T}_{\text{\tiny{Chr}}}$. The main difference is that one is allowed to use only a single thick edge or one monomial decoration on only one node within the decorated tree. 
		\begin{proposition}
			Let $\tau \in \mathcal{T}$ with $ \mathrm{deg}(\tau) \leq 0 $ then 
			\begin{itemize}
				\item $ \tau $  contains one monomial decoration $(0,1)$ and all  its nodes are of the form
				\begin{equation*} 
					\begin{tikzpicture}[scale=0.4,baseline=-2]
						\draw (0,0) node[xi,label= {[label distance=-0.2em]below: \tiny  $  $} ] {};
					\end{tikzpicture},  \begin{tikzpicture}[scale=0.4,baseline=-2]
						\coordinate (root) at (0,-0.4);
						\coordinate (t2) at (-0.8,0.5);
						\coordinate (t3) at (0.8,0.5);
						\draw[] (t2) -- (root);
						\draw[] (t3) -- (root);
						\draw (0,-0.4) node[xi,label= {[label distance=-0.2em]below: \tiny  $     $} ] {};
						\draw (-0.9,0.7) node[] {\tiny$\tau_1$};
						\draw (0.9,0.7) node[] {\tiny$\tau_n$};
						\draw (0,0.7) node[] {\tiny$\cdots$};
					\end{tikzpicture},     \begin{tikzpicture}[scale=0.4,baseline=-2]
						\coordinate (root) at (0,-0.4);
						\coordinate (t1) at (-0.3,0.5);
						\coordinate (t2) at (-1.1,0.5);
						\coordinate (t3) at (1.1,0.5);
						\coordinate (t4) at (1.1,0.5);
						\coordinate (tau1) at (1,0.6);
						\draw[kernels2] (t1) -- (root);
						\draw[kernels2] (t2) -- (root);
						\draw[] (t3) -- (root);
						\node[not,label= {[label distance=-0.2em]below: \tiny  $     $}] (rootnode) at (root) {};
						\draw (-1.2,0.7) node[] {\tiny$\tau_1$};
						\node[not] (rootnode) at (root) {};
						\draw (1.3,0.7) node[] {\tiny$\tau_m$};
						\draw (-0.4,0.7) node[] {\tiny$\tau_{\tiny{2}}$};
						\draw (0.5,0.7) node[] {\tiny$\cdots$};
					\end{tikzpicture},  m,n \in \mathbb{N}^*, m \geq 2,  \tau_i \in \mathcal{T}.
				\end{equation*}
				\item $ \tau $ does not have any monomial decoration and it contains at most one node of type
				\begin{equation*}
					\begin{tikzpicture}[scale=0.4,baseline=-2]
						\coordinate (root) at (0,-0.4);
						\coordinate (t1) at (-0.3,0.5);
						\coordinate (t2) at (-1.1,0.5);
						\coordinate (t3) at (1.1,0.5);
						\coordinate (t4) at (1.1,0.5);
						\coordinate (tau1) at (1,0.6);
						\draw[] (t1) -- (root);
						\draw[kernels2] (t2) -- (root);
						\draw[] (t3) -- (root);
						\node[not,label= {[label distance=-0.2em]below: \tiny  $     $}] (rootnode) at (root) {};
						\draw (-1.2,0.7) node[] {\tiny$\tau_1$};
						\node[not] (rootnode) at (root) {};
						\draw (1.3,0.7) node[] {\tiny$\tau_m$};
						\draw (-0.4,0.7) node[] {\tiny$\tau_{\tiny{2}}$};
						\draw (0.5,0.7) node[] {\tiny$\cdots$};
					\end{tikzpicture},  m,n \in \mathbb{N}^*, m \geq 2,  \tau_i \in \mathcal{T}.
				\end{equation*}
				The other nodes are the ones that define the Christoffel trees.
			\end{itemize}
		\end{proposition}
		\begin{proof}The equation is subcritical and the lowest degree is given by that of the noise which $ \alpha > -2 $. If $ \tau $ contains more monomials decorations than just $(0,1)$, then they can be removed and one still gets a decorated tree with a negative degree below $-2$ which is not possible. Removing one thick edge boils down to improving the regularity by $+1$, which has the same effect as  a monomial decoration $(0,1)$. This observation allows us to conclude.
		\end{proof}

		One observes for the generalised KPZ equation that the elementary differentials are monomials in $f,g$ and their derivatives. This introduces some redundancy in the coding. For example, one has
		\begin{equation*}
			F[\IiIXiiib](u) =  F[\IiIXiii](u) = 2 f(u) g'(u) g(u)^2.
		\end{equation*}
		This observation motivates the introduction of a new encoding where one avoids such identities.
		We consider a set of abstract variables $ z_{(\generic,k)}, z_{(\derivatives,k)} $ with $ k \in \mathbb{N} $. Then, we consider the following set of monomials
		\begin{equation*}
			\begin{aligned}
				\mathfrak{M}_{ \derivatives, \generic} &= \lbrace z^{\beta} = \prod_{k \geq 0}  \left( z_{( \derivatives,k)} \right)^{\beta( \derivatives,k )}  \prod_{k \geq 0}\left(  z_{(\generic,k)}\right)^{\beta(\generic,k)} : 
				[\beta] =1 \rbrace\\
			\end{aligned}
		\end{equation*}
		where the $\beta : \{ (\generic,k), (\derivatives,k), \, k \in \mathbb{N} \} \rightarrow \mathbb{N} $ have finite support and one has the following population condition $ [\beta] =1 $ with $ [\beta] $ defined by
		\begin{equation*}  \begin{aligned}
				[\beta] & = \sum_{k \in \mathbb{N}} k \beta(   \generic,k) + (2+k) \beta(\derivatives,k)  -  \sum_{k \in \mathbb{N}} \beta(  \generic,k) +  \beta(\derivatives,k)
				\\ & = \sum_{k \in \mathbb{N}} (k-1) \beta(   \generic,k) + (1+k) \beta(\derivatives,k).
			\end{aligned}
		\end{equation*}
		We denote by 	$\mathcal{M}_{ \derivatives, \generic}$ the linear span of $ 	\mathfrak{M}_{ \derivatives, \generic} $.
		One can interpret each variable $ z_{(\generic,k)}, z_{(\derivatives,k)} $
		as a node with $k$ thin incoming edges and two thick incoming edges for $  z_{(\derivatives,k)} $.
		One can define  a degree map $ \mathrm{deg} $ for these multi-indices by first defining it on each $z$ variable
		\begin{equation*}
			\begin{aligned}
				\mathrm{deg}( z_{(\generic,k)}  ) & = \alpha + 2 k , \quad 	\mathrm{deg}( z_{(\derivatives,k)}  ) =  2 + 2 k,
				\\ \mathrm{deg}(  z^{\beta} ) &= \sum_{k \in \mathbb{N}} ( \beta(\generic,k) \mathrm{deg}( z_{(\generic,k)} ) + \beta(\derivatives,k) \mathrm{deg}(z_{(\derivatives,k)})).
			\end{aligned}
		\end{equation*}
		The population condition is a combinatorial condition that guarantees the possibility of constructing at least one trees out of these nodes.
		Below, we provide a map $ \Psi $ between the Christoffel  trees and $ 	\mathfrak{M}_{ \derivatives, \generic} $
		\begin{equation}
			\begin{aligned}
				\Psi \left( \begin{tikzpicture}[scale=0.4,baseline=-2]
					\coordinate (root) at (0,-0.4);
					\coordinate (t2) at (-0.8,0.5);
					\coordinate (t3) at (0.8,0.5);
					\draw[] (t2) -- (root);
					\draw[] (t3) -- (root);
					\draw (0,-0.4) node[xi,label= {[label distance=-0.2em]below: \tiny  $     $} ] {};
					\draw (-0.9,0.7) node[] {\tiny$\tau_1$};
					\draw (0.9,0.7) node[] {\tiny$\tau_n$};
					\draw (0,0.7) node[] {\tiny$\cdots$};
				\end{tikzpicture} \right) &= z_{(\generic,n)}  \prod_{i=1}^n \Psi(\sigma_i), \quad	\Psi \left( \begin{tikzpicture}[scale=0.4,baseline=-2]
					\coordinate (root) at (0,-0.4);
					\coordinate (t1) at (-0.3,0.5);
					\coordinate (t2) at (-1.1,0.5);
					\coordinate (t3) at (1.1,0.5);
					\coordinate (t4) at (1.1,0.5);
					\coordinate (tau1) at (1,0.6);
					\draw[kernels2] (t1) -- (root);
					\draw[kernels2] (t2) -- (root);
					\draw[] (t3) -- (root);
					\node[not] (rootnode) at (root) {};
					\draw (-1.2,0.7) node[] {\tiny$\tau_1$};
					\node[not] (rootnode) at (root) {};
					\draw (1.3,0.7) node[] {\tiny$\tau_m$};
					\draw (-0.4,0.7) node[] {\tiny$\tau_{\tiny{2}}$};
					\draw (0.5,0.7) node[] {\tiny$\cdots$};
				\end{tikzpicture} \right)  = z_{(\begin{tikzpicture}[scale=0.1,baseline=-2]
						\coordinate (root) at (0,-0.4);
						\coordinate (t1) at (-.8,1.3);
						\coordinate (t2) at (.8,1.3);
						\draw[kernels2] (t1) -- (root);
						\draw[kernels2] (t2) -- (root);
						\node[not] (rootnode) at (root) {};
					\end{tikzpicture},m-2)} \prod_{i=1}^m \Psi(\tau_i). 
			\end{aligned}
		\end{equation}
		
		Multi-indices were first introduced in the context of quasi-linear SPDEs in \cite{OSSW}.	The recentering of stochastic data indexed by multi-indices is performed in \cite{LOT,BL23} and their convergence is established in \cite{LOTT}. A solution theory has been proposed in \cite{Broux25} for a large class of subcritical equations. A survey on multi-indices for singular SPDEs can be found in \cite{BOT24}.
		In the next Lemma, we show that the degree map $ \mathrm{deg} $ is preserved by the map $ \Psi $. It is an opportunity for te reader to manipulate the recursive formula provided by $\Psi$.
		
		\begin{lemma} For every Christoffel tree $\tau$, one has
			\begin{equation*}
				\mathrm{deg}(\tau) = \mathrm{deg}(\Psi(\tau)).
			\end{equation*}
		\end{lemma}
		\begin{proof}
			Indeed, one proceeds by induction. It is clear that
			\begin{equation*}
				\mathrm{deg}(\generic) = \alpha = \mathrm{deg}(\generic,0) = \mathrm{deg}(\Psi(\generic)).
			\end{equation*}
			Then,  one has
			\begin{equation*}
				\mathrm{deg}(\Psi \left( \begin{tikzpicture}[scale=0.4,baseline=-2]
					\coordinate (root) at (0,-0.4);
					\coordinate (t2) at (-0.8,0.5);
					\coordinate (t3) at (0.8,0.5);
					\draw[] (t2) -- (root);
					\draw[] (t3) -- (root);
					\draw (0,-0.4) node[xi,label= {[label distance=-0.2em]below: \tiny  $     $} ] {};
					\draw (-0.9,0.7) node[] {\tiny$\tau_1$};
					\draw (0.9,0.7) node[] {\tiny$\tau_n$};
					\draw (0,0.7) node[] {\tiny$\cdots$};
				\end{tikzpicture} \right)) = \mathrm{deg}(z_{(\generic,n)}  \prod_{i=1}^n \Psi(\tau_i)) =  \mathrm{deg}(z_{(\generic,n)}) + \sum_{i=1}^n \mathrm{deg}( \Psi(\tau_i) ).
			\end{equation*}
			On the other hand,
			\begin{equation*}
				\mathrm{deg}( \begin{tikzpicture}[scale=0.4,baseline=-2]
					\coordinate (root) at (0,-0.4);
					\coordinate (t2) at (-0.8,0.5);
					\coordinate (t3) at (0.8,0.5);
					\draw[] (t2) -- (root);
					\draw[] (t3) -- (root);
					\draw (0,-0.4) node[xi,label= {[label distance=-0.2em]below: \tiny  $     $} ] {};
					\draw (-0.9,0.7) node[] {\tiny$\tau_1$};
					\draw (0.9,0.7) node[] {\tiny$\tau_n$};
					\draw (0,0.7) node[] {\tiny$\cdots$};
				\end{tikzpicture} ) = \mathrm{deg}(\generic) + \sum_{i=1}^n \mathrm{deg}( \tau_i ).
			\end{equation*}
			We conclude by applying the induction hypothesis to the $ \tau_i $ and using the base case on $\generic$. The proof for the other product with two thick edges proceeds in the same way.
		\end{proof}
		In the Table~\ref{tab:kpz_relevant}, we list the Christoffel trees with a negative degree and their corresponding multi-indices. The decorated trees associated with the same multi-index are grouped together.
		As an example, one has
		\begin{equation*}
			\Psi(\IiIXiiib) =\Psi(\IiIXiii)  = z_{ (\generic,0)}^2 z_{ (\generic,1)}z_{( \derivatives,0)}.
		\end{equation*}
		Another set of multi-indices  of interest in the sequel is
		\begin{equation*}
			\begin{aligned}
				\mathfrak{M}_{ \begin{tikzpicture}[scale=0.2,baseline=-2]
						\coordinate (root) at (0,0);
						\node[diff] (rootnode) at (root) {};
					\end{tikzpicture}, \derivatives, \generic} &= \lbrace z^{\beta} = \prod_{k \geq 0} \left( z_{( \begin{tikzpicture}[scale=0.2,baseline=-2]
						\coordinate (root) at (0,0);
						\node[diff] (rootnode) at (root) {};
					\end{tikzpicture},k)} \right)^{\beta( \begin{tikzpicture}[scale=0.2,baseline=-2]
						\coordinate (root) at (0,0);
						\node[diff] (rootnode) at (root) {};
					\end{tikzpicture},k )}  ( z_{( \derivatives,k)} )^{\beta( \derivatives,k )}  \prod_{k \geq 0}(  z_{(\generic,k)})^{\beta(\generic,k)} : 
				[\beta] =1 \rbrace\\
			\end{aligned}
		\end{equation*}
		where we have added an extra variable $ z_{(\begin{tikzpicture}[scale=0.2,baseline=-2]
				\coordinate (root) at (0,0);
				\node[diff] (rootnode) at (root) {};
			\end{tikzpicture},k)} $ and we assume that one has exactly one variable of this type. We do not assign a degree to these new multi-indices as they will be used to characterise the chain rule symmetry. The new variable will correspond to an infinitesimal change of coordinates. One is able to define elementary differentials associated with these multi-indices
		\begin{equation*}
			F( z_{( \derivatives,k)} )(u) = 
			f^{(k)}(u), \quad F(z_{( \generic,k)})(u) = g^{(k)}(u), \quad F(z_{(\begin{tikzpicture}[scale=0.2,baseline=-2]
					\coordinate (root) at (0,0);
					\node[diff] (rootnode) at (root) {};
				\end{tikzpicture},k)} )(u) = h^{(k)}(u).
		\end{equation*}
		Then, it extended multiplicatively to a monomial. One has
		\begin{equation*}
			\begin{aligned}
				F(z_{ (\generic,0)}^2 z_{ (\generic,1)}z_{( \derivatives,0)})(u) & = 	F(z_{ (\generic,0)})(u)^2 F( z_{ (\generic,1)})(u) F(z_{( \derivatives,0)})(u) 
				\\
				&= 	g(u)^2 g'(u) f(u) 
				\\ & = 	F(\Psi(\IiIXiiib)) = F( \Psi(\IiIXiii)). 
			\end{aligned}
		\end{equation*}
		In fact one has
		\begin{lemma} For every Christoffel tree $ \tau $, one has
			\begin{equation*}
				F( \tau ) = F(\Psi(\tau)).
			\end{equation*}
		\end{lemma}
		\begin{proof} The proof works by induction as the one for the degree map.
		\end{proof}

		\begin{table}[ht]
			\begin{tabular}{c | c | c}
				Degree  &  Multi-indices &  Trees \\
				\hline
				&&\\
				$-1-2 \kappa$ & $ z_{ (\generic,0)}z_{ (\generic,1)}$, $\quad z_{ (\generic,0)}^2 z_{(\derivatives,0)}$  &\Xii \,,\;\IXitwo  \\ & &  \\
				\hline 
				& & \\
				& $z_{ (\generic,0)}^2z_{ (\generic,2)}$, $ \quad z_{ (\generic,0)}z_{ (\generic,1)}^2$,  &  \Xiiia \,,
				\Xiii\,, \\ $-\frac{1}{2} - 3 \kappa$& $ z_{ (\generic,0)}^2z_{ (\generic,1)}z_{( \derivatives,0)}$, &\{ \IiIXiiib \,, \IiIXiii\,\}     \\  & $ z_{ (\generic,0)}^3 z_{( \derivatives,0)}^2$,     $\quad z_{ (\generic,0)}^3z_{( \derivatives,1)}$&  \IiIXiiic \,, \IiXiiic\,  \\  & & \\ \hline
				&&\\ & $z_{ (\generic,0)}z_{ (\generic,1)}^3$, $\quad  z_{ (\generic,0)}^3z_{ (\generic,3)} $,  & \Xiiii \,, \;  \Xiiiib\,, \;   \\ & $z_{ (\generic,0)}^2z_{ (\generic,1)}z_{ (\generic,2)}$,& \{ \Xiiiic \,, \Xiiiie \}, \\  & $ z_{ (\generic,0)}^2z_{ (\generic,1)}^2z_{(\derivatives,0)}$,  & \{\IiXiiiia \,,  \IiXiiiib \,, \IiXiiiic \,, \iiIiXiiiib \}\,,  \\ & $ z_{ (\generic,0)}^3z_{ (\generic,1)}z_{( \derivatives,0)}$,   & \{\Xiiiieb \,,  \Xiiiicb \},   \;   \\ $ - 4 \kappa $ &$ z_{ (\generic,0)}^3z_{ (\generic,1)}z_{(\derivatives,0)}^2 $,& \{ \IiXiiiiab \,, \IiXiiiibc \,, \IiXiiiiac \,, \iiIiXiiiic \},\\&  $z_{ (\generic,0)}^4 z_{( \derivatives,0)}^3$, & \{\Xitwo \,, \iiIiXiiii \}, \\ &$ z_{ (\generic,0)}^4z_{( \derivatives,0)}z_{(\derivatives,1)},$& \{  \Xiiiieab \,, \Xiiiieabbis \,, \Xiiiicab  \},\; \\ &  $ z_{ (\generic,0)}^4 z_{(\derivatives,2)} $, $\quad  z_{ (\generic,0)}^3 z_{ (\generic,1)}z_{( \derivatives,1)} $& \; \Xiiiiba, \{\Xiiiica \,,  \Xiiiiea \,, \Xiiiieabis\} \\ && \\
				\hline 
			\end{tabular}
			\caption{We present all the Christoffel trees with negative degree (except the noise) and their multi-indices. The degree is computed for the space-time white noise. }
			\label{tab:kpz_relevant}
		\end{table}
		
		The next theorem justifies the use of multi-indices as they provide the best encoding of the elementary differentials for scalar-valued equations.
		We introduce the following notation
		\begin{equation*}
			F(z^{\beta}) = 	F_{f,g}^{h}(z^{\beta})
		\end{equation*}
		in order to stress the dependency on the functions $f,g$ and $h$.
		\begin{theorem} \label{injectivity_Upsilon}
			The map $  F_{f,g}^{h} $ is injective in the sense that if for any $f,g,h \in \mathcal{C}^{\infty}(\mathbb{R}, \mathbb{R})$ one has $F_{f,g}^h(v) = 0 $ for some $v\in \mathcal{M}_{\begin{tikzpicture}[scale=0.2,baseline=-2]
					\coordinate (root) at (0,0);
					\node[diff] (rootnode) at (root) {};
				\end{tikzpicture},\derivatives,\generic}$ then $ v=0$.
		\end{theorem}
		\begin{proof}
			Writing the hypothesis in coordinates, this means that for any $f,g,h \in \mathcal{C}^{\infty}(\mathbb{R}, \mathbb{R})$ one has
			\begin{equation*}
				\sum_{\beta} \lambda_{\beta} F_{f, g}^h(z^\beta) = 0 \,,
			\end{equation*}
			where the sum runs over all  $ z^\beta \in \mathfrak{M}_{\begin{tikzpicture}[scale=0.2,baseline=-2]
					\coordinate (root) at (0,0);
					\node[diff] (rootnode) at (root) {};
				\end{tikzpicture}, \derivatives, \generic}$  and the coefficients $ \lambda_{\beta} $  are all zero with the exception of a finite number. We want to prove $\lambda_{\beta}=0$ for all $\beta$. We can therefore choose specific $f, g,h$ for discriminating the various $ z^{\beta} $. Let $ m $ be the highest integer such that  there exists $ z^{\beta} $ with $  \lambda_{\beta} \neq 0 $ and $ (\beta(\begin{tikzpicture}[scale=0.2,baseline=-2]
				\coordinate (root) at (0,0);
				\node[diff] (rootnode) at (root) {};
			\end{tikzpicture},m),
			\beta( \derivatives,m ), \beta(\generic,m)) \neq (0,0,0)$. We define
			\begin{equation*}
				\begin{aligned}
					h(x) & =	h_{r_0,...,r_m}(x) = \sum_{k=0}^m r_k \frac{x^k}{k!}, \quad f(x) =	f_{s_0,...,s_m}(x) = \sum_{k=0}^m s_k \frac{x^k}{k!},
					\\  g(x) & =	g_{t_0,...,t_m}(x) = \sum_{k=0}^m t_k \frac{x^k}{k!}.
				\end{aligned}
			\end{equation*} 
			By choosing  functions $\Gamma,\sigma,h$ parametrised by $r_0,s_0,t_0,...,r_m,s_m,t_m$  one has
			\begin{equation*}
				F^h_{f,g}(z^\beta)(0) =  \prod_{k=0}^m \left( 
				\left( 	r_k \right)^{\beta(\begin{tikzpicture}[scale=0.2,baseline=-2]
						\coordinate (root) at (0,0);
						\node[diff] (rootnode) at (root) {};
					\end{tikzpicture},k)}
				\left( 2  s_k  \right)^{\beta( \derivatives,k )} \left( 	t_k  \right)^{\beta(\generic,k)} \right).
			\end{equation*}
			One can observe that we obtain a monomial in $r_0,s_0,t_0,...,r_m,s_m,t_m$ which is uniquely associated to a multi-index $\beta$. Since this family of monomials is clearly free and describes smooth functions up to Taylor polynomials, we can conclude that $v=0$.
		\end{proof}
		
		We finish this introduction to multi-indices by mentioning a recursive construction of the processes $ \Pi_y z^{\beta} $.
		One can write a hierarchy of equations
		\begin{equation*}
			\begin{aligned}
				(\partial_t - \partial_x^2)  \Pi_y z^{\beta}& = \sum_{\beta =e_{(\derivatives,k)} + \alpha_1 + \alpha_2 + \sum_{i=1}^k \beta_i} \prod_{i=1}^k \Pi_y z^{\beta_i}  \partial_x \Pi_y z^{\alpha_1} \partial_x \Pi_y z^{\alpha_2} \\ & +  \sum_{\beta =e_{(\generic,k)} + \sum_{i=1}^k \beta_i} \prod_{i=1}^k \Pi_y z^{\beta_i}  \xi
			\end{aligned}
		\end{equation*}
		where the $\alpha_i, \beta_i$ are populated multi-indices and the $ e_{(\generic,k)} $ are non-zero and equal to one on $(\generic,k)$. One has the same definition for $e_{(\derivatives,k)}$.  These equations define inductively the stochastic iterated integrals and one can insert recentering steps by adding well-chosen polynomials to get $ \Pi_y z^{\beta} $. This construction also works for decorated trees but in the case of multi-indices one groups more terms together. One can observe that the choices of $ e_{(\generic,k)}, e_{(\derivatives,k)}  $ in the decomposition of $\beta$ correspond to the choice of a root and then one chooses the branches with the $\alpha_i, \beta_i$.

		\section{Cancellations via integration by parts}
		
		In order to get a solution theory, one has to construct the new monomial basis that is the $\Pi_x \tau$. It is sufficient to get them for $\tau$ with a negative degree as the ones with a positive degree can be recovered via an extension theorem. One has the following theorem 
		
		\begin{theorem}
			For every $\tau \in \mathcal{T}$ with $ \mathrm{deg}(\tau) \leq 0 $, one can construct a renormalised monomial $ \widehat{\Pi}^{\varepsilon}_z \tau $ such that uniformly in $\varepsilon$
			\begin{equation*}
				\mathbb{E}( |(\widehat{\Pi}^{\varepsilon}_z \tau)(\varphi^\lambda_z) |^2 ) \lesssim \lambda^{2 \deg(\tau)} 
			\end{equation*}
			where $ \varphi^{\lambda}_z $ is a suitable test function rescaled around $z = (t,x)$. One has for $\bar{z} = (\bar{t},\bar{x})$
			\begin{equation*}
				\varphi^{\lambda}_z(\bar{z}) = \lambda^{-3} \varphi( \lambda^{-2}(\bar{t}-t), \lambda^{-1}(\bar{x}-x) ).
			\end{equation*}
		\end{theorem}
		
		In the previous theorem, one can control only the second moment for the convergence because the noise is Gaussian. This is the formulation given in \cite{reg}. Getting this theorem in full generality for a large class of equations was a major open problem for singular SPDEs. One solves it with several techniques.
		One has to first build algebraically the renormalised objects $ \widehat{\Pi}_z^{\varepsilon} \tau $ where one has to add the renormalisation without breaking the recentering procedure. This recentering is needed to get the correct scaling behaviour in the limit. This task has been implemented for decorated trees in \cite{BHZ,BR18} and for multi-indices in \cite{LOT,BL23}. Then, it remains to prove the convergence of the renormalised monomials.
		There are two main approaches
		\begin{itemize}
			\item A diagrammatic approach which builds Feynman diagrams and proceeds with the BPHZ renormalisation run on these diagrams. 
			One can see how and why the renormalisation has to be put on these diagrams. This was performed in \cite{CH}.
			\item An inductive approach based on a spectral gap assumption on the noise. The main idea is to control the moments of  $ \widehat{\Pi}_z^{\varepsilon} \tau  $ via the moments of its Malliavin derivative $ \delta \widehat{\Pi}_z^{\varepsilon} \tau  $ which can be obtained via a reconstruction argument. The renormalisation is used to remove the divergent part in the spectral gap inequality that is $ \mathbb{E}(\widehat{\bm{\Pi}}^{\varepsilon} \tau ) $. Here, $\widehat{\bm{\Pi}}^{\varepsilon}$ is called the pre-model, the stochastic iterated integrals without any recentering.  This approach was first proposed in \cite{LOTT} for multi-indices, and it was then extended to decorated trees in \cite{BN23,HS,BH23}.
		\end{itemize}
		Let us stress that the spectral gap is also used in the convergence proof of stochastic data constructed via iterated paraproducts (see \cite{BM26}). It is also close in spirit to the flow recursion that allows one to perform an inductive proof in \cite{Duc21}.
		Describing precisely this convergence procedure is beyond the scope of this course. We will briefly explain the diagrammatic approach for a few terms of the general KPZ equation and observe some cancellations among renormalisation constants. These cancellations are obtained via some integration by parts identities and a key identity on the heat kernel. 
		We start with $\tau =  \Xii$ and one has
		\begin{equation*}
			(\Pi_z^{\varepsilon}  \Xii)(\bar{z}) = (\bm{\Pi}^{\varepsilon} \Xii)(\bar{z}) -  \xi_{\varepsilon}(\bar{z}) (K * \xi_{\varepsilon})(z), \quad (\bm{\Pi}^{\varepsilon} \Xii)(\bar{z}) = (\xi_{\varepsilon} K * \xi_{\varepsilon})(\bar{z}).
		\end{equation*}
		Then, taking the expectation and using the fact that $(\bm{\Pi}^{\varepsilon}  \Xii)(\bar{z})$ is invariant by translation in law, one has
		\begin{equation*}
			\begin{aligned}
				\mathbb{E}( (\bm{\Pi}^{\varepsilon} \Xii)(\bar{z}) ) &  = \mathbb{E}( (\xi_{\varepsilon} K * \xi_{\varepsilon})(0))
				=  \mathbb{E}( \int \xi_{\varepsilon}(0) K(-z) \xi_{\varepsilon}(z) dz )
				\\ &  = \int K(-z) \mathbb{E}( \xi_{\varepsilon}(0)  \xi_{\varepsilon}(z)) dz 
				= \int K(-z) \varrho_{\varepsilon}^2(z) dz
			\end{aligned}
		\end{equation*}
		where $ \varrho_{\varepsilon}^2(z) = (\varrho_{\varepsilon} * \varrho_{\varepsilon})(z) $ which converges to a Dirac delta distribution at the origin when $\varepsilon$ tends to zero. As the heat kernel is singular at zero, this expectation is divergent. A similar computation shows that the term $\mathbb{E}(\xi_{\varepsilon}(\bar{z}) (K * \xi_{\varepsilon})(z))$ is not divergent. Therefore, the renormalised stochastic iterated integral is given by
		\begin{equation*}
			(\widehat{\Pi}^{\varepsilon}_z  \Xii)(\bar{z}) = 	({\Pi}^{\varepsilon}_z  \Xii)(\bar{z})	 - \mathbb{E}( (\bm{\Pi}^{\varepsilon} \Xii  )(0)).
		\end{equation*}
		It turns out that this renormalisation is compensated by another. Indeed, one has
		\begin{equation*}
			(\Pi_z^{\varepsilon} \IXitwo)(\bar{z})  = (\bm{\Pi}^{\varepsilon} \IXitwo)(\bar{z}) =  (\partial_x K * \xi_{\varepsilon})^2(\bar{z}).
		\end{equation*}
		Then
		\begin{equation*}
			\begin{aligned}
				\mathbb{E}( (\Pi_z^{\varepsilon}  \IXitwo)(\bar{z}) ) &  = \mathbb{E}( (\partial_x K * \xi_{\varepsilon})^2(0)  )
				=  \mathbb{E}( (\int \partial_x K(-z) \xi_{\varepsilon}(z) dz)^2 )
				\\ &  = \int \partial_x K(-z) \partial_x K(-\bar{z}) \mathbb{E}( \xi_{\varepsilon}(z)  \xi_{\varepsilon}(\bar{z})) dz d \bar{z}
				\\ & = 	\int \partial_x K(-z) \partial_x K(-\bar{z}) \varrho_{\varepsilon} * \varrho_{\varepsilon}(z-\bar{z})  dz d \bar{z}.
			\end{aligned}
		\end{equation*}
		Here, $(\partial_x K)^2$ is not integrable around zero. One sets
		\begin{equation*}
			(\widehat{\Pi}^{\varepsilon}_z  \IXitwo)(\bar{z}) = (\partial_x K * \xi_{\varepsilon})^2(\bar{z}) - \mathbb{E}( (\bm{\Pi}^{\varepsilon} \IXitwo  )(0)).
		\end{equation*}
		In order to see the cancellation, we use the following identity on the heat kernel:
		\begin{equation} \label{key_identity}
			\int \partial_x K(z_1 - z) \partial_x K(z_2 - z) dz = \frac{1}{2} ( K(z_1 - z_2) + K(z_2 -z_1))
		\end{equation}
		which is a consequence of the fact that $K$ is the fundamental solution of the heat equation. In practice, one cannot use exactly \eqref{key_identity} as one works with a mollified kernel $\varrho_{\varepsilon} * K$. Moreover, one does not use the heat kernel but a truncated version. Therefore, the identity is true up to an error term $R_{\varepsilon}$ that converges to zero. We will omit this error term and proceed formally with the computations of the cancellations and use the symbol $\approx$.
		By applying \eqref{key_identity} with $z_1 =z_2$, one gets
		\begin{equation} \label{first_cancellation}
			\mathbb{E}( (\bm{\Pi}^{\varepsilon} \Xii  )(0))  \approx 
			\mathbb{E}( (\bm{\Pi}^{\varepsilon}  \IXitwo  )(0)).
		\end{equation}
		To compute other cancellations, one needs more graphical notations. We will use decorated graphs where nodes are integrable variables and  edges are kernels evaluated at their end points. Some edges will be oriented.
		\begin{itemize}
			\item An edge of type $
			\begin{tikzpicture}[baseline=0.6cm,scale=0.35]
				\node at (3,2) [dot,label=below: \scriptsize{$z_1$}] (1) {};
				\node at (0,2) [dot,label=below: \scriptsize{$z_2$}] (2) {};
				\draw[kernel,->] (2) to (1);
			\end{tikzpicture} $
			encodes the kernel $K(z_1 - z_2)$.
			\item An edge of type $
			\begin{tikzpicture}[baseline=0.6cm,scale=0.35]
				\node at (3,2) [dot,label=below: \scriptsize{$z_1$}] (1) {};
				\node at (0,2) [dot,label=below: \scriptsize{$z_2$}] (2) {};
				\draw[kernels,->] (2) to (1);
			\end{tikzpicture} $
			encodes the kernel $ \partial_x K(z_1 - z_2)$.
			\item A node of type  $
			\begin{tikzpicture}[baseline=0.6cm,scale=0.35]
				\node at (0,2) [root] (2) {};
			\end{tikzpicture} $ means that the variable is evaluated at zero. For example, one has that $
			\begin{tikzpicture}[baseline=0.6cm,scale=0.35]
				\node at (3,2) [dot,label=below: \scriptsize{$z_1$}] (1) {};
				\node at (0,2) [root] (2) {};
				\draw[kernel,->] (2) to (1);
			\end{tikzpicture} $
			encodes the kernel $K(z_1)$.
			\item An edge of type 
			$	\begin{tikzpicture}[baseline=0.6cm,scale=0.35]
				\node at (3,2) [dot,label=below: \scriptsize{$z_1$}] (1) {};
				\node at (0,2) [dot,label=below: \scriptsize{$z_2$}] (2) {};
				\draw[rho] (2) to (1);
			\end{tikzpicture} $ encodes a mollifier $ \varrho(z_1 - z_2) $. Due to symmetry of the mollifier we do not put any direction on this edge.
		\end{itemize}
		In the sequel, we will not distinguish  between $K$ and $\varrho_{\varepsilon} * K $ in the graphical notation. We will represent the mollified  edge when it is not convolved with a kernel.
		One can reinterpret \eqref{key_identity} graphically as 
		\begin{equation} \label{graphical_K}
			\begin{tikzpicture}[baseline=0.6cm,scale=0.35]
				\node at (-3,2) [dot,label=below: \scriptsize{$z_1$}] (0) {};
				\node at (3,2) [dot,label=below: \scriptsize{$z_2$}] (1) {};
				\node at (0,2) [dot,label=below: \scriptsize{$z$}] (2) {};
				\draw[kernels,->] (2) to (1);
				\draw[kernels,->] (2) to (0);
			\end{tikzpicture} = \frac{1}{2} \begin{tikzpicture}[baseline=0.6cm,scale=0.35]
				\node at (-1.5,2) [dot,label=below: \scriptsize{$z_1$}] (0) {};
				\node at (1.5,2) [dot,label=below: \scriptsize{$z_2$}] (1) {};
				\draw[kernel,->] (1) to (0);
			\end{tikzpicture}  + \frac{1}{2} \begin{tikzpicture}[baseline=0.6cm,scale=0.35]
				\node at (-1.5,2) [dot,label=below: \scriptsize{$z_1$}] (0) {};
				\node at (1.5,2) [dot,label=below: \scriptsize{$z_2$}] (1) {};
				\draw[kernel,->] (0) to (1);
			\end{tikzpicture}
		\end{equation}
		where we have denoted the associated variable below each node. One can perform some pictorial manipulations. For example,
		\begin{equation} \label{permute_edge}
			\begin{tikzpicture}[baseline=0.6cm,scale=0.35]
				\node at (-3,2) [dot,label=below: \scriptsize{$z_1$}] (0) {};
				\node at (3,2) [dot,label=below: \scriptsize{$z_2$}] (1) {};
				\node at (0,2) [dot,label=below: \scriptsize{$z$}] (2) {};
				\draw[kernels,->] (1) to (2);
				\draw[kernels,->] (0) to (2);
			\end{tikzpicture} = 	\begin{tikzpicture}[baseline=0.6cm,scale=0.35]
				\node at (-3,2) [dot,label=below: \scriptsize{$z_1$}] (0) {};
				\node at (3,2) [dot,label=below: \scriptsize{$z_2$}] (1) {};
				\node at (0,2) [dot,label=below: \scriptsize{$z$}] (2) {};
				\draw[kernels,->] (2) to (1);
				\draw[kernels,->] (2) to (0);
			\end{tikzpicture} 
		\end{equation}
		which corresponds to the change of variable	 $ z' = z_1+z_2 - z $
		\begin{equation*}
			\int \partial_x K(z_1 - z) \partial_x K(z_2 - z) dz =  \int \partial_x K(z'-z_2) \partial_x K(z' - z_1) dz'.
		\end{equation*}
		Another rule that we will use frequently is the following. In the graphs considered, there will be one distinguished node coloured green associated with the variable zero. As the integrals associated with these graphs are invariant under translation, one can replace this distinguished node with another. For example, one has in the next graph
		
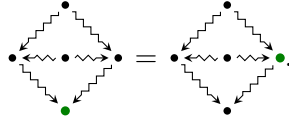
\begin{equation*}
			\begin{tikzpicture}[baseline=0.6cm,scale=0.35]
				\node at (0,0) [root] (0) {};
				\node at (2,2) [dot] (1) {};
				\node at (-2,2) [dot] (2) {};
				\node at (0,4) [dot] (3) {};
				\node at (0,2) [dot] (4) {};
				\draw[kernels,->] (1) to (0);
				\draw[kernels,->] (2) to (0);
				\draw[kernels,->] (3) to (1);
				\draw[kernels,->] (3) to (2);
				\draw[kernels,->] (4) to (2);
				\draw[kernels,->] (4) to (1);
			\end{tikzpicture} \, =  \, 	\begin{tikzpicture}[baseline=0.6cm,scale=0.35]
				\node at (0,0) [dot] (0) {};
				\node at (2,2) [root] (1) {};
				\node at (-2,2) [dot] (2) {};
				\node at (0,4) [dot] (3) {};
				\node at (0,2) [dot] (4) {};
				\draw[kernels,->] (1) to (0);
				\draw[kernels,->] (2) to (0);
				\draw[kernels,->] (3) to (1);
				\draw[kernels,->] (3) to (2);
				\draw[kernels,->] (4) to (2);
				\draw[kernels,->] (4) to (1);
			\end{tikzpicture}. 
		\end{equation*}
		We begin by redoing the computation of the cancellation \eqref{first_cancellation} with these graphical notations. One has
		\begin{equation*}
			\mathbb{E}( (\bm{\Pi}^{\varepsilon}  \IXitwo  )(0)) = \begin{tikzpicture}[baseline=0.6cm,scale=0.35]
				\node at (0,1) [root] (0) {};
				\node at (0,4) [dot] (1) {};
				\draw[decorate,shorten >=2pt,shorten <=2pt, decoration={zigzag, segment length=4pt, amplitude=1pt, post length = 5pt, pre length = 2pt},->] (1) to[bend right = 60] (0);
				\draw[decorate,shorten >=2pt,shorten <=2pt, decoration={zigzag, segment length=4pt, amplitude=1pt, post length = 5pt, pre length = 2pt},->] (1) to[bend left = 60] (0);
			\end{tikzpicture} \approx \begin{tikzpicture}[baseline=0.6cm,scale=0.35]
				\node at (0,1) [root] (0) {};
				\node at (0,4) [dot] (1) {};
				\draw[decorate,shorten >=2pt,shorten <=2pt, decoration={ segment length=4pt, amplitude=1pt, post length = 5pt, pre length = 2pt},->] (1) to[bend right = 60] (0);
				\draw[decorate,densely dashed,semithick,shorten >=2pt,shorten <=2pt, decoration={ segment length=4pt, amplitude=1pt, post length = 5pt, pre length = 2pt},->] (1) to[bend left = 60] (0);
			\end{tikzpicture} = \mathbb{E}( (\bm{\Pi}^{\varepsilon} \Xii  )(0))
		\end{equation*}
		where we have applied \eqref{graphical_K} with $z_1 = z_2 = 0$. Now, we want to consider more subtle cancellations by looking at the two stochastic integrals below:
		\begin{equation*}
			\begin{aligned}
				(\bm{\Pi}^{\varepsilon} \Xitwo ) & = ( \partial_x K * (    (\partial_xK *(  \partial_x K *\xi_{\varepsilon})^2) (\partial_x K *  \xi_{\varepsilon}) )) (\partial_x K *  \xi_{\varepsilon})
				\\
				(\bm{\Pi}^{\varepsilon} \iiIiXiiii)	 & =  (\partial_x K *( (\partial_x K * \xi_{\varepsilon})^2 ))^2.
			\end{aligned}
		\end{equation*}
		When one wants to compute the expectation of these stochastic integrals, one has to compute terms of the form $ \mathbb{E}(\prod_{i=1}^4\xi^{\varepsilon}(z_i)) $. As the $\xi^{\varepsilon}(z_i)$ are Gaussian variables, one can use Isserlis's theorem. Let $I$ be a finite set, and let $(X_i)_{i \in I}$ be a collection of
		 jointly centred Gaussian random variables. Then
		\begin{equation} \label{Isserles}
			\mathbb{E}( \prod_{i \in I} X_i ) = \sum_{\mathfrak{p}
				\in \mathcal{P}(I)} \prod_{(i,j) \in \mathfrak{p}} \mathbb{E}(X_i X_j)
		\end{equation}
		where $ \mathcal{P}(I) $ are partitions of $I$ with two elements of $I$ in each block of the partition. In the stochastic integrals, this corresponds to pair the noises $\xi_{\varepsilon}$. Then, one has
		\begin{equation*}
			\begin{aligned}
				\mathbb{E}((\bm{\Pi}^{\varepsilon} \Xitwo )(0)) & = 2 (\bm{\Pi}^{\varepsilon} \Xitwoo)(0) + (\bm{\Pi}^{\varepsilon} \Xitwooo)(0) =    2 \begin{tikzpicture}[baseline=0.6cm,scale=0.35]
					\node at (0,0) [root] (0) {};
					\node at (0,3) [dot] (1) {};
					\node at (-2,3) [dot] (2) {};
					\node at (2,4.5) [dot] (3) {};
					
					\node at (0,6) [dot] (4) {};
					
					\draw[kernels,->] (1) to (0);
					\draw[kernels,->] (2) to (0);
					\draw[kernels,->] (3) to (4);
					\draw[kernels,->] (3) to (1);
					\draw[kernels,->] (2) to (4);
					\draw[kernels,->] (4) to (1);
				\end{tikzpicture} +  \begin{tikzpicture}[baseline=0.6cm,scale=0.35]
					\node at (0,0) [root] (0) {};
					\node at (0,3) [dot] (1) {};
					\node at (-2,1.5) [dot] (2) {};
					\node at (0,9) [dot] (3) {};
					
					\node at (0,6) [dot] (4) {};
					
					\draw[kernels,->] (1) to (0);
					\draw[kernels,->] (2) to (0);
					\draw[decorate,shorten >=2pt,shorten <=2pt, decoration={zigzag, segment length=4pt, amplitude=1pt, post length = 5pt, pre length = 2pt},->] (3) to[bend right = 60] (4);
					\draw[decorate,shorten >=2pt,shorten <=2pt, decoration={zigzag, segment length=4pt, amplitude=1pt, post length = 5pt, pre length = 2pt},->] (3) to[bend left = 60] (4);
					\draw[kernels,->] (2) to (1);
					\draw[kernels,->] (4) to (1);
				\end{tikzpicture}
			\end{aligned}.
		\end{equation*}
		Here, we have added an extra colour to the noises to represent two pairs. Now, one understands the map $\bm{\Pi}^{\varepsilon}$ as mapping decorated trees with pairs to the correct Feynman diagram. Using the fact that the top loop in the second term is invariant under translation, one has
		\begin{equation*}
			\begin{tikzpicture}[baseline=0.6cm,scale=0.35]
				\node at (0,0) [root] (0) {};
				\node at (0,3) [dot] (1) {};
				\node at (-2,1.5) [dot] (2) {};
				\node at (0,9) [dot] (3) {};
				
				\node at (0,6) [dot] (4) {};
				
				\draw[kernels,->] (1) to (0);
				\draw[kernels,->] (2) to (0);
				\draw[decorate,shorten >=2pt,shorten <=2pt, decoration={zigzag, segment length=4pt, amplitude=1pt, post length = 5pt, pre length = 2pt},->] (3) to[bend right = 60] (4);
				\draw[decorate,shorten >=2pt,shorten <=2pt, decoration={zigzag, segment length=4pt, amplitude=1pt, post length = 5pt, pre length = 2pt},->] (3) to[bend left = 60] (4);
				\draw[kernels,->] (2) to (1);
				\draw[kernels,->] (4) to (1);
			\end{tikzpicture} = \begin{tikzpicture}[baseline=0.6cm,scale=0.35]
				\node at (0,0) [root] (0) {};
				\node at (0,3) [dot] (1) {};
				\node at (-2,1.5) [dot] (2) {};

				\node at (0,6) [dot] (4) {};
				
				\draw[kernels,->] (1) to (0);
				\draw[kernels,->] (2) to (0);
				\draw[kernels,->] (2) to (1);
				\draw[kernels,->] (4) to (1);
			\end{tikzpicture} \, \, \, \begin{tikzpicture}[baseline=0.6cm,scale=0.35]
				\node at (0,1) [root] (0) {};
				\node at (0,4) [dot] (1) {};
				\draw[decorate,shorten >=2pt,shorten <=2pt, decoration={zigzag, segment length=4pt, amplitude=1pt, post length = 5pt, pre length = 2pt},->] (1) to[bend right = 60] (0);
				\draw[decorate,shorten >=2pt,shorten <=2pt, decoration={zigzag, segment length=4pt, amplitude=1pt, post length = 5pt, pre length = 2pt},->] (1) to[bend left = 60] (0);
			\end{tikzpicture} = 0.
		\end{equation*}
		The last equality is obtained by replacing the heat kernel $K$ with a version that cancels out monomials   up to a certain degree by convolution. Indeed,
		one has
		\begin{equation*}
			\begin{tikzpicture}[baseline=0.6cm,scale=0.35]
				\node at (0,1) [root] (0) {};
				\node at (0,4) [dot] (1) {};
				\draw[decorate,shorten >=2pt,shorten <=2pt, decoration={zigzag, segment length=4pt, amplitude=1pt, post length = 5pt, pre length = 2pt},->] (1) to (0);
			\end{tikzpicture} = \int K(-z) dz = 0.
		\end{equation*}
		As a second example, one has 
		\begin{equation*}
			\begin{aligned}
				\mathbb{E}( 	(\bm{\Pi}^{\varepsilon} \iiIiXiiii)	(0)) & = 2 ( \bm{\Pi}^{\varepsilon}  \iiIiXiiiibb )(0)  +  ( \bm{\Pi}^{\varepsilon}  \iiIiXiiiiaa )(0) = 2 \,	\begin{tikzpicture}[baseline=0.6cm,scale=0.35]
					\node at (0,0) [root] (0) {};
					\node at (2,2) [dot] (1) {};
					\node at (-2,2) [dot] (2) {};
					\node at (0,4) [dot] (3) {};
					\node at (0,2) [dot] (4) {};
					\draw[kernels,->] (1) to (0);
					\draw[kernels,->] (2) to (0);
					\draw[kernels,->] (3) to (1);
					\draw[kernels,->] (3) to (2);
					\draw[kernels,->] (4) to (2);
					\draw[kernels,->] (4) to (1);
				\end{tikzpicture} + 	\begin{tikzpicture}[baseline=0.6cm,scale=0.35]
					\node at (0,0) [root] (0) {};
					\node at (2,2) [dot] (1) {};
					\node at (-2,2) [dot] (2) {};
					\node at (2,5) [dot] (3) {};
					\node at (-2,5) [dot] (4) {};
					\draw[kernels,->] (1) to (0);
					\draw[kernels,->] (2) to (0);
					\draw[decorate,shorten >=2pt,shorten <=2pt, decoration={zigzag, segment length=4pt, amplitude=1pt, post length = 5pt, pre length = 2pt},->] (3) to[bend right = 60] (1);
					\draw[decorate,shorten >=2pt,shorten <=2pt, decoration={zigzag, segment length=4pt, amplitude=1pt, post length = 5pt, pre length = 2pt},->] (3) to[bend left = 60] (1);
					\draw[decorate,shorten >=2pt,shorten <=2pt, decoration={zigzag, segment length=4pt, amplitude=1pt, post length = 5pt, pre length = 2pt},->] (4) to[bend right = 60] (2);
					\draw[decorate,shorten >=2pt,shorten <=2pt, decoration={zigzag, segment length=4pt, amplitude=1pt, post length = 5pt, pre length = 2pt},->] (4) to[bend left = 60] (2);
				\end{tikzpicture}.
			\end{aligned}
		\end{equation*}
		For the same reason as before, the second  Feynman diagram above is zero. In the next proposition, we state the cancellation between the two renormalisation constants.
		\begin{proposition}
			One has
			\begin{equation*} \label{cancellation_1}
				2 \,	\mathbb{E}((\bm{\Pi}^{\varepsilon} \Xitwo )(0)) + \mathbb{E}( 	(\bm{\Pi}^{\varepsilon} \iiIiXiiii)	(0)) \approx 0.
			\end{equation*}
		\end{proposition}
		\begin{proof}
			One starts to notice that
			\begin{equation*}
				\begin{tikzpicture}[baseline=0.6cm,scale=0.35]
					\node at (0,0) [root] (0) {};
					\node at (0,3) [dot] (1) {};
					\node at (-2,3) [dot] (2) {};
					\node at (2,4.5) [dot] (3) {};
					
					\node at (0,6) [dot] (4) {};
					
					\draw[kernels,->] (1) to (0);
					\draw[kernels,->] (2) to (0);
					\draw[kernels,->] (3) to (4);
					\draw[kernels,->] (3) to (1);
					\draw[kernels,->] (2) to (4);
					\draw[kernels,->] (4) to (1);
				\end{tikzpicture} \,  \approx \, \frac{1}{4}
				\begin{tikzpicture}[baseline=0.6cm,scale=0.35]
					\node at (0,0) [root] (0) {};
					\node at (0,3) [dot] (1) {};
					\node at (0,6) [dot] (4) {};
					\draw[kernel,->] (4) to[bend right=60] (0);
					\draw[kernel,->] (4) to[bend left=60] (1);
					\draw[kernels,->] (1) to (0);
					\draw[kernels,->] (4) to (1);
				\end{tikzpicture} \, + \, \frac{1}{4}
				\begin{tikzpicture}[baseline=0.6cm,scale=0.35]
					\node at (0,0) [root] (0) {};
					\node at (0,3) [dot] (1) {};
					\node at (0,6) [dot] (4) {};
					\draw[kernel,<-] (4) to[bend right=60] (0);
					\draw[kernel,<-] (4) to[bend left=60] (1);
					\draw[kernels,->] (1) to (0);
					\draw[kernels,->] (4) to (1);
				\end{tikzpicture} \, + \, \frac{1}{4} \begin{tikzpicture}[baseline=0.6cm,scale=0.35]
					\node at (0,0) [root] (0) {};
					\node at (0,3) [dot] (1) {};
					\node at (0,6) [dot] (4) {};
					\draw[kernel,<-] (4) to[bend right=60] (0);
					\draw[kernel,->] (4) to[bend left=60] (1);
					\draw[kernels,->] (1) to (0);
					\draw[kernels,->] (4) to (1);
				\end{tikzpicture} \, + \, \frac{1}{4} \begin{tikzpicture}[baseline=0.6cm,scale=0.35]
					\node at (0,0) [root] (0) {};
					\node at (0,3) [dot] (1) {};
					\node at (0,6) [dot] (4) {};
					\draw[kernel,->] (4) to[bend right=60] (0);
					\draw[kernel,<-] (4) to[bend left=60] (1);
					\draw[kernels,->] (1) to (0);
					\draw[kernels,->] (4) to (1);
				\end{tikzpicture}. 
			\end{equation*}
			Then, by an integration by parts, one gets
			\begin{equation*}
				\begin{tikzpicture}[baseline=0.6cm,scale=0.35]
					\node at (0,0) [root] (0) {};
					\node at (0,3) [dot] (1) {};
					\node at (0,6) [dot] (4) {};
					\draw[kernel,->] (4) to[bend right=60] (0);
					\draw[kernel,->] (4) to[bend left=60] (1);
					\draw[kernels,->] (1) to (0);
					\draw[kernels,->] (4) to (1);
				\end{tikzpicture}  \, = \,  - \frac{1}{2} 	\begin{tikzpicture}[baseline=0.6cm,scale=0.35]
					\node at (0,0) [root] (0) {};
					\node at (0,3) [dot] (1) {};
					\node at (0,6) [dot] (4) {};
					\draw[decorate,shorten >=2pt,shorten <=2pt, decoration={zigzag, segment length=4pt, amplitude=1pt, post length = 5pt, pre length = 2pt},->] (4) to[bend right = 60] (0);
					\draw[kernels,->] (4) to[bend left=60] (1);
					\draw[kernels,->] (1) to (0);
					\draw[kernel,->] (4) to (1);
				\end{tikzpicture}  \, =  \,  - \frac{1}{2} 	\begin{tikzpicture}[baseline=0.6cm,scale=0.35]
					\node at (0,0) [dot] (0) {};
					\node at (0,3) [root] (1) {};
					\node at (0,6) [dot] (4) {};
					\draw[decorate,shorten >=2pt,shorten <=2pt, decoration={zigzag, segment length=4pt, amplitude=1pt, post length = 5pt, pre length = 2pt},->] (0) to[bend left = 60] (4);
					\draw[kernels,->] (4) to[bend left=60] (1);
					\draw[kernels,->] (0) to (1);
					\draw[kernel,->] (4) to (1); 
				\end{tikzpicture} 
				\, \approx  \,  - \frac{1}{4} \,	\begin{tikzpicture}[baseline=0.6cm,scale=0.35]
					\node at (0,1) [root] (0) {};
					\node at (0,4) [dot] (1) {};
					\draw[kernel,->] (1) to[bend right=60] (0);
					\draw[kernel,->] (1) to[bend left=60] (0);
					\draw[kernel,->] (1) to (0);
				\end{tikzpicture}   - \frac{1}{4} \,	\begin{tikzpicture}[baseline=0.6cm,scale=0.35]
					\node at (0,1) [root] (0) {};
					\node at (0,4) [dot] (1) {};
					\draw[kernel,->] (1) to[bend left=60] (0);
					\draw[kernel,->] (0) to[bend left=60] (1);
					\draw[kernel,->] (1) to (0);
				\end{tikzpicture}.
			\end{equation*}	
			From the previous two computations, one can neglect all the terms with a direct loop, meaning that one can follow a path within the Feynman diagram following the orientation of the edges and then come back to the node one started with. This is due to the fact that the heat kernel is non-anticipative (zero for negative time). The times in the integration are ordered according to the orientation of the edges. This idea was first used in \cite{wong} for computing renormalisation constants.
			Therefore
			\begin{equation*}
				\mathbb{E}((\bm{\Pi}^{\varepsilon} \Xitwo )(0))	\approx	- \frac{1}{8} \,	\begin{tikzpicture}[baseline=0.6cm,scale=0.35]
					\node at (0,1) [root] (0) {};
					\node at (0,4) [dot] (1) {};
					\draw[kernel,->] (1) to[bend right=60] (0);
					\draw[kernel,->] (1) to[bend left=60] (0);
					\draw[kernel,->] (1) to (0);
				\end{tikzpicture}.
			\end{equation*}
			One the other hand
			\begin{equation*}
				\begin{tikzpicture}[baseline=0.6cm,scale=0.35]
					\node at (0,0) [root] (0) {};
					\node at (2,2) [dot] (1) {};
					\node at (-2,2) [dot] (2) {};
					\node at (0,4) [dot] (3) {};
					\node at (0,2) [dot] (4) {};
					\draw[kernels,->] (1) to (0);
					\draw[kernels,->] (2) to (0);
					\draw[kernels,->] (3) to (1);
					\draw[kernels,->] (3) to (2);
					\draw[kernels,->] (4) to (2);
					\draw[kernels,->] (4) to (1);
				\end{tikzpicture} \, = \, 
				\begin{tikzpicture}[baseline=0.6cm,scale=0.35]
					\node at (0,0) [dot] (0) {};
					\node at (2,2) [root] (1) {};
					\node at (-2,2) [dot] (2) {};
					\node at (0,4) [dot] (3) {};
					\node at (0,2) [dot] (4) {};
					\draw[kernels,->] (0) to (1);
					\draw[kernels,->] (0) to (2);
					\draw[kernels,->] (3) to (1);
					\draw[kernels,->] (3) to (2);
					\draw[kernels,->] (4) to (2);
					\draw[kernels,->] (4) to (1);
				\end{tikzpicture}
				\, \approx  \,   \frac{1}{8} \,	\begin{tikzpicture}[baseline=0.6cm,scale=0.35]
					\node at (0,1) [root] (0) {};
					\node at (0,4) [dot] (1) {};
					\draw[kernel,->] (1) to[bend right=60] (0);
					\draw[kernel,->] (1) to[bend left=60] (0);
					\draw[kernel,->] (1) to (0);
				\end{tikzpicture}   + \frac{3}{8} \,	\begin{tikzpicture}[baseline=0.6cm,scale=0.35]
					\node at (0,1) [root] (0) {};
					\node at (0,4) [dot] (1) {};
					\draw[kernel,->] (1) to[bend left=60] (0);
					\draw[kernel,->] (0) to[bend left=60] (1);
					\draw[kernel,->] (1) to (0);
				\end{tikzpicture}.
			\end{equation*}
			For the same reason as before, one can remove the diagram with a loop and get
			\begin{equation*}
				\mathbb{E}( 	(\bm{\Pi}^{\varepsilon} \iiIiXiiii)	(0))  \approx \frac{1}{4} \,	\begin{tikzpicture}[baseline=0.6cm,scale=0.35]
					\node at (0,1) [root] (0) {};
					\node at (0,4) [dot] (1) {};
					\draw[kernel,->] (1) to[bend right=60] (0);
					\draw[kernel,->] (1) to[bend left=60] (0);
					\draw[kernel,->] (1) to (0);
				\end{tikzpicture}. 
			\end{equation*}
			which allows us to conclude.
		\end{proof}

		The BPHZ choice of renormalisation given in \cite{BHZ} corresponds to change $ \bm{\Pi}^{\varepsilon} $ into
		$\hat{\bm{\Pi}}^{\varepsilon}$ such that for every decorated tree $\tau$ with negative degree, one has
		\begin{equation*}
			\mathbb{E}( (\hat{\bm{\Pi}}^{\varepsilon} \tau)(0) ) = 0.
		\end{equation*}

		Cancellations for singular SPDEs were first observed for the KPZ equation in \cite{Hai13} where \eqref{cancellation_1} was computed in the Fourier space. The techniques presented in this physical space proof come from \cite{wong,HQ18}. The main result on the chain rule for space-time white noise was first obtained by computing many cancellations with the same flavour as above (see \cite{Bru}). More general integration by parts identities are provided in \cite{Mate19} and were used to remove non-local counter-terms for the solution theory provided in \cite{MH} for a simple class of quasi-linear equations. One can find similar identities in the context of dispersive PDEs with random initial data (see \cite{BDNY,DH23,BT26}).
		
		\section{Characterisation of the chain rule symmetry }
		
		In this section, we make precise the notion of chain rule symmetry and we characterise this symmetry by providing a basis constructed from iterated covariant derivatives. We focus mainly on multi-indices as they involve only elementary proofs. These proofs were first performed in \cite{CBB}. The case of decorated trees requires more advanced algebra and the use of homological tools. We sketch the main ideas of the proof at the end of the section.  
		Setting $u = \phi^{-1}(v)$ and applying the chain rule, we find that $v$ satisfies
		\[
		\partial_t v =
		\partial_x^{2} v +
		(\phi\cdot f)(v)\,(\partial_x v)^{2} + (\phi\cdot h)(v)
		+ (\phi\cdot g)(v)\,\xi ,
		\]
		where the new coefficients are given by
		\begin{equation*} 
			\begin{aligned}
			&(\phi \cdot g)(\phi(u)) = \phi'(u)\, g(u), \quad (\phi \cdot h)(\phi(u)) = \phi'(u)\, h(u), \\ 
			\label{eq:phi_action_Gamma}
			&(\phi \cdot f)(\phi(u))\,(\phi'(u))^2 = \phi'(u)\, f(u) - \phi''(u).
			\end{aligned}
		\end{equation*}
		The term $h$ corresponds to the counter-terms given by the renormalisation of the stochastic data. Moving this renormalisation from the stochastic data to the equation is a non-trivial task that was  first solved for decorated trees in \cite{BCCH}. A simple and more general proof is given in  \cite{BB21b} based on the preparation maps introduced in \cite{BR18}. The proof in the context of multi-indices is given in \cite{BL23} and follows the ideas of \cite{BB21b}. This is also the case for the paracontrolled ansatz given in \cite{BM26}.
		From these general results, the renormalised equation is given by 
		\begin{equation}\label{eq:renorm_constant3}
			\begin{aligned}
				\partial_t u_{\varepsilon}  &= \partial_x^2 u_{ \varepsilon} + f(u_{\varepsilon})\,(\partial_x u_{\varepsilon})^2 +
				g(u_{\varepsilon})\, \xi_{\varepsilon} + \sum_{z^\beta\in \mathfrak{M}_{ \derivatives, \generic}^-} C_{\varepsilon}(z^\beta)F_{f,g}(z^{\beta})(u_{\varepsilon})
			\end{aligned}
		\end{equation}
		where $\mathfrak{M}_{ \derivatives, \generic}^- \subset \mathfrak{M}_{ \derivatives, \generic}$ consists of elements with negative degrees and the $C_{\varepsilon}(z^\beta)$ are the renormalisation constants associated with the $z^{\beta}$ which arises from the choice of renormalisation performed on the stochastic data. One observes that the counter-terms depend on $ f,g $ via the elementary differentials $ F_{f,g}$.
		By applying the previous transformation rule to elements in the image of $F_{f,g}$, we obtain a rigorous symmetry transformation of our system. Below, we provide the precise definition of the chain rule symmetry
		\begin{definition} \label{def_symmetries}
			We define the subspace $V_{\geo}\subset \mathcal{M}_{ \derivatives, \generic}$ as the subspace generated by those elements $v\in \mathcal{M}_{ \derivatives, \generic}$ such that for all choices of $f$, $g$ and all diffeomorphisms $\phi \colon \mathbb{R} \to \mathbb{R}$ homotopic to the identity, one has 
			\begin{equation*}
				\varphi\cdot F_{f  ,g}(v) = F_{\varphi \cdot f \, \varphi \cdot g}(v)\,.
			\end{equation*}
		\end{definition}
		The previous characterisation is not easy to manipulate. One wants to find a natural basis and compute its dimension. Therefore, a more combinatorial characterisation is required. We first introduce natural derivations and the Lie bracket on multi-indices that will allow us to explicitly state this combinatorial characterisation.
		At the level of multi-indices we define $ [\cdot, \cdot] \colon \mathcal{M}_{\begin{tikzpicture}[scale=0.2,baseline=-2]
				\coordinate (root) at (0,0);
				\node[diff] (rootnode) at (root) {};
			\end{tikzpicture}, \derivatives, \generic} \times \mathcal{M}_{\begin{tikzpicture}[scale=0.2,baseline=-2]
				\coordinate (root) at (0,0);
				\node[diff] (rootnode) at (root) {};
			\end{tikzpicture}, \derivatives, \generic} \to \mathcal{M}_{\begin{tikzpicture}[scale=0.2,baseline=-2]
				\coordinate (root) at (0,0);
				\node[diff] (rootnode) at (root) {};
			\end{tikzpicture}, \derivatives, \generic}  $ as
		\begin{equation*}
			\,	 [v_1, v_2] = v_1 D v_2 - v_2 D v_1
		\end{equation*}
		where $D\colon \mathcal{M}_{\begin{tikzpicture}[scale=0.2,baseline=-2]
				\coordinate (root) at (0,0);
				\node[diff] (rootnode) at (root) {};
			\end{tikzpicture}, \derivatives, \generic} \to \mathcal{M}_{\begin{tikzpicture}[scale=0.2,baseline=-2]
				\coordinate (root) at (0,0);
				\node[diff] (rootnode) at (root) {};
			\end{tikzpicture}, \derivatives, \generic} $ is the derivation on multi-indices given by
		\begin{equation*}
			D   	z_{ (\begin{tikzpicture}[scale=0.2,baseline=-2]
					\coordinate (root) at (0,0);
					\node[diff] (rootnode) at (root) {};
				\end{tikzpicture},k)}   =  	z_{ (\begin{tikzpicture}[scale=0.2,baseline=-2]
					\coordinate (root) at (0,0);
					\node[diff] (rootnode) at (root) {};
				\end{tikzpicture},k+1)}, \quad D  z_{( \derivatives,k)}    =  z_{( \derivatives,k+1)}, \quad D	z_{ (\generic,k)}  = 	z_{ (\generic,k+1)}
		\end{equation*}
		and then extended to $\mathcal{M}_{\begin{tikzpicture}[scale=0.2,baseline=-2]
				\coordinate (root) at (0,0);
				\node[diff] (rootnode) at (root) {};
			\end{tikzpicture}, \derivatives, \generic} $ via the Leibniz rule. Below, we define a geometric linear map denoted by $\phi_\geo$. It is given in terms of the Lie bracket $[\cdot,\cdot]$ and the derivation $D$.

		\begin{definition} \label{def:injectivity_Upsilon}
			We first define the geometric map $\phi_\geo : \mathcal{M}_{ \derivatives, \generic} \rightarrow \mathcal{M}_{\begin{tikzpicture}[scale=0.2,baseline=-2]
					\coordinate (root) at (0,0);
					\node[diff] (rootnode) at (root) {};
				\end{tikzpicture}, \derivatives, \generic}$  by setting 
			\begin{equation*}
				\begin{aligned}
					\phi_\geo(	z_{ (\generic,k)})  &= D^{k}  [	z_{ (\generic,0 )}, z_{ (\begin{tikzpicture}[scale=0.2,baseline=-2]
							\coordinate (root) at (0,0);
							\node[diff] (rootnode) at (root) {};
						\end{tikzpicture},0)}] \\
					\phi_\geo( z_{( \derivatives,k)} ) & = - D^{k+1} \left( 	z_{( \derivatives,0)} z_{ (\begin{tikzpicture}[scale=0.2,baseline=-2]
							\coordinate (root) at (0,0);
							\node[diff] (rootnode) at (root) {};
						\end{tikzpicture},0)} 
					\right) - 2  z_{ (\begin{tikzpicture}[scale=0.2,baseline=-2]
							\coordinate (root) at (0,0);
							\node[diff] (rootnode) at (root) {};
						\end{tikzpicture},k+2) }
				\end{aligned}
			\end{equation*}
			and then it is extended via the Leibniz rule.
			We define the geometric map $ \hat{\phi}_\geo \colon  \mathcal{M}_{ \derivatives, \generic} \rightarrow \mathcal{M}_{\begin{tikzpicture}[scale=0.2,baseline=-2]
					\coordinate (root) at (0,0);
					\node[diff] (rootnode) at (root) {};
				\end{tikzpicture}, \derivatives, \generic} $ by simply putting 
			\begin{equation*}
				\hat{\phi}_\geo(z^{\beta}) = \phi_\geo(z^{\beta}) - [ z^{\beta}, z_{ (\begin{tikzpicture}[scale=0.2,baseline=-2]
						\coordinate (root) at (0,0);
						\node[diff] (rootnode) at (root) {};
					\end{tikzpicture},0)} ]\,.
			\end{equation*}
		\end{definition}
		From the definition of $\phi_\geo$ is easy to see that one has
		\begin{equation}
			\label{commutation_derivative}
			\phi_\geo(D z^{\beta}) = D \phi_\geo( z^{\beta}).
		\end{equation}
		In the next Theorem, one gets a characterisation  of $ V_{\geo} $ as the kernel of $ \hat{\varphi}_{\geo} $. 
		
		\begin{theorem} \label{geo_chain_rule} One has $v \in	V_{\geo} $ if and only if $ \hat{\varphi}_{\geo}(v) = 0$.
		\end{theorem}
		\begin{proof}
			We only prove the if part. The proof for the only if is given in \cite[Proposition 6.2]{BGHZ}. We consider a family of maps $ (\psi_{t})_{t \geq 0}  $ with $ \psi_0 = \id, \partial_t \psi |_{t=0} = h $. Let $ v \in V_{\geo} $, one has from Definition \ref{def_symmetries}
			\begin{equation*}
				\psi_t \cdot F_{ f, g}(v) = F_{\psi_t \cdot f, \psi_t \cdot g}(v).
			\end{equation*}
			Then, one takes the derivative at time $t=0$ on both sides of the previous equality:
			\begin{equation*}
				\partial_t	\left( \psi_t \cdot F_{ f, g}(v) \right)|_{t=0} =	\partial_t	\left( F_{\psi_t \cdot f, \psi_t \cdot g}(v)\right)|_{t=0}.
			\end{equation*}
			In what follows, we compute both sides of the identity separately. Many simplifications occur, as evaluated at $t=0$, most terms $ \psi_t $ reduce to the identity map and $ \psi'_t $ reduces to $h$.
			One first has
			\begin{equation*}
				\begin{aligned}
					\partial_t	\left(	\psi_t \cdot F^h_{ f, g}(v) \right) |_{t=0} & = \partial_t	\left(	(\psi_t' \, F^h_{ f, g}(v)) \circ \psi_t^{-1} \right)|_{t=0} \\
					& = 	h'  F_{ f, g}(v) - \partial_u F_{ f, g}(v) h 
					\\ & =F^h_{f,g}\left( [v, z_{ (\begin{tikzpicture}[scale=0.2,baseline=-2]
							\coordinate (root) at (0,0);
							\node[diff] (rootnode) at (root) {};
						\end{tikzpicture},0)} ] \right).
				\end{aligned}
			\end{equation*}
			Then, on the other hand, one has
			\begin{equation*}
				\begin{aligned}
					\partial_t F_{\psi_t \cdot f, \psi_t \cdot g}(  z_{ (\generic,k)} )|_{t=0} &= 	\partial_t \left(  (\psi_t' \, g) \; \circ \psi_t^{(-1)} \right)^{(k)}  |_{t=0} 
					\\ &= \left(\partial_t \left(  (\psi_t' \, g) \; \circ \psi_t^{(-1)} \right) |_{t=0} \right)^{(k)}
					\\ & =  \left( h' g - h g'  \right)^{(k)}  
					\\ &= F^{h}_{f,g} \left(  [	z_{ (\generic,0 )}, z_{ (\begin{tikzpicture}[scale=0.2,baseline=-2]
							\coordinate (root) at (0,0);
							\node[diff] (rootnode) at (root) {};
						\end{tikzpicture},0)}] \right)^{(k)}
					\\ & = F^{h}_{f,g} \left(D^{k}  [	z_{ (\generic,0 )}, z_{ (\begin{tikzpicture}[scale=0.2,baseline=-2]
							\coordinate (root) at (0,0);
							\node[diff] (rootnode) at (root) {};
						\end{tikzpicture},0)}] \right)
					\\ & = F^{h}_{f,g} [	\phi_\geo(	z_{ (\generic,k)}) ].
				\end{aligned}
			\end{equation*}
			One has also
			\begin{equation*}
				\begin{aligned}
					\partial_t F_{\psi_t \cdot f, \psi_t \cdot g}(   z_{(\derivatives,k)} )|_{t=0} & =  2 \	\partial_t \left( 
					\frac{	\psi_t' \, f - \psi_t''}{(\psi_t')^2} \circ \psi_t^{-1} \right)^{(k)}|_{t=0}
					\\ & = 2 \left( \partial_t \left( 
					\frac{	\psi_t' \, f - \psi_t''}{(\psi_t')^2} \circ \psi_t^{-1} \right)|_{t=0} \right)^{(k)}
					\\ &= 2 \left(  h' f - h f' -  h'' - 2 h'f \right)^{(k)}
					\\ & =  F^{h}_{f,g}  \left(  - 	z_{( \derivatives,0)} z_{ (\begin{tikzpicture}[scale=0.2,baseline=-2]
							\coordinate (root) at (0,0);
							\node[diff] (rootnode) at (root) {};
						\end{tikzpicture},1)} - z_{( \derivatives,1)}
					z_{ (\begin{tikzpicture}[scale=0.2,baseline=-2]
							\coordinate (root) at (0,0);
							\node[diff] (rootnode) at (root) {};
						\end{tikzpicture},0)}  - 2  z_{ (\begin{tikzpicture}[scale=0.2,baseline=-2]
							\coordinate (root) at (0,0);
							\node[diff] (rootnode) at (root) {};
						\end{tikzpicture},2)}
					\right)^{(k)}
					\\ & =   F^{h}_{f,g}  \left(  - D^{k+1}\left(	z_{( \derivatives,0)} z_{ (\begin{tikzpicture}[scale=0.2,baseline=-2]
							\coordinate (root) at (0,0);
							\node[diff] (rootnode) at (root) {};
						\end{tikzpicture},0)}  \right)  - 2  z_{ (\begin{tikzpicture}[scale=0.2,baseline=-2]
							\coordinate (root) at (0,0);
							\node[diff] (rootnode) at (root) {};
						\end{tikzpicture},k+2)}
					\right)
					\\ & = F^{h}_{f,g}  \left(  \phi_\geo( z_{( \derivatives,k)} ) \right).
				\end{aligned}
			\end{equation*}
			In the end, gathering the various terms, one gets for all  $f,g,h \in \mathcal{C}^{\infty}(\mathbb{R}, \mathbb{R})$
			\begin{equation*}
				F^{h}_{f,g} \left( \phi_\geo(v) \right)
				=  F^{h}_{f,g} \left([ v, z_{ (\begin{tikzpicture}[scale=0.2,baseline=-2]
						\coordinate (root) at (0,0);
						\node[diff] (rootnode) at (root) {};
					\end{tikzpicture},0)} ] \right).
			\end{equation*}
			Hence we conclude from the injectivity of the map $ F_{f,g}^{h} $ in Theorem~\ref{injectivity_Upsilon}.
		\end{proof}
		One notices that the injectivity of the elementary differentials is crucial for obtaining the combinatorial characterisation. When one considers the geometric KPZ equation for a sufficiently large system, decorated trees provide the correct combinatorial set to achieve the injectivity. From the negative results of \cite{BL25}, one does not expect to see any intermediate combinatorial set for small dimensions greater than one.

		After this characterisation, we provide a candidate for a generating set of $ V_{\geo} $. We start by recalling one natural geometric operation, namely the covariant derivatives given for a Christoffel symbol $ \Gamma_{\beta, \gamma}^{\alpha} $ and smooth vector fields $ \sigma_{i}^{\alpha} $  by
		\begin{equation*}
			\nabla_{\sigma_i} \sigma_j = \sigma_i^{\beta} \partial_{\beta} \sigma_j + \Gamma_{\beta \gamma} \sigma_i^\beta \sigma_j^\gamma.
		\end{equation*}
		In the one-dimensional case, one gets
		\begin{equation*}
			\nabla_{g} g = g g' + f g^2.
		\end{equation*}
		Inspired by this definition, we define combinatorial covariant derivatives on multi-indices given below.
		For any  $v_1, v_2\in\mathcal {M}_{ \derivatives, \generic}$	we define the covariant derivative $\nabla_{v_1} v_2\in \mathcal {M}_{ \derivatives, \generic}$  as  follows
		\begin{equation} \label{covariant_derivative}
			\nabla_{v_1} v_2 = v_1 D v_2 + \frac{1}{2}  z_{( \derivatives,0)} v_1 v_2.
		\end{equation}
		One can easily check that $\nabla$ is a well-defined map $\nabla\colon \mathcal {M}_{ \derivatives, \generic} \times \mathcal {M}_{ \derivatives, \generic} \to \mathcal {M}_{ \derivatives, \generic} $. 
		One has the following identity
		\begin{proposition}
			For any $v_1, v_2\in\mathcal {M}_{ \derivatives, \generic}$, one has 
			\begin{equation*}
				F_{f,g}( \nabla_{v_1} v_2 ) 
				= \nabla_{ 	F_{f,g}(v_1)} 	F_{f,g}(v_2).
			\end{equation*}
		\end{proposition}
		\begin{proof}
			One has by definition of the covariant combinatorial derivative \eqref{covariant_derivative}
			\begin{equation*}
				\begin{aligned}
					F_{f,g}( \nabla_{v_1} v_2 )  & =F_{f,g}( v_1 D v_2 )  + \frac{1}{2}  F_{f,g}(z_{( \derivatives,0)} v_1 v_2)
					\\ & = F_{f,g}( v_1)  F_{f,g}( D v_2 )  + \frac{1}{2}  F_{f,g}(z_{( \derivatives,0)})   F_{f,g}(v_1)  F_{f,g}( v_2)
					\\ & = F_{f,g}( v_1)  F_{f,g}'(  v_2 )  +   f   F_{f,g}(v_1)  F_{f,g}( v_2)
					\\ & =  \nabla_{ 	F_{f,g}(v_1)} 	F_{f,g}(v_2)
				\end{aligned}
			\end{equation*}
			where we have used the multiplicativity of $ F_{f,g} $ on the second line. On the third line, we have used the following two identities
			\begin{equation*}
				F_{f,g}( D v_2 ) = F_{f,g}'( v_2 ), \quad	\frac{1}{2}  F_{f,g}(z_{( \derivatives,0)})   = f.
			\end{equation*}
		\end{proof}
		
		In the following proposition, we show
		that the structure of $\hat{\phi}_\geo$ is preserved under covariant derivatives.  We provide a proof that uses only the combinatorial characterisation of $ V_{\geo} $.
		\begin{proposition} \label{geometric_Nabla} For every 	$v_1, v_2 \in	V_{\geo}$, one has
			$\nabla_{v_1} v_2  \in V_{\geo}$.
		\end{proposition}
		\begin{proof}
			One has to check that $ 	\nabla_{v_1} v_2 $ belongs to the kernel of $ \hat{\varphi}_{\geo} $. One first has
			\begin{equation*}
				\begin{aligned}
					\phi_\geo \left( \nabla_{v_1} v_2  \right)
					&  = \phi_\geo \left(  v_1 D v_2 + \frac{1}{2}  z_{( \derivatives,0)} v_1 v_2  \right)
					\\
					& = \phi_\geo \left(  v_1 \right) D v_2 + 
					v_1 D   \phi_\geo \left( v_2 \right) +  \frac{1}{2}   z_{( \derivatives,0)}  \phi_\geo  \left( v_1 \right) v_2 \\ &  + \frac{1}{2}   z_{( \derivatives,0)}   v_1  \phi_\geo  \left( v_2 \right)  \frac{1}{2} \phi_\geo  \left(  z_{( \derivatives,0)} \right) v_1 v_2
					\\
					& = \nabla_{ \phi_\geo \left(v_1\right)} v_2  +  \nabla_{v_1} \phi_\geo \left(v_2 \right) + \frac{1}{2} \phi_\geo  \left(  z_{( \derivatives,0)} \right) v_1 v_2
					\\ & = \nabla_{ [  v_1, z_{ (\begin{tikzpicture}[scale=0.2,baseline=-2]
								\coordinate (root) at (0,0);
								\node[diff] (rootnode) at (root) {};
							\end{tikzpicture},0)} ]} v_2  +  \nabla_{v_1} [  v_2, z_{ (\begin{tikzpicture}[scale=0.2,baseline=-2]
							\coordinate (root) at (0,0);
							\node[diff] (rootnode) at (root) {};
						\end{tikzpicture},0)} ]  + \frac{1}{2} \phi_\geo  \left(  z_{( \derivatives,0)} \right) v_1 v_2
				\end{aligned}
			\end{equation*}
			where we have used the fact that $ \phi_\geo $ is a derivation and \eqref{commutation_derivative} in the second line and the fact that  for $i \in \lbrace 1,2 \rbrace$ $\phi_\geo \left(v_i\right) =  [  v_i, z_{ (\begin{tikzpicture}[scale=0.2,baseline=-2]
					\coordinate (root) at (0,0);
					\node[diff] (rootnode) at (root) {};
				\end{tikzpicture},0)} ]$ in the last line.
			Then, continuing the computation we obtain
			\begin{equation*}
				\begin{aligned}
					\nabla_{ [  v_1, z_{ (\begin{tikzpicture}[scale=0.2,baseline=-2]
								\coordinate (root) at (0,0);
								\node[diff] (rootnode) at (root) {};
							\end{tikzpicture},0)} ]} v_2 & =  v_1 D z_{ (\begin{tikzpicture}[scale=0.2,baseline=-2]
							\coordinate (root) at (0,0);
							\node[diff] (rootnode) at (root) {};
						\end{tikzpicture},0)} D v_2 -  D v_1  z_{ (\begin{tikzpicture}[scale=0.2,baseline=-2]
							\coordinate (root) at (0,0);
							\node[diff] (rootnode) at (root) {};
						\end{tikzpicture},0)} D v_2 + \frac{1}{2}  z_{( \derivatives,0)} v_1 D z_{ (\begin{tikzpicture}[scale=0.2,baseline=-2]
							\coordinate (root) at (0,0);
							\node[diff] (rootnode) at (root) {};
						\end{tikzpicture},0)}  v_2
					\\ &  - \frac{1}{2}  z_{( \derivatives,0)} D v_1  z_{ (\begin{tikzpicture}[scale=0.2,baseline=-2]
							\coordinate (root) at (0,0);
							\node[diff] (rootnode) at (root) {};
						\end{tikzpicture},0)}  v_2
					\\   \nabla_{v_1} [  v_2, z_{ (\begin{tikzpicture}[scale=0.2,baseline=-2]
							\coordinate (root) at (0,0);
							\node[diff] (rootnode) at (root) {};
						\end{tikzpicture},0)} ] & = z_{ (\begin{tikzpicture}[scale=0.2,baseline=-2]
							\coordinate (root) at (0,0);
							\node[diff] (rootnode) at (root) {};
						\end{tikzpicture},2)}  v_1 v_2 + z_{ (\begin{tikzpicture}[scale=0.2,baseline=-2]
							\coordinate (root) at (0,0);
							\node[diff] (rootnode) at (root) {};
						\end{tikzpicture},1)}  v_1 D v_2 - z_{ (\begin{tikzpicture}[scale=0.2,baseline=-2]
							\coordinate (root) at (0,0);
							\node[diff] (rootnode) at (root) {};
						\end{tikzpicture},1)}  v_1 D v_2 - z_{ (\begin{tikzpicture}[scale=0.2,baseline=-2]
							\coordinate (root) at (0,0);
							\node[diff] (rootnode) at (root) {};
						\end{tikzpicture},0)}  Dv_1 Dv_2
					\\ & + \frac{1}{2}  z_{( \derivatives,0)} v_1 D z_{ (\begin{tikzpicture}[scale=0.2,baseline=-2]
							\coordinate (root) at (0,0);
							\node[diff] (rootnode) at (root) {};
						\end{tikzpicture},0)}  v_2 - \frac{1}{2}  z_{( \derivatives,0)} v_1  z_{ (\begin{tikzpicture}[scale=0.2,baseline=-2]
							\coordinate (root) at (0,0);
							\node[diff] (rootnode) at (root) {};
						\end{tikzpicture},0)}  D v_2
					\\  \frac{1}{2} \phi_\geo  \left(  z_{( \derivatives,0)} \right) v_1 v_2 & =   - \frac{1}{2} D \left( z_{( \derivatives,0)}  z_{ (\begin{tikzpicture}[scale=0.2,baseline=-2]
							\coordinate (root) at (0,0);
							\node[diff] (rootnode) at (root) {};
						\end{tikzpicture},0)} \right) v_1 v_2
					-    z_{ (\begin{tikzpicture}[scale=0.2,baseline=-2]
							\coordinate (root) at (0,0);
							\node[diff] (rootnode) at (root) {};
						\end{tikzpicture},2)}  v_1 v_2.
				\end{aligned}
			\end{equation*}
			On the other hand, one has
			\begin{equation*}
				\begin{aligned}
					\	[  \nabla_{v_1} v_2, z_{ (\begin{tikzpicture}[scale=0.2,baseline=-2]
							\coordinate (root) at (0,0);
							\node[diff] (rootnode) at (root) {};
						\end{tikzpicture},0)} ]
					& = v_1 D v_2 D z_{ (\begin{tikzpicture}[scale=0.2,baseline=-2]
							\coordinate (root) at (0,0);
							\node[diff] (rootnode) at (root) {};
						\end{tikzpicture},0)} + \frac{1}{2}  z_{( \derivatives,0)} v_1 v_2 D z_{ (\begin{tikzpicture}[scale=0.2,baseline=-2]
							\coordinate (root) at (0,0);
							\node[diff] (rootnode) at (root) {};
						\end{tikzpicture},0)}
					\\ & - D(v_1 D v_2)  z_{ (\begin{tikzpicture}[scale=0.2,baseline=-2]
							\coordinate (root) at (0,0);
							\node[diff] (rootnode) at (root) {};
						\end{tikzpicture},0)} + \frac{1}{2} D \left(  z_{( \derivatives,0)} v_1 v_2 \right)  z_{ (\begin{tikzpicture}[scale=0.2,baseline=-2]
							\coordinate (root) at (0,0);
							\node[diff] (rootnode) at (root) {};
						\end{tikzpicture},0)}.
				\end{aligned}
			\end{equation*}
			Combining the two sides we conclude  that
			$\phi_\geo \left( \nabla_{v_1} v_2  \right) = 	[  \nabla_{v_1} v_2, z_{ (\begin{tikzpicture}[scale=0.2,baseline=-2]
					\coordinate (root) at (0,0);
					\node[diff] (rootnode) at (root) {};
				\end{tikzpicture},0)} ]$.
		\end{proof}
		
		For any integer $N \ge 1$, we define $\mathfrak{B}_N$ to be the set obtained by iterating the covariant derivative $N-1$ times on $z_{(\generic,0)}$. We define this set recursively by $\mathfrak{B}_1=\{ z_{(\generic,0)}\}$ and for any $N\geq 1$ by
		\begin{equation}\label{eq:def_Bn}
			\mathfrak{B}_{N+1}= \{\nabla_{v} \, w\colon  v\in \mathfrak{B}_{k}, w\in \mathfrak{B}_{N+1-k}, \, 1\leq k\leq N\}\,.
		\end{equation}
		For instance, one has
		\begin{equation*}
			\begin{aligned}
				\mathfrak{B}_{2} &= \left \lbrace	 		\nabla_{\generic}  \generic\right \rbrace\,,\quad \mathfrak{B}_{3} = \left \lbrace	 	\nabla_{\generic} 	\nabla_{\generic}  \generic, \, 	\nabla_{\nabla_{\generic}  \generic }   \generic \right \rbrace\,,\\
				\mathfrak{B}_{4} &= \left \lbrace	\nabla_{\generic} 	\nabla_{\generic} 	\nabla_{\generic}  \generic, \, 	\nabla_{\generic} 	\nabla_{\nabla_{\generic}  \generic }   \generic, \,  	\nabla_{\nabla_{\generic}  \generic } 	\nabla_{\generic} 	  \generic , \,		\nabla_{\nabla_{\generic} \nabla_{\generic}  \generic }   \generic, \,	\nabla_{\nabla_{\nabla_{\generic} \generic}   \generic }   \generic \right \rbrace\,,
			\end{aligned}
		\end{equation*}
		where we have made the following abuse of notations replacing $ 	z_{ (\generic,0)} $ by $ \generic $. We denote by $V_{\geo}^N$ the subspace of $ V_{\geo} $ containing $N$ noises meaning $N$ variables of type $ z_{(\generic,k)} $.  By construction of $\mathfrak{B}_N$ and using Proposition \ref{geometric_Nabla}, all of its elements belong to $V_{\geo}^N$. One can  immediately see that the elements of $\mathfrak{B}_{N}$ do not constitute a linearly independent family in general. One has the following identities for elements of $\mathfrak{B}_4$
		\begin{equation*}\label{eq:extra_relations}
			\nabla_{\nabla_{\generic}  \generic } 	\nabla_{\generic} 	  \generic =	\nabla_{\nabla_{\generic} \nabla_{\generic}  \generic }   \generic\,, \, \quad
			2 \,	\nabla_{\nabla_{\generic} \nabla_{\generic}  \generic }   \generic - 	\nabla_{\generic} 	\nabla_{\nabla_{\generic}  \generic }   \generic   = 	\nabla_{\nabla_{\nabla_{\generic} \generic}   \generic }    \generic\,. 
		\end{equation*}
		In the sequel, we will use the following sets
		\begin{equation*}
			\begin{aligned}
				\mathfrak{M}_{\generic} & = \lbrace z^{\beta} = \prod_{k \geq 0}\left( z_{(\generic,k)}\right)^{\beta(\generic,k)} : 
				[\beta] =1 ,\; |\beta|\geq 2\rbrace, \quad \mathcal{M}_{\generic} = 	\langle \mathfrak{M}_{\generic }\rangle,  \\ \mathfrak{M}_{\generic}^N & = \lbrace z^{\beta}  \in \mathfrak{M}_{\generic}  : |\beta| =  N 
				\rbrace, \quad \mathcal{M}_{\generic}^N = 	\langle \mathfrak{M}_{\generic }^N\rangle.
			\end{aligned}
		\end{equation*}
		where $ |\beta| $ is the number of variables of type $ z_{(\generic,k)} $ that appear within $z^{\beta}$.
		
		\begin{proposition} \label{lower_bound_dimension}
			For any $N\geq 2$ the linear vector space generated by $\mathfrak{B}_N$ is a subspace of 	$V_{\geo}^N$ with dimension larger than $ \mathrm{Card}(\mathfrak{M}_{\generic}^N)$.
		\end{proposition}
		\begin{proof}
			One considers the canonical projection 
			$ \pi \colon \mathcal {M}_{ \derivatives, \generic}\to \mathcal{M}_{\generic} $ that sends variables $ z_{(\derivatives,k)} $ to zero. One has
			\begin{equation*}
				\pi	\nabla_{v_1} v_2 = \pi(v_1) D \pi(v_2).
			\end{equation*}
			Therefore, it is sufficient to show that one generates 	$\mathfrak{M}_{\generic}$ with $ z_{(\generic,0)} $ and $ D$. One observes from the definition of the map $\Psi$ that for each $v \in \mathfrak{M}_{\generic}$, there exists a decorated tree $ \tau  $ belonging to $ \mathcal{T}_{\generic} $ the subset of $\mathcal{T}_{\text{\tiny{Chr}}}$, that does not contain decorated trees with thick edges. The linear span of $\mathcal{T}_{\generic}$ is generated by $ \generic $ and the grafting product $ \curvearrowright $. We provide a brief proof of this fact below.
			The  grafting product $ \curvearrowright $ between two trees  $ \sigma $ and $ \tau $ is defined as
			\begin{equation*}
				\sigma \curvearrowright \tau = \sum_{u \in N_{\tau}} \sigma \curvearrowright_u \tau
			\end{equation*} 
			with $ N_{\tau} $ is the set of nodes of $\tau$ and $ \curvearrowright_u $ connects the root of $ \sigma $ to the node $u$ via a new thin edge. As an example, one has
			\begin{equation*}
				\Xii \curvearrowright	\Xii = \Xiiii + \Xiiiie.
			\end{equation*}
			One has 
			\begin{equation} \label{identity_graft}
				\begin{tikzpicture}[scale=0.4,baseline=-2]
					\coordinate (root) at (0,-0.4);
					\coordinate (t2) at (-0.8,0.5);
					\coordinate (t3) at (0.8,0.5);
					\draw[] (t2) -- (root);
					\draw[] (t3) -- (root);
					\draw (0,-0.4) node[xi,label= {[label distance=-0.2em]below: \tiny  $     $} ] {};
					\draw (-0.9,0.7) node[] {\tiny$\tau_1$};
					\draw (0.9,0.7) node[] {\tiny$\tau_n$};
					\draw (0,0.7) node[] {\tiny$\cdots$};
				\end{tikzpicture} =  {\scriptsize \tau_1} \curvearrowright \begin{tikzpicture}[scale=0.4,baseline=-2]
					\coordinate (root) at (0,-0.4);
					\coordinate (t2) at (-0.8,0.5);
					\coordinate (t3) at (0.8,0.5);
					\draw[] (t2) -- (root);
					\draw[] (t3) -- (root);
					\draw (0,-0.4) node[xi,label= {[label distance=-0.2em]below: \tiny  $     $} ] {};
					\draw (-0.9,0.7) node[] {\tiny$\tau_2$};
					\draw (0.9,0.7) node[] {\tiny$\tau_n$};
					\draw (0,0.7) node[] {\tiny$\cdots$};
				\end{tikzpicture} - \sum_{i=2}^n \begin{tikzpicture}[scale=0.4,baseline=-2]
					\coordinate (root) at (0,-0.4);
					\coordinate (t2) at (-1.4,0.5);
					\coordinate (t3) at (1.4,0.5);
					\coordinate (t4) at (0,0.5);
					\draw[] (t2) -- (root);
					\draw[] (t3) -- (root);
					\draw[] (t4) -- (root);
					\draw (0,-0.4) node[xi,label= {[label distance=-0.2em]below: \tiny  $     $} ] {};
					\draw (-1.4,0.7) node[] {\tiny$\tau_2$};
					\draw (1.4,0.7) node[] {\tiny$\tau_n$};
					\draw (0.75,0.7) node[] {\tiny$ \cdots$};
					\draw (-0.65,0.7) node[] {\tiny$ \cdots$};
					\draw (0,0.7) node[] {\tiny$\bar{\tau}_i$};
				\end{tikzpicture}, \quad \bar{\tau}_i = \tau_1 \curvearrowright \tau_i.
			\end{equation}
			We can then apply the induction hypothesis  to conclude on the fact that $\mathcal{T}_{\generic}$ is generated by $ \generic $ and the grafting product $ \curvearrowright $, since the trees  on the right-hand side have a smaller number of edges attached to the root.
			We are now in position  to complete the proof. Let $v \in \mathfrak{M}_{\generic}$, then there exist $ \tau_i, \sigma_i \in \mathcal{T}_{\generic} $ such that
			\begin{equation} \label{v_grafts}
				v = \Psi(\sum_{i=1}^n \sigma_i \curvearrowright \tau_i)  =   \sum_{i=1}^n \Psi(\sigma_i) D \Psi(\tau_i) 
			\end{equation}
			where we have used the fact that $ \Psi $ is a morphism that sends $ \curvearrowright $ to $D$. The latter statement can be proved by using the inductive definition of the grafting product given in \eqref{identity_graft}. We conclude by applying the induction hypothesis to $ \Psi(\sigma_i)$ and $ \Psi(\tau_i) $.
		\end{proof}
		One can use the projection $\pi$ to check that  the family  $ \tilde{\mathfrak{B}}_4 =   \left \lbrace	\nabla_{\generic} 	\nabla_{\generic} 	\nabla_{\generic}  \generic, \, 	\nabla_{\generic} 	\nabla_{\nabla_{\generic}  \generic }   \generic, \,
		\nabla_{\nabla_{\nabla_{\generic} \generic}   \generic }   \generic \right \rbrace $ is linearly independent. Indeed, one has
		\begin{equation*}
			\begin{aligned}
				\pi \nabla_{\generic} 	\nabla_{\generic} 	\nabla_{\generic}  \generic & = 
				\pi	\nabla_{\generic} 	\nabla_{\generic}  \left( z_{ (\generic,0)} z_{ (\generic,1)}  
				\right)
				\\ &= \pi	\nabla_{\generic} 	 \left( z_{ (\generic,0)}^2 z_{ (\generic,2)}  +  z_{ (\generic,0)} z_{ (\generic,1)}^2 \right) 
				\\ &= z_{ (\generic,0)}^3 z_{ (\generic,3)} + 4 z_{ (\generic,0)}^2 z_{ (\generic,1)} z_{ (\generic,2)} +  z_{ (\generic,0)} z_{ (\generic,1)}^3  \,, 
				\\
				\pi\nabla_{\nabla_{\nabla_{\generic} \generic}   \generic }    \generic = 	z_{ (\generic,1)} \pi \nabla_{\nabla_{\generic} \generic}   \generic  & =   z_{ (\generic,1)}^2 \pi \nabla_{\generic} \generic   = z_{ (\generic,0)}  z_{ (\generic,1)}^3\,,
				\\
				\pi	\nabla_{\generic} 	\nabla_{\nabla_{\generic}  \generic }   \generic = \pi	\nabla_{\generic} \left( z_{ (\generic,0)}  z_{ (\generic,1)}^2 \right) & = 2 z_{ (\generic,0)}^2 z_{ (\generic,1)} z_{ (\generic,2)}  + z_{ (\generic,0)}  z_{ (\generic,1)}^3\,.
			\end{aligned}
		\end{equation*}
		We conclude from the triangular structure in the terms $z_{ (\generic,0)}^3 z_{ (\generic,3)}, z_{ (\generic,0)}^2 z_{ (\generic,1)} z_{ (\generic,2)}, z_{ (\generic,0)}  z_{ (\generic,1)}^3$.
		In the following proposition, we show that one derives an upper bound that matches the lower bound.
		\begin{proposition} \label{upper_bound_dimension}
			For any $N\geq 2$ one has $\dim(V_{\geo}^N) \leq \mathrm{Card}(\mathfrak{M}_{\generic}^N)$.
		\end{proposition}
		
		Propositions \ref{lower_bound_dimension} and \ref{upper_bound_dimension} show that $ \mathfrak{B} $ generates $ V_{\geo}$ with an explicit dimension.
		
		\begin{theorem}  \label{dim_geo}
			For any $N\geq 2$ the linear vector space generated by $\mathfrak{B}_N$ coincides with	$V_{\geo}^N$ and $\dim(V_{\geo}^N) = \mathrm{Card}(\mathfrak{M}_{\generic}^N)$.
		\end{theorem}
		
		The full proof of Proposition \ref{upper_bound_dimension} is given in \cite[Proposition 3.13]{CBB}. We provide below the main idea and illustrate it for the case $N=4$. One considers the pre-image under $ \hat{\varphi}_{\geo} $ which boils down to looking at the pre-image of multi-indices of the form $  z_{ (\begin{tikzpicture}[scale=0.2,baseline=-2]
				\coordinate (root) at (0,0);
				\node[diff] (rootnode) at (root) {};
			\end{tikzpicture},k)} z^{\eta} $. From Definition \ref{def:injectivity_Upsilon}, a variable of the form $ z_{ (\begin{tikzpicture}[scale=0.2,baseline=-2]
				\coordinate (root) at (0,0);
				\node[diff] (rootnode) at (root) {};
			\end{tikzpicture},k)} $ comes either from  $z_{ (\generic,\ell)}$ or $ z_{(\derivatives,\ell)}$. Then, one builds a triangular system for which it is easy to compute a lower bound on the dimension of its space of its solutions.
		The triangularity comes from the fact that one gets only one leading term with a maximum number of variables $ z_{(\derivatives,k)} $ in each equation of the system. The multi-indices without such a factor are the free variables of the system. 
		One has for $ v \in V_{\geo}^4 $ 
		\begin{equation*}
			-2	 \,  \langle z_{( \derivatives,2)}   z_{(\generic,0)}^4, v  \rangle  +  \, \langle z_{(\generic,0)}^3z_{(\generic,3)}, v \rangle = 0.
		\end{equation*}
		where $ <\cdot,\cdot> $ is an inner product on multi-indices given by
		\begin{equation*}
			< z^{\beta} , z^{\bar{\beta}} > = \delta_{\beta, \bar{\beta}}.
		\end{equation*}
		One gets such an equation from the following computations:  
		\begin{equation*}
			\begin{aligned}
				\phi_\geo(	z_{ (\generic,3)})  &= D^{3}  [	z_{ (\generic,0 )}, z_{ (\begin{tikzpicture}[scale=0.2,baseline=-2]
						\coordinate (root) at (0,0);
						\node[diff] (rootnode) at (root) {};
					\end{tikzpicture},0)}]  = 	z_{ (\generic,0 )}  z_{ (\begin{tikzpicture}[scale=0.2,baseline=-2]
						\coordinate (root) at (0,0);
						\node[diff] (rootnode) at (root) {};
					\end{tikzpicture},4)} + \cdots \\
				\phi_\geo( z_{( \derivatives,2)} ) & = - D^{3} \left( 	z_{( \derivatives,0)} z_{ (\begin{tikzpicture}[scale=0.2,baseline=-2]
						\coordinate (root) at (0,0);
						\node[diff] (rootnode) at (root) {};
					\end{tikzpicture},0)} 
				\right) - 2  z_{ (\begin{tikzpicture}[scale=0.2,baseline=-2]
						\coordinate (root) at (0,0);
						\node[diff] (rootnode) at (root) {};
					\end{tikzpicture},4) } = - 2  z_{ (\begin{tikzpicture}[scale=0.2,baseline=-2]
						\coordinate (root) at (0,0);
						\node[diff] (rootnode) at (root) {};
					\end{tikzpicture},4) } + \cdots
			\end{aligned}
		\end{equation*}
		In this case, one looks at the pre-image of $  z_{ (\begin{tikzpicture}[scale=0.2,baseline=-2]
				\coordinate (root) at (0,0);
				\node[diff] (rootnode) at (root) {};
			\end{tikzpicture},4) } z^{\eta} $ with  $ z^{\eta} = z_{(\generic,0)}^4 $. One gets only two possibilities that correspond to the two multi-indices in the equation above. If one considers $ z_{ (\begin{tikzpicture}[scale=0.2,baseline=-2]
				\coordinate (root) at (0,0);
				\node[diff] (rootnode) at (root) {};
			\end{tikzpicture},3)} z^{\eta}$, then one has two possible choices for the cofactor $z^{\eta}=z_{ (\generic,0)}^4z_{( \derivatives,0)} $,  $z^{\eta}= z_{ (\generic,0)}^3 z_{ (\generic,1)}$. One gets
		\begin{equation*}
			\begin{aligned}
				-2 \	\langle z_{ (\generic,0)}^4z_{( \derivatives,0)}z_{( \derivatives,1)} , v \rangle - \langle z_{( \derivatives,2)}   z_{(\generic,0)}^4, v\rangle + \langle z_{ (\generic,0)}^3z_{ (\generic,1)}z_{( \derivatives,0)}, v\rangle &= 0\,,
				\\
				-2 \	\langle  z_{ (\generic,0)}^3 z_{ (\generic,1)}z_{( \derivatives,1)}, v \rangle +  2\langle z_{ (\generic,0)}^3z_{ (\generic,3)} ,v \rangle + \langle z_{ (\generic,0)}^2z_{ (\generic,1)}z_{ (\generic,2)}, v \rangle  &= 0\, .
			\end{aligned}
		\end{equation*}
		In the last case $ z_{ (\begin{tikzpicture}[scale=0.2,baseline=-2]
				\coordinate (root) at (0,0);
				\node[diff] (rootnode) at (root) {};
			\end{tikzpicture},2)} z^{\eta} $, one has four possible choices for the cofactor $z^{\eta}=  z_{ (\generic,0)}^4$,   $z^{\eta}=z_{(\generic,0)}^2 z_{(\generic,1)}^2$,  $z^{\eta}= z_{ (\generic,0)}^3z_{ (\generic,1)}$,  $z^{\eta}= z_{ (\generic,0)}^4z_{(\derivatives,1)}$,  which imply the four  equations 
		\begin{equation*}
			\begin{aligned}
				&	-6 \	\langle z_{( \derivatives,0)}^3   z_{(\generic,0)}^4, v \rangle + \ \langle z_{(\generic,0)}^3 z_{(\generic,1)}, v \rangle  = 0\,,
				\\&
				-2\	\langle z_{( \derivatives,0)}   z_{(\generic,0)}^2 z_{(\generic,1)}^2, v \rangle + 3\ \langle z_{(\generic,0)} z_{(\generic,1)}^3, v \rangle +\ \langle z_{(\generic,0)} z_{(\generic,1)}z_{(\generic,2)}, v \rangle = 0\,,\\&
				-2\langle z_{ (\generic,0)}^3z_{ (\generic,1)}z_{( \derivatives,0)}, v\rangle -4 \	\langle z_{( \derivatives,0)}^2   z_{(\generic,0)}^3 z_{(\generic,1)}, v \rangle + 2\ \langle z_{(\generic,0)}^2 z_{(\generic,1)}^2, v \rangle =0\,,\\&
				- 2\	\langle z_{ (\generic,0)}^4z_{(\derivatives,1)}z_{(\derivatives,0)} v \rangle + \langle z_{(\generic,0)}^3 z_{(\generic,1)}z_{(\derivatives,1)},v \rangle - 3\langle z_{( \derivatives,2)}   z_{(\generic,0)}^4,v \rangle = 0\,.
			\end{aligned}
		\end{equation*}
		One notices that the system is triangular with three free variables $ \ \langle z_{(\generic,0)}^3 z_{(\generic,1)}, v \rangle $, $ \langle z_{(\generic,0)} z_{(\generic,1)}^3, v \rangle $ and  $\ \langle z_{(\generic,0)} z_{(\generic,1)}z_{(\generic,2)}, v \rangle$.
		The proof of Proposition~\ref{upper_bound_dimension} is inspired by a proof by Franck Gabriel  that was part of a preliminary draft of \cite{BGHZ}. As with multi-indices, one can conduct the same analysis on the pre-images of the map $ \hat{\varphi}_{\geo} $. It is more involved as the definition of $  \hat{\varphi}_{\geo}$ on decorated trees is more complex. Unfortunately the system obtained is not triangular and one has to compute the orbits by hand. These orbits arise from the following observation:
		\begin{equation*}
			\begin{aligned}
				\hat{\phi}_\geo \left( \begin{tikzpicture}[scale=0.4,baseline=-2]
					\coordinate (root) at (0,-0.4);
					\coordinate (t1) at (-0.3,0.5);
					\coordinate (t2) at (-1.1,0.5);
					\coordinate (t3) at (1.1,0.5);
					\coordinate (t4) at (0.3,0.5);
					\coordinate (tau1) at (1,0.6);
					\draw[kernels2] (t1) -- (root);
					\draw[kernels2] (t2) -- (root);
					\draw[] (t3) -- (root);
					\draw[] (t4) -- (root);
					\node[not] (rootnode) at (root) {};
					\draw (-1.2,0.7) node[] {\tiny$\tau_1$};
					\node[not] (rootnode) at (root) {};
					\draw (1.3,0.7) node[] {\tiny$\tau_4$};
					\draw (-0.4,0.7) node[] {\tiny$\tau_{\tiny{2}}$};
					\draw (0.4,0.7) node[] {\tiny$\tau_3$};
				\end{tikzpicture}  \right)& = - 2 \, \begin{tikzpicture}[scale=0.4,baseline=-2]
					\coordinate (root) at (0,-0.4);
					\coordinate (t1) at (-0.3,0.5);
					\coordinate (t2) at (-1.1,0.5);
					\coordinate (t3) at (1.1,0.5);
					\coordinate (t4) at (0.3,0.5);
					\coordinate (tau1) at (1,0.6);
					\draw[] (t1) -- (root);
					\draw[] (t2) -- (root);
					\draw[] (t3) -- (root);
					\draw[] (t4) -- (root);
					\node[not] (rootnode) at (root) {};
					\draw (-1.2,0.7) node[] {\tiny$\tau_1$};
					\node[diff] (rootnode) at (root) {};
					\draw (1.3,0.7) node[] {\tiny$\tau_4$};
					\draw (-0.4,0.7) node[] {\tiny$\tau_{\tiny{2}}$};
					\draw (0.4,0.7) node[] {\tiny$\tau_3$};
				\end{tikzpicture} + (\cdots)
				\\
				\hat{\phi}_\geo \left( \begin{tikzpicture}[scale=0.4,baseline=-2]
					\coordinate (root) at (0,-0.4);
					\coordinate (t1) at (-0.3,0.5);
					\coordinate (t2) at (-1.1,0.5);
					\coordinate (t3) at (1.1,0.5);
					\coordinate (t4) at (0.3,0.5);
					\coordinate (tau1) at (1,0.6);
					\draw[kernels2] (t1) -- (root);
					\draw[kernels2] (t2) -- (root);
					\draw[] (t3) -- (root);
					\draw[] (t4) -- (root);
					\node[not] (rootnode) at (root) {};
					\draw (-1.2,0.7) node[] {\tiny$\tau_1$};
					\node[not] (rootnode) at (root) {};
					\draw (1.3,0.7) node[] {\tiny$\tau_4$};
					\draw (-0.4,0.7) node[] {\tiny$\tau_{\tiny{3}}$};
					\draw (0.4,0.7) node[] {\tiny$\tau_2$};
				\end{tikzpicture}  \right) & = - 2 \, \begin{tikzpicture}[scale=0.4,baseline=-2]
					\coordinate (root) at (0,-0.4);
					\coordinate (t1) at (-0.3,0.5);
					\coordinate (t2) at (-1.1,0.5);
					\coordinate (t3) at (1.1,0.5);
					\coordinate (t4) at (0.3,0.5);
					\coordinate (tau1) at (1,0.6);
					\draw[] (t1) -- (root);
					\draw[] (t2) -- (root);
					\draw[] (t3) -- (root);
					\draw[] (t4) -- (root);
					\node[not] (rootnode) at (root) {};
					\draw (-1.2,0.7) node[] {\tiny $\tau_1$};
					\node[diff] (rootnode) at (root) {};
					\draw (1.3,0.7) node[] {\tiny $\tau_4$};
					\draw (-0.4,0.7) node[] {\tiny $\tau_{2}$};
					\draw (0.4,0.7) node[] {\tiny $\tau_3$};
				\end{tikzpicture} + (\cdots).
			\end{aligned}
		\end{equation*}
		where the remainder terms in $(\cdots)$ are not equal. This identity is computed
		from the definition of $ \hat{\phi}_\geo $ on decorated given in \cite[Section 6.1]{BGHZ}.
		If the two trees $\tau_2$ and $\tau_3$ are distinct then they produce two variables in the system with the same number of thick brown edges. This is where one must compute the orbits. 
		While it is possible to perform this computation by hand for small $N$, it becomes intractable in general without sophisticated techniques. Proofs by hand were performed in \cite{BGHZ} up to four noises. Therefore, one needs to change perspective by considering tools coming from
		operadic theory and homological algebra as used in \cite{BD24}. We finish this section by summarising the main arguments of \cite{BD24}.
		
		When one considers the geometric KPZ for sufficiently large $n,d$, the multi-indices are not sufficient to encode the elementary differentials and one must use an extension of the set $ \mathcal{T}_{\text{\tiny{Chr}}} $ to obtain the injectivity of Theorem \ref{injectivity_Upsilon}. It is obtained by adding decorations to the noise type edges. But in fact, one works with a smaller set for the counter-terms in the equation. Indeed, if one considers the following stochastic iterated integrals: 
		\begin{equation*}
			\begin{aligned}
				&	(\partial_x K * \xi^{\varepsilon}_i) \, \partial_x K* ( \xi^{\varepsilon}_j  K * (\xi^{\varepsilon}_k	 K * \xi^{\varepsilon}_{\ell}))
				= \int \partial_x K(\cdot - z_1) 
				\partial_x K(\cdot - z_2) \\ &  K(z_2 - z_3)
				K(z_3 - z_4) \xi^{\varepsilon}_i(z_1) \xi^{\varepsilon}_j(z_2) \xi^{\varepsilon}_k(z_3)
				\xi^{\varepsilon}_{\ell}(z_4) d z_1 ... d z_4. 
			\end{aligned}
		\end{equation*}
		Then, taking the expectation and using 
		Isserlis's theorem \eqref{Isserles}, one has
		\begin{equation*}
			\begin{aligned}
				&	\mathbb{E}( (\partial_x K * \xi^{\varepsilon}_i) \, \partial_x K* ( \xi^{\varepsilon}_j  K * (\xi^{\varepsilon}_k	 K * \xi^{\varepsilon}_{\ell})) )
				= \int \partial_x K(\cdot - z_1) 
				\partial_x K(\cdot - z_2) \\ &  K(z_2 - z_3)
				K(z_3 - z_4) \mathbb{E}(\xi^{\varepsilon}_i(z_1) \xi^{\varepsilon}_i(z_2)) \mathbb{E}(\xi^{\varepsilon}_{\ell}(z_3)
				\xi^{\varepsilon}_{\ell}(z_4)) d z_1 ... d z_4 + (\cdots)
			\end{aligned}
		\end{equation*}
	Since the noises $ \xi_i^{\varepsilon} $ are  i.i.d, one can identify the indices inside the expectation ( $i=j$ and $k = \ell$). One does not need to specify the labels on the noises but only their partition into pairs which significantly reduces  the coding. For example, at the level of elementary differentials, one has 
		\begin{equation*}
			\big( F_{\Gamma, \sigma}(\IiXiiiiacc)\big)^{\alpha}(u)
			= \sum_{i,j = 1}^n 2   \Gamma^\alpha_{\beta\gamma}(u)\,   \sigma_i^\beta(u)\,\partial_{\eta}\sigma_i^\gamma(u)\,\partial_{\zeta}\sigma_j^\eta(u)\, \sigma_j^\zeta(u)\;.
		\end{equation*}
		The elementary differentials looks similar to those with multi-indices in terms of monomials in $ \Gamma, \sigma$ and their derivatives. The main difference lies in the indices that encode the tree structure via the connectivity of the various nodes. One natural generating set consists of the iterated combinatorial derivatives which are defined on these decorated trees by
		\begin{equation} \label{covariant_derivative_semi_linear}
			\nabla_{\tau_1}\tau_2 = 
			\begin{tikzpicture}[scale=0.2,baseline=2]
				\draw[symbols]  (-.5,2.5) -- (0,0) ;
				\draw[tinydots] (0,0)  -- (0,-1.3);
				\node[var] (root) at (0,-0.1) {\tiny{$ \tau_1 $ }};
				\node[var] (diff) at (-0.5,2.5) {\tiny{$ \tau_2 $ }};
				\node[blank] at (0.3,1.25) {\tiny{}};
			\end{tikzpicture} \, +  \frac{1}{2}\;
			\begin{tikzpicture}[scale=0.2,baseline=2]
				\coordinate (root) at (0,0);
				\coordinate (t1) at (-1,2);
				\coordinate (t2) at (1,2);
				\draw[tinydots] (root)  -- +(0,-0.8);
				\draw[kernels2] (t1) -- (root);
				\draw[kernels2] (t2) -- (root);
				\node[not] (rootnode) at (root) {};
				\node[var] (t1) at (t1) {\tiny{$ \tau_1 $}};
				\node[var] (t1) at (t2) {\tiny{$ \tau_2 $}};
				\node[blank] at (-1.2,0.6) {\tiny{$$}};
				\node[blank] at (1.2,0.6) {\tiny{$$}};
			\end{tikzpicture}
		\end{equation}
		where the $ \begin{tikzpicture}[scale=0.2,baseline=2]
			\draw[symbols]  (-.5,2.5) -- (0,0) ;
			\draw[tinydots] (0,0)  -- (0,-1.3);
			\node[var] (root) at (0,-0.1) {\tiny{$ \tau_1 $ }};
			\node[var] (diff) at (-0.5,2.5) {\tiny{$ \tau_2 $ }};
			\node[blank] at (0.3,1.25) {\tiny{}};
		\end{tikzpicture} $  are all  the decorated trees obtained by  grafting  $ \tau_2 $ onto $ \tau_1 $ (connecting the root of $ \tau_2 $ to a node of $ \tau_1 $ via a thin edge  and summing over all the possibilities). Of course, in the end, we retain only those iterated covariant derivatives that respect the fact that the noises must appear in pairs. Below, we list the  $16$ elements containing two and four noises
		\begin{equation*}
			\begin{aligned}
				&
				\big\{ \nabla_{\generic}\generic,  \nabla_{\genericb}\nabla_{\generic}\nabla_{\genericb}\generic, 
				\nabla_{\generic}\nabla_{\genericb}\nabla_{\genericb}\generic,
				\nabla_{\genericb}\nabla_{\genericb}\nabla_{\generic}\generic, 
				\nabla_{\genericb}\nabla_{\nabla_{\genericb}\generic}\generic, 
				\nabla_{\nabla_{\genericb}\generic} \nabla_{\genericb}\generic, \\ & 
				\nabla_{\nabla_{\generic}\genericb} \nabla_{\genericb}\generic, 
				\nabla_{\nabla_{\genericb}\genericb}\nabla_{\generic}\generic, 
				\nabla_{\nabla_{\genericb}\nabla_{\genericb}\generic}\generic,
				\nabla_{\nabla_{\genericb}\nabla_{\generic}\generic}\genericb,
				\nabla_{\nabla_{\nabla_{\genericb}\generic}\generic}\genericb, \\ & 
				\nabla_{\nabla_{\nabla_{\genericb}\generic}\genericb}\generic, 
				\nabla_{\generic}\nabla_{\nabla_{\genericb}\generic}\genericb,
				\nabla_{\nabla_{\generic}\nabla_{\genericb}\generic}\genericb,
				\nabla_{\nabla_{\nabla_{\generic}\generic}\genericb}\genericb,
				\nabla_{\genericb}\nabla_{\nabla_{\generic}\generic}\genericb\ 
				\big\}.
			\end{aligned}
		\end{equation*}
		Then, one wants to check that this is actually a generating set for the chain rule and to compute its dimension. It turns out that this is quite tricky with such representation. Also, one can replace Gaussian noises by non-Gaussian ones. In this case, one obtains an analogue of Isserlis's theorem which considers general partitions rather than only pairs. Hence, one wants to remove the generators (pairs of noises in the Gaussian case) and carry out the proof in this more general context. This is where Species (introduced by Joyal in \cite{J81}) and Operad theory (see \cite{LV12} for an overview) enter the game. One works with decorated trees from $ \mathcal{T}_{\text{\tiny{Chr}}} $ but where the noises are replaced with labelled nodes. For example, one has
		\begin{equation*}
			\Xii \longrightarrow 	
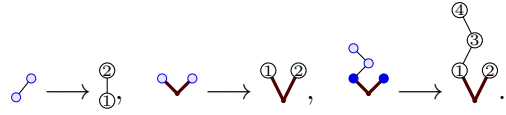
\begin{tikzpicture}[scale=0.2,baseline=-5]
				\coordinate (root) at (0,-1);
				\coordinate (center) at (0,1);
				\draw[] (root) -- (center);
				\node[var1] (rootnode) at (center) {\tiny{$ 2 $}};
				\node[var1] (rootnode) at (root) { \tiny{$  1 $}};
			\end{tikzpicture}, \quad \IXitwo \longrightarrow \begin{tikzpicture}[scale=0.2,baseline=-5]
				\coordinate (root) at (0,-1);
				\coordinate (right) at (1,1);
				\coordinate (left) at (-1,1);
				\draw[kernels2] (root) -- (left);
				\draw[kernels2] (root) -- (right);
				\node[var1] (rootnode) at (left) {\tiny{$ 1 $}};
				\node[var1] (rootnode) at (right) {\tiny{$ 2 $}};
				\node[not] (rootnode) at (root) {};
			\end{tikzpicture}, \quad	\IiXiiiiacc \longrightarrow \begin{tikzpicture}[scale=0.2,baseline=-5]
				\coordinate (root) at (0,-1);
				\coordinate (right) at (1,1);
				\coordinate (left) at (-1,1);
				\coordinate (leftv) at (-0,3);
				\coordinate (leftvv) at (-1,5);
				\draw[] (left) -- (leftv);
				\draw[] (leftv) -- (leftvv);
				\draw[kernels2] (root) -- (left);
				\draw[kernels2] (root) -- (right);
				\node[var1] (rootnode) at (left) {\tiny{$ 1 $}};
				\node[var1] (rootnode) at (right) {\tiny{$ 2 $}};
				\node[not] (rootnode) at (root) {};
				\node[var1] (rootnode) at (leftv) { \tiny{$  3 $}};
				\node[var1] (rootnode) at (leftvv) { \tiny{$  4 $}};
			\end{tikzpicture}.
		\end{equation*}
		The labels serve to distinguish the nodes. The \emph{species of Christoffel trees} assigns to every finite set $S$ the vector space $\NT(S)$, which is spanned by the basis of all Christoffel trees $\mathcal{T}_{\text{\tiny{Chr}}}$ whose noise nodes are placed in a bijection with $S$. From this species, one can define an operad by considering the partial compositions defined below:
		\begin{equation*}
			\begin{tikzpicture}[scale=0.2,baseline=-5]
				\coordinate (root) at (0,-1);
				\coordinate (center) at (0,1);
				\draw[] (root) -- (center);
				\node[var1] (rootnode) at (center) {\tiny{$ 2 $}};
				\node[var1] (rootnode) at (root) { \tiny{$  1 $}};
			\end{tikzpicture}\,\circ_{\tiny{1}} \begin{tikzpicture}[scale=0.2,baseline=-5]
				\coordinate (root) at (0,-1);
				\coordinate (right) at (1,1);
				\coordinate (left) at (-1,1);
				\draw[kernels2] (root) -- (left);
				\draw[kernels2] (root) -- (right);
				\node[var1] (rootnode) at (left) {\tiny{$ a $}};
				\node[var1] (rootnode) at (right) {\tiny{$ b $}};
				\node[not] (rootnode) at (root) {};
			\end{tikzpicture}=
			\begin{tikzpicture}[scale=0.2,baseline=-5]
				\coordinate (root) at (0,-1);
				\coordinate (right) at (1,1);
				\coordinate (left) at (-1,1);
				\coordinate (leftv) at (-1,3);
				\draw[] (left) -- (leftv);
				\draw[kernels2] (root) -- (left);
				\draw[kernels2] (root) -- (right);
				\node[var1] (rootnode) at (left) {\tiny{$ a $}};
				\node[var1] (rootnode) at (right) {\tiny{$ b $}};
				\node[not] (rootnode) at (root) {};
				\node[var1] (rootnode) at (leftv) { \tiny{$  2 $}};
			\end{tikzpicture}+
			\begin{tikzpicture}[scale=0.2,baseline=-5]
				\coordinate (root) at (0,-1);
				\coordinate (right) at (1,1);
				\coordinate (rightv) at (1,3);
				\coordinate (left) at (-1,1);
				\draw[] (right) -- (rightv);
				\draw[kernels2] (root) -- (left);
				\draw[kernels2] (root) -- (right);
				\node[var1] (rootnode) at (left) {\tiny{$ a $}};
				\node[var1] (rootnode) at (right) {\tiny{$ b $}};
				\node[not] (rootnode) at (root) {};
				\node[var1] (rootnode) at (rightv) { \tiny{$  2 $}};
			\end{tikzpicture}+
			\begin{tikzpicture}[scale=0.2,baseline=-5]
				\coordinate (root) at (0,-1);
				\coordinate (right) at (2,1);
				\coordinate (left) at (0,1);
				\coordinate (newv) at (-2,1);
				\draw[] (root) -- (newv);
				\draw[kernels2] (root) -- (left);
				\draw[kernels2] (root) -- (right);
				\node[var1] (rootnode) at (left) {\tiny{$ a $}};
				\node[var1] (rootnode) at (right) {\tiny{$ b $}};
				\node[not] (rootnode) at (root) {};
				\node[var1] (rootnode) at (newv) { \tiny{$  2 $}};
			\end{tikzpicture}
			,\quad
			\begin{tikzpicture}[scale=0.2,baseline=-5]
				\coordinate (root) at (0,-1);
				\coordinate (center) at (0,1);
				\draw[] (root) -- (center);
				\node[var1] (rootnode) at (center) {\tiny{$ 2 $}};
				\node[var1] (rootnode) at (root) { \tiny{$  1 $}};
			\end{tikzpicture}\, \circ_2 \begin{tikzpicture}[scale=0.2,baseline=-5]
				\coordinate (root) at (0,-1);
				\coordinate (right) at (1,1);
				\coordinate (left) at (-1,1);
				\draw[kernels2] (root) -- (left);
				\draw[kernels2] (root) -- (right);
				\node[var1] (rootnode) at (left) {\tiny{$ a $}};
				\node[var1] (rootnode) at (right) {\tiny{$ b $}};
				\node[not] (rootnode) at (root) {};
			\end{tikzpicture}=
			\begin{tikzpicture}[scale=0.2,baseline=-5]
				\coordinate (root) at (0,-1);
				\coordinate (vert) at (0,1);
				\coordinate (right) at (1,3);
				\coordinate (left) at (-1,3);
				\draw[] (root) -- (vert);
				\draw[kernels2] (vert) -- (left);
				\draw[kernels2] (vert) -- (right);
				\node[var1] (rootnode) at (left) {\tiny{$ a $}};
				\node[var1] (rootnode) at (right) {\tiny{$ b $}};
				\node[not] (rootnode) at (vert) {};
				\node[var1] (rootnode) at (root) { \tiny{$  1 $}};
			\end{tikzpicture}
			.
		\end{equation*}
		
		One inserts the labelled tree into the node with the corresponding label and connects the edges previously attached to this node to the nodes of the inserted tree in all possible ways.
		In the first computation, one inserts the tree at the node labelled $ 1 $ and has three possibilities for reconnecting the edge attached to that node. To form an operad, the partial composition must satisfy natural properties. First, one has the identity axiom via the existence of an element $ \mathrm{e} $ of arity one such that 
		\begin{equation*}
			\tau \circ_i\mathrm{e}  =\mathrm{e}\circ_1 \tau=\tau.
			\end{equation*}
			The second axiom is the sequential composition axiom
			\begin{equation*}
				(\tau_1 \circ_i \tau_2) \circ_j \tau_3  = \tau_1 \circ_i (\tau_2 \circ_j \tau_3), 
				\end{equation*}
				where $i$ is a label of $ \tau_1 $ and $ j $ is a label of $\tau_2$.
				The third axiom is the parallel composition axiom given by 
					\begin{equation*}
					(\tau_1 \circ_i \tau_2) \circ_j \tau_3  = (\tau_1 \circ_j \tau_3) \circ_i \tau_2, 
				\end{equation*}
			where $i, j$ are labels of $ \tau_1 $.
		
		\begin{proposition} \label{operad_christoffel}
			The partial compositions $\circ_v$ make the linear species $\NT$ into an operad. That operad is generated by its binary operations given below
			\begin{equation} \label{generators}
				\begin{tikzpicture}[scale=0.2,baseline=-5]
					\coordinate (root) at (0,-1);
					\coordinate (center) at (0,1);
					\draw[] (root) -- (center);
					\node[var1] (rootnode) at (center) {\tiny{$ 2 $}};
					\node[var1] (rootnode) at (root) { \tiny{$  1 $}};
				\end{tikzpicture} \, , \quad \begin{tikzpicture}[scale=0.2,baseline=-5]
					\coordinate (root) at (0,-1);
					\coordinate (center) at (0,1);
					\draw[] (root) -- (center);
					\node[var1] (rootnode) at (center) {\tiny{$ 1 $}};
					\node[var1] (rootnode) at (root) { \tiny{$  2 $}};
				\end{tikzpicture} \, , \quad \begin{tikzpicture}[scale=0.2,baseline=-5]
					\coordinate (root) at (0,-1);
					\coordinate (right) at (1,1);
					\coordinate (left) at (-1,1);
					\draw[kernels2] (root) -- (left);
					\draw[kernels2] (root) -- (right);
					\node[var1] (rootnode) at (left) {\tiny{$ 1 $}};
					\node[var1] (rootnode) at (right) {\tiny{$ 2 $}};
					\node[not] (rootnode) at (root) {};
				\end{tikzpicture}.
			\end{equation}
			Moreover, that operad is isomorphic to the coproduct of the operads pre-Lie denoted by $\PL$ and the operad  Commutative-Magmatic denoted by $\ComMag$. This coproduct is denoted by  $\PL \vee \ComMag$.
		\end{proposition}
		This proposition is taken from \cite[Proposition 4.6]{BD24}. The first two generators in \eqref{generators} are the generators of the pre-Lie operad $ \PL $. Indeed, if we compute the following partial compositions, one gets 
		\begin{equation*}
			(\begin{tikzpicture}[scale=0.2,baseline=-5]
				\coordinate (root) at (0,-1);
				\coordinate (center) at (0,1);
				\draw[] (root) -- (center);
				\node[var1] (rootnode) at (center) {\tiny{$ 2 $}};
				\node[var1] (rootnode) at (root) { \tiny{$  1 $}};
			\end{tikzpicture} \circ_1 \tau_1) \circ_2 \tau_2 = \tau_2 \curvearrowright  \tau_1
		\end{equation*}
		where $ \curvearrowright $ is the grafting product that we already introduced in the proof of Proposition \ref{lower_bound_dimension}. This product is pre-Lie as its satisfies the following identity
		\begin{equation} \label{pre_lie_ident}
			(\tau_1  \curvearrowright \tau_2 )   \curvearrowright \tau_3  - \tau_1  \curvearrowright (\tau_2   \curvearrowright \tau_3) = 	(\tau_2  \curvearrowright \tau_1 )   \curvearrowright \tau_3  - \tau_2  \curvearrowright (\tau_1   \curvearrowright \tau_3).
		\end{equation}
		The last generator of \eqref{generators} corresponds to the generator of $ \ComMag $. It is commuttative as one has 
		\begin{equation*}
			(\begin{tikzpicture}[scale=0.2,baseline=-5]
				\coordinate (root) at (0,-1);
				\coordinate (right) at (1,1);
				\coordinate (left) at (-1,1);
				\draw[kernels2] (root) -- (left);
				\draw[kernels2] (root) -- (right);
				\node[var1] (rootnode) at (left) {\tiny{$ 1 $}};
				\node[var1] (rootnode) at (right) {\tiny{$ 2 $}};
				\node[not] (rootnode) at (root) {};
			\end{tikzpicture} \circ_1 \tau_1) \circ_2 \tau_2  = 	(\begin{tikzpicture}[scale=0.2,baseline=-5]
				\coordinate (root) at (0,-1);
				\coordinate (right) at (1,1);
				\coordinate (left) at (-1,1);
				\draw[kernels2] (root) -- (left);
				\draw[kernels2] (root) -- (right);
				\node[var1] (rootnode) at (left) {\tiny{$ 1 $}};
				\node[var1] (rootnode) at (right) {\tiny{$ 2 $}};
				\node[not] (rootnode) at (root) {};
			\end{tikzpicture} \circ_1 \tau_2) \circ_2 \tau_1 
		\end{equation*}
		and is magmatic as we do not assume any additional identities for this product. Then, the coproduct of operads can be viewed as a concatenation of the lists of the generators. Using this formalism, one lifts the map $ \hat{\varphi}_{\geo} $ to  a map of species 
		\[
		\hat \Phi_{\geo}\colon \NT\to \NT\vee\mathbb{K} \, \begin{tikzpicture}[scale=0.2,baseline=-2]
			\coordinate (root) at (0,0);
			\node[diff] (rootnode) at (root) {};
		\end{tikzpicture}
		\]
		such that
		\begin{gather*}
			\hat \Phi_\geo \left( \begin{tikzpicture}[scale=0.2,baseline=-5]
				\coordinate (root) at (0,-1);
				\coordinate (center) at (0,1);
				\draw[] (root) -- (center);
				\node[var1] (rootnode) at (center) {\tiny{$ 2 $}};
				\node[var1] (rootnode) at (root) { \tiny{$  1 $}};
			\end{tikzpicture} \right) = \hat \Phi_\geo \left(  \begin{tikzpicture}[scale=0.2,baseline=-5]
				\coordinate (root) at (0,-1);
				\coordinate (center) at (0,1);
				\draw[] (root) -- (center);
				\node[var1] (rootnode) at (center) {\tiny{$ 1 $}};
				\node[var1] (rootnode) at (root) { \tiny{$  2 $}};
			\end{tikzpicture} \right) =
			\begin{tikzpicture}[scale=0.2,baseline=-5]
				\coordinate (root) at (0,-1);
				\coordinate (right) at (1,1);
				\coordinate (left) at (-1,1);
				\draw[] (root) -- (left);
				\draw[] (right) -- (root);
				\node[var1] (rootnode) at (left) {\tiny{$ 1 $}};
				\node[var1] (rootnode) at (right) {\tiny{$ 2 $}};
				\node[diff] (rootnode) at (root) {};
			\end{tikzpicture} ,\quad 
			\hat \Phi_\geo \left(\begin{tikzpicture}[scale=0.2,baseline=-5]
				\coordinate (root) at (0,-1);
				\coordinate (right) at (1,1);
				\coordinate (left) at (-1,1);
				\draw[kernels2] (root) -- (left);
				\draw[kernels2] (root) -- (right);
				\node[var1] (rootnode) at (left) {\tiny{$ 1 $}};
				\node[var1] (rootnode) at (right) {\tiny{$ 2 $}};
				\node[not] (rootnode) at (root) {};
			\end{tikzpicture} \right) = -2\begin{tikzpicture}[scale=0.2,baseline=-5]
				\coordinate (root) at (0,-1);
				\coordinate (right) at (1,1);
				\coordinate (left) at (-1,1);
				\draw[] (root) -- (left);
				\draw[] (right) -- (root);
				\node[var1] (rootnode) at (left) {\tiny{$ 1 $}};
				\node[var1] (rootnode) at (right) {\tiny{$ 2 $}};
				\node[diff] (rootnode) at (root) {};
			\end{tikzpicture}.
		\end{gather*} 
		The map $\hat \Phi_\geo$ can be viewed as a deformation of a natural derivation that appears in the operadic twisting. This procedure originates in the work of Willwacher \cite[Appendix~I]{W15} who introduced it in the context of Kontsevich's graph complexes. The work \cite{dotsenko2021homotopical} gives a precise description of  the operadic twisting $\Tw(\PL) = (\PL \vee \mathbb{K} \, \begin{tikzpicture}[scale=0.2,baseline=-2]
			\coordinate (root) at (0,0);
			\node[diff] (rootnode) at (root) {};
		\end{tikzpicture}, d_{\Tw} )$ of $ \PL $. It is a differential graded operad and $ \begin{tikzpicture}[scale=0.2,baseline=-2]
			\coordinate (root) at (0,0);
			\node[diff] (rootnode) at (root) {};
		\end{tikzpicture} $ is a Maurer-Cartan element. In the sense that one can split the derivation $  d_{\Tw} $ into 
		\begin{equation*}
			d_{\Tw}=d_{\MC}+\mathrm{ad}_{\ell_1^{\begin{tikzpicture}[scale=0.2,baseline=-2]
						\coordinate (root) at (0,0);
						\node[diff] (rootnode) at (root) {};
			\end{tikzpicture}}}.
		\end{equation*}
		The generator $ \begin{tikzpicture}[scale=0.2,baseline=-2]
			\coordinate (root) at (0,0);
			\node[diff] (rootnode) at (root) {};
		\end{tikzpicture} $ is a new operation of arity~$0$ and homological degree~$-1$. 
		The differential $d_{\MC}$ vanishes on $\PL$ and satisfies \begin{equation*}
			d_{\MC}(\begin{tikzpicture}[scale=0.2,baseline=-2]
				\coordinate (root) at (0,0);
				\node[diff] (rootnode) at (root) {};
			\end{tikzpicture})=-\frac12[\begin{tikzpicture}[scale=0.2,baseline=-2]
				\coordinate (root) at (0,0);
				\node[diff] (rootnode) at (root) {};
			\end{tikzpicture},\begin{tikzpicture}[scale=0.2,baseline=-2]
				\coordinate (root) at (0,0);
				\node[diff] (rootnode) at (root) {};
			\end{tikzpicture}]  := -\frac12 \begin{tikzpicture}[scale=0.2,baseline=-5]
				\coordinate (root) at (0,-1);
				\coordinate (center) at (0,1);
				\draw[] (root) -- (center);
				\node[diff] (rootnode) at (center) {\tiny{$ $}};
				\node[diff] (rootnode) at (root) { \tiny{$  $}};
			\end{tikzpicture}.
		\end{equation*} 
		The Lie bracket $ [\cdot, \cdot] $ is given on other elements by
		\begin{equation*}
			[\tau_1, \tau_2] = 	\begin{tikzpicture}[scale=0.2,baseline=-5]
				\coordinate (root) at (0,-1);
				\coordinate (center) at (0,1.5);
				\draw[] (root) -- (center);
				\node[var1] (rootnode) at (center) {\tiny{$ \tau_2 $}};
				\node[var1] (rootnode) at (root) { \tiny{$  \tau_1 $}};
			\end{tikzpicture} -
			\begin{tikzpicture}[scale=0.2,baseline=-5]
				\coordinate (root) at (0,-1);
				\coordinate (center) at (0,1.5);
				\draw[] (root) -- (center);
				\node[var1] (rootnode) at (center) {\tiny{$ \tau_1 $}};
				\node[var1] (rootnode) at (root) { \tiny{$  \tau_2 $}};
			\end{tikzpicture}, 
		\end{equation*} 
		which is simply the antisymmetrisation of the pre-Lie product. Identity  \eqref{pre_lie_ident} implies the Jacobi identity for the bracket $ [\cdot, \cdot] $.
		The element $\ell_1^{\alpha}$ is defined by
		\begin{equation*}
			\ell_1^{\begin{tikzpicture}[scale=0.2,baseline=-2]
					\coordinate (root) at (0,0);
					\node[diff] (rootnode) at (root) {};
			\end{tikzpicture}}(a_1)=[\begin{tikzpicture}[scale=0.2,baseline=-2]
				\coordinate (root) at (0,0);
				\node[diff] (rootnode) at (root) {};
			\end{tikzpicture},a_1] =  \begin{tikzpicture}[scale=0.2,baseline=-5]
				\coordinate (root) at (0,-1);
				\coordinate (center) at (0,1);
				\draw[] (root) -- (center);
				\node[diff] (rootnode) at (center) {\tiny{$ $}};
				\node[var1] (rootnode) at (root) { \tiny{\tiny{$  1 $}}};
			\end{tikzpicture} -
			\begin{tikzpicture}[scale=0.2,baseline=-5]
				\coordinate (root) at (0,-1);
				\coordinate (center) at (0,1);
				\draw[] (root) -- (center);
				\node[var1] (rootnode) at (center) {\tiny{\tiny{$ 1 $}}};
				\node[diff] (rootnode) at (root) { \tiny{$  $}};
			\end{tikzpicture}.
		\end{equation*}
		In the sequel, we will use the shorthand notation $ \ell_1^{\begin{tikzpicture}[scale=0.2,baseline=-2]
				\coordinate (root) at (0,0);
				\node[diff] (rootnode) at (root) {};
		\end{tikzpicture}} $ instead of $ \ell_1^{\begin{tikzpicture}[scale=0.2,baseline=-2]
				\coordinate (root) at (0,0);
				\node[diff] (rootnode) at (root) {};
		\end{tikzpicture}}(a_1) $.
		Then, one sets \begin{equation*}
			\mathrm{ad}_{\ell_1^{\begin{tikzpicture}[scale=0.2,baseline=-2]
						\coordinate (root) at (0,0);
						\node[diff] (rootnode) at (root) {};
			\end{tikzpicture}}}(\mu)=\ell_1^{\begin{tikzpicture}[scale=0.2,baseline=-2]
					\coordinate (root) at (0,0);
					\node[diff] (rootnode) at (root) {};
			\end{tikzpicture}}\circ_1\mu-(-1)^{|\mu|}\sum_i\mu\circ_i\ell_1^{\begin{tikzpicture}[scale=0.2,baseline=-2]
					\coordinate (root) at (0,0);
					\node[diff] (rootnode) at (root) {};
			\end{tikzpicture}}.
		\end{equation*} 
		where $ |\mu| $ is the homological degree of the element $\mu$.
		If we apply the differential $ 	d_{\Tw} $ to the generators of $ \PL $, one gets
		\begin{equation*}
			\begin{aligned}
				d_{\Tw} \left( \,   \begin{tikzpicture}[scale=0.2,baseline=-5]
					\coordinate (root) at (0,-1);
					\coordinate (center) at (0,1);
					\draw[] (root) -- (center);
					\node[var1] (rootnode) at (center) {\tiny{$ 2 $}};
					\node[var1] (rootnode) at (root) { \tiny{$  1 $}};
				\end{tikzpicture} \, \right)  & =
				\mathrm{ad}_{\ell_1^{\begin{tikzpicture}[scale=0.2,baseline=-2]
							\coordinate (root) at (0,0);
							\node[diff] (rootnode) at (root) {};
				\end{tikzpicture}}}\left( \begin{tikzpicture}[scale=0.2,baseline=-5]
					\coordinate (root) at (0,-1);
					\coordinate (center) at (0,1);
					\draw[] (root) -- (center);
					\node[var1] (rootnode) at (center) {\tiny{$ 2 $}};
					\node[var1] (rootnode) at (root) { \tiny{$  1 $}};
				\end{tikzpicture} \right)
				= \ell_1^{\begin{tikzpicture}[scale=0.2,baseline=-2]
						\coordinate (root) at (0,0);
						\node[diff] (rootnode) at (root) {};
				\end{tikzpicture}}\circ_1 \begin{tikzpicture}[scale=0.2,baseline=-5]
					\coordinate (root) at (0,-1);
					\coordinate (center) at (0,1);
					\draw[] (root) -- (center);
					\node[var1] (rootnode) at (center) {\tiny{$ 2 $}};
					\node[var1] (rootnode) at (root) { \tiny{$  1 $}};
				\end{tikzpicture}-(-1)^{0}\sum_i \begin{tikzpicture}[scale=0.2,baseline=-5]
					\coordinate (root) at (0,-1);
					\coordinate (center) at (0,1);
					\draw[] (root) -- (center);
					\node[var1] (rootnode) at (center) {\tiny{$ 2 $}};
					\node[var1] (rootnode) at (root) { \tiny{$  1 $}};
				\end{tikzpicture} \circ_i\ell_1^{\begin{tikzpicture}[scale=0.2,baseline=-2]
						\coordinate (root) at (0,0);
						\node[diff] (rootnode) at (root) {};
				\end{tikzpicture}}
				=
				\, \begin{tikzpicture}[scale=0.2,baseline=-5]
					\coordinate (root) at (0,-1);
					\coordinate (right) at (1,1);
					\coordinate (left) at (-1,1);
					\draw[] (root) -- (left);
					\draw[] (right) -- (root);
					\node[var1] (rootnode) at (left) {\tiny{$ 1 $}};
					\node[var1] (rootnode) at (right) {\tiny{$ 2 $}};
					\node[diff] (rootnode) at (root) {$  $};
				\end{tikzpicture}.
			\end{aligned}
		\end{equation*}
		This is due to the following computations
		\begin{equation*}
			\begin{aligned}
				\ell_1^{\begin{tikzpicture}[scale=0.2,baseline=-2]
						\coordinate (root) at (0,0);
						\node[diff] (rootnode) at (root) {};
				\end{tikzpicture}}\circ_1 \begin{tikzpicture}[scale=0.2,baseline=-5]
					\coordinate (root) at (0,-1);
					\coordinate (center) at (0,1);
					\draw[] (root) -- (center);
					\node[var1] (rootnode) at (center) {\tiny{$ 2 $}};
					\node[var1] (rootnode) at (root) { \tiny{$  1 $}};
				\end{tikzpicture} & =
				(\begin{tikzpicture}[scale=0.2,baseline=-5]
					\coordinate (root) at (0,-1);
					\coordinate (center) at (0,1);
					\draw[] (root) -- (center);
					\node[diff] (rootnode) at (center) {\tiny{$ $}};
					\node[var1] (rootnode) at (root) { \tiny{\tiny{$  1 $}}};
				\end{tikzpicture} -
				\begin{tikzpicture}[scale=0.2,baseline=-5]
					\coordinate (root) at (0,-1);
					\coordinate (center) at (0,1);
					\draw[] (root) -- (center);
					\node[var1] (rootnode) at (center) {\tiny{\tiny{$ 1 $}}};
					\node[diff] (rootnode) at (root) { \tiny{$  $}};
				\end{tikzpicture}) \circ_1 \begin{tikzpicture}[scale=0.2,baseline=-5]
					\coordinate (root) at (0,-1);
					\coordinate (center) at (0,1);
					\draw[] (root) -- (center);
					\node[var1] (rootnode) at (center) {\tiny{$ 2 $}};
					\node[var1] (rootnode) at (root) { \tiny{$  1 $}};
				\end{tikzpicture}
				=	 \begin{tikzpicture}[scale=0.2,baseline=-5]
					\coordinate (root) at (0,-1);
					\coordinate (center) at (0,1);
					\coordinate (centerc) at (0,3);
					\draw[] (root) -- (center);
					\draw[] (center) -- (centerc);
					\node[var1] (rootnode) at (center) {\tiny{$ 2 $}};
					\node[var1] (rootnode) at (root) { \tiny{$  1 $}};
					\node[diff] (rootnode) at (centerc) { \tiny{$  $}};
				\end{tikzpicture} + \begin{tikzpicture}[scale=0.2,baseline=-5]
					\coordinate (root) at (0,-1);
					\coordinate (center) at (1,1);
					\coordinate (centerc) at (-1,1);
					\draw[] (root) -- (center);
					\draw[] (root) -- (centerc);
					\node[var1] (rootnode) at (center) {\tiny{$ 2 $}};
					\node[var1] (rootnode) at (root) { \tiny{$  1 $}};
					\node[diff] (rootnode) at (centerc) { \tiny{$  $}};
				\end{tikzpicture} - \begin{tikzpicture}[scale=0.2,baseline=-5]
					\coordinate (root) at (0,1);
					\coordinate (center) at (0,3);
					\coordinate (centerc) at (0,-1);
					\draw[] (root) -- (center);
					\draw[] (center) -- (centerc);
					\node[var1] (rootnode) at (center) {\tiny{$ 2 $}};
					\node[var1] (rootnode) at (root) { \tiny{$  1 $}};
					\node[diff] (rootnode) at (centerc) { \tiny{$  $}};
				\end{tikzpicture}, \quad |\begin{tikzpicture}[scale=0.2,baseline=-5]
					\coordinate (root) at (0,-1);
					\coordinate (center) at (0,1);
					\draw[] (root) -- (center);
					\node[var1] (rootnode) at (center) {\tiny{$ 2 $}};
					\node[var1] (rootnode) at (root) { \tiny{$  1 $}};
				\end{tikzpicture}| = 0, \\ \begin{tikzpicture}[scale=0.2,baseline=-5]
					\coordinate (root) at (0,-1);
					\coordinate (center) at (0,1);
					\draw[] (root) -- (center);
					\node[var1] (rootnode) at (center) {\tiny{$ 2 $}};
					\node[var1] (rootnode) at (root) { \tiny{$  1 $}};
				\end{tikzpicture} \circ_1 \ell_1^{\begin{tikzpicture}[scale=0.2,baseline=-2]
						\coordinate (root) at (0,0);
						\node[diff] (rootnode) at (root) {};
				\end{tikzpicture}} & = 
				\begin{tikzpicture}[scale=0.2,baseline=-5]
					\coordinate (root) at (0,-1);
					\coordinate (center) at (0,1);
					\draw[] (root) -- (center);
					\node[var1] (rootnode) at (center) {\tiny{$ 2 $}};
					\node[var1] (rootnode) at (root) { \tiny{$  1 $}};
				\end{tikzpicture} \circ_1		(\begin{tikzpicture}[scale=0.2,baseline=-5]
					\coordinate (root) at (0,-1);
					\coordinate (center) at (0,1);
					\draw[] (root) -- (center);
					\node[diff] (rootnode) at (center) {\tiny{$ $}};
					\node[var1] (rootnode) at (root) { \tiny{\tiny{$  1 $}}};
				\end{tikzpicture} -
				\begin{tikzpicture}[scale=0.2,baseline=-5]
					\coordinate (root) at (0,-1);
					\coordinate (center) at (0,1);
					\draw[] (root) -- (center);
					\node[var1] (rootnode) at (center) {\tiny{\tiny{$ 1 $}}};
					\node[diff] (rootnode) at (root) { \tiny{$  $}};
				\end{tikzpicture})
				=
				\begin{tikzpicture}[scale=0.2,baseline=-5]
					\coordinate (root) at (0,-1);
					\coordinate (center) at (0,1);
					\coordinate (centerc) at (0,3);
					\draw[] (root) -- (center);
					\draw[] (center) -- (centerc);
					\node[var1] (rootnode) at (centerc) {\tiny{$ 2 $}};
					\node[var1] (rootnode) at (root) { \tiny{$  1 $}};
					\node[diff] (rootnode) at (center) { \tiny{$  $}};
				\end{tikzpicture} + \begin{tikzpicture}[scale=0.2,baseline=-5]
					\coordinate (root) at (0,-1);
					\coordinate (center) at (1,1);
					\coordinate (centerc) at (-1,1);
					\draw[] (root) -- (center);
					\draw[] (root) -- (centerc);
					\node[var1] (rootnode) at (center) {\tiny{$ 2 $}};
					\node[var1] (rootnode) at (root) { \tiny{$  1 $}};
					\node[diff] (rootnode) at (centerc) { \tiny{$  $}};
				\end{tikzpicture} - \begin{tikzpicture}[scale=0.2,baseline=-5]
					\coordinate (root) at (0,-1);
					\coordinate (center) at (0,1);
					\coordinate (centerc) at (0,3);
					\draw[] (root) -- (center);
					\draw[] (center) -- (centerc);
					\node[var1] (rootnode) at (centerc) {\tiny{$ 2 $}};
					\node[var1] (rootnode) at (center) { \tiny{$  1 $}};
					\node[diff] (rootnode) at (root) { \tiny{$  $}};
				\end{tikzpicture} - \begin{tikzpicture}[scale=0.2,baseline=-5]
					\coordinate (root) at (0,-1);
					\coordinate (center) at (1,1);
					\coordinate (centerc) at (-1,1);
					\draw[] (root) -- (center);
					\draw[] (root) -- (centerc);
					\node[var1] (rootnode) at (center) {\tiny{$ 2 $}};
					\node[var1] (rootnode) at (centerc) { \tiny{$  1 $}};
					\node[diff] (rootnode) at (root) { \tiny{$  $}};
				\end{tikzpicture}, \\ \begin{tikzpicture}[scale=0.2,baseline=-5]
					\coordinate (root) at (0,-1);
					\coordinate (center) at (0,1);
					\draw[] (root) -- (center);
					\node[var1] (rootnode) at (center) {\tiny{$ 2 $}};
					\node[var1] (rootnode) at (root) { \tiny{$  1 $}};
				\end{tikzpicture} \circ_2 \ell_1^{\begin{tikzpicture}[scale=0.2,baseline=-2]
						\coordinate (root) at (0,0);
						\node[diff] (rootnode) at (root) {};
				\end{tikzpicture}} & = 	\begin{tikzpicture}[scale=0.2,baseline=-5]
					\coordinate (root) at (0,-1);
					\coordinate (center) at (0,1);
					\draw[] (root) -- (center);
					\node[var1] (rootnode) at (center) {\tiny{$ 2 $}};
					\node[var1] (rootnode) at (root) { \tiny{$  1 $}};
				\end{tikzpicture} \circ_2		(\begin{tikzpicture}[scale=0.2,baseline=-5]
					\coordinate (root) at (0,-1);
					\coordinate (center) at (0,1);
					\draw[] (root) -- (center);
					\node[diff] (rootnode) at (center) {\tiny{$ $}};
					\node[var1] (rootnode) at (root) { \tiny{\tiny{$  1 $}}};
				\end{tikzpicture} -
				\begin{tikzpicture}[scale=0.2,baseline=-5]
					\coordinate (root) at (0,-1);
					\coordinate (center) at (0,1);
					\draw[] (root) -- (center);
					\node[var1] (rootnode) at (center) {\tiny{\tiny{$ 1 $}}};
					\node[diff] (rootnode) at (root) { \tiny{$  $}};
				\end{tikzpicture}) = \begin{tikzpicture}[scale=0.2,baseline=-5]
					\coordinate (root) at (0,-1);
					\coordinate (center) at (0,1);
					\coordinate (centerc) at (0,3);
					\draw[] (root) -- (center);
					\draw[] (center) -- (centerc);
					\node[var1] (rootnode) at (center) {\tiny{$ 2 $}};
					\node[var1] (rootnode) at (root) { \tiny{$  1 $}};
					\node[diff] (rootnode) at (centerc) { \tiny{$  $}};
				\end{tikzpicture} -	\begin{tikzpicture}[scale=0.2,baseline=-5]
					\coordinate (root) at (0,-1);
					\coordinate (center) at (0,1);
					\coordinate (centerc) at (0,3);
					\draw[] (root) -- (center);
					\draw[] (center) -- (centerc);
					\node[var1] (rootnode) at (centerc) {\tiny{$ 2 $}};
					\node[var1] (rootnode) at (root) { \tiny{$  1 $}};
					\node[diff] (rootnode) at (center) { \tiny{$  $}};
				\end{tikzpicture}.
			\end{aligned}
		\end{equation*}
		In the computation above, we have relabelled the nodes after the partial composition.
		We consider the derivation $d_0$ of the coproduct of operads $\PL\vee\ComMag\vee\mathbb{K} \, \begin{tikzpicture}[scale=0.2,baseline=-2]
			\coordinate (root) at (0,0);
			\node[diff] (rootnode) at (root) {};
		\end{tikzpicture}$ defined by
		\begin{gather*}
			d_0 \left( \begin{tikzpicture}[scale=0.2,baseline=-5]
				\coordinate (root) at (0,-1);
				\coordinate (center) at (0,1);
				\draw[] (root) -- (center);
				\node[var1] (rootnode) at (center) {\tiny{$ 2 $}};
				\node[var1] (rootnode) at (root) { \tiny{$  1 $}};
			\end{tikzpicture} \right) = d_0 \left(  \begin{tikzpicture}[scale=0.2,baseline=-5]
				\coordinate (root) at (0,-1);
				\coordinate (center) at (0,1);
				\draw[] (root) -- (center);
				\node[var1] (rootnode) at (center) {\tiny{$ 1 $}};
				\node[var1] (rootnode) at (root) { \tiny{$  2 $}};
			\end{tikzpicture} \right) =d_0(\begin{tikzpicture}[scale=0.2,baseline=-2]
				\coordinate (root) at (0,0);
				\node[diff] (rootnode) at (root) {};
			\end{tikzpicture})=
			0 ,\quad 
			d_0 \left(\begin{tikzpicture}[scale=0.2,baseline=-5]
				\coordinate (root) at (0,-1);
				\coordinate (right) at (1,1);
				\coordinate (left) at (-1,1);
				\draw[kernels2] (root) -- (left);
				\draw[kernels2] (root) -- (right);
				\node[var1] (rootnode) at (left) {\tiny{$ 1 $}};
				\node[var1] (rootnode) at (right) {\tiny{$ 2 $}};
				\node[not] (rootnode) at (root) {};
			\end{tikzpicture} \right) = -2\begin{tikzpicture}[scale=0.2,baseline=-5]
				\coordinate (root) at (0,-1);
				\coordinate (right) at (1,1);
				\coordinate (left) at (-1,1);
				\draw[] (root) -- (left);
				\draw[] (right) -- (root);
				\node[var1] (rootnode) at (left) {\tiny{$ 1 $}};
				\node[var1] (rootnode) at (right) {\tiny{$ 2 $}};
				\node[diff] (rootnode) at (root) {};
			\end{tikzpicture}.
		\end{gather*} 
		\begin{proposition} \label{correspondence_ker_phi_0}
			The map $\mathrm{d}= d_{\Tw} +d_0$ turns $\PL\vee\ComMag\vee\mathbb{K} \, \begin{tikzpicture}[scale=0.2,baseline=-2]
				\coordinate (root) at (0,0);
				\node[diff] (rootnode) at (root) {};
			\end{tikzpicture}$ into a differential graded operad. The degree zero homology of that operad is naturally isomorphic to $\ker\hat{\Phi}_{\geo}$.
		\end{proposition}
		
		We briefly explain what one needs to prove for the previous proposition.
		One first shows that $\mathrm{d}^2=0$. It is sufficient to prove it on the generators of the operad. We recall that the homological degree of $
		\begin{tikzpicture}[scale=0.2,baseline=-2]
			\coordinate (root) at (0,0);
			\node[diff] (rootnode) at (root) {};
		\end{tikzpicture}$ is $-1$, therefore the degree zero part of $\PL\vee\ComMag\vee\mathbb{K}\,\begin{tikzpicture}[scale=0.2,baseline=-2]
			\coordinate (root) at (0,0);
			\node[diff] (rootnode) at (root) {};
		\end{tikzpicture}$ is precisely $\PL\vee\ComMag$, and one has a commutative diagram 
		\[
		\begin{tikzcd}
			\NT \arrow[r, "\hat{\Phi}_{\geo}"] \arrow[d]
			& \NT\vee \mathbb{K}\,\begin{tikzpicture}[scale=0.2,baseline=-2]
				\coordinate (root) at (0,0);
				\node[diff] (rootnode) at (root) {};
			\end{tikzpicture}  \arrow[d] \\
			\PL\vee\ComMag \arrow[r, "\mathrm{d}"]
			& \PL\vee\ComMag\vee\mathbb{K}\,\begin{tikzpicture}[scale=0.2,baseline=-2]
				\coordinate (root) at (0,0);
				\node[diff] (rootnode) at (root) {};
			\end{tikzpicture}
		\end{tikzcd} ,
		\]
		with vertical arrows being isomorphisms given by Proposition~\ref{operad_christoffel}.  Before computing the homology of the previous differential graded operad, we recall the definition of 
		the operad $\LA$ of Lie-admissible algebras. It can be factorised as $ \Lie\vee\ComMag $ where $ \Lie $ is the operad of Lie algebras. Its precise  description can be found in \cite[Section 13.2]{LV12}. A basis is given by binary planar trees where the labels are on the leaves.
		\begin{theorem}\label{th:twisting}
			The homology of the operad 
			\[
			(\PL\vee\ComMag\vee \mathbb{K}\, \begin{tikzpicture}[scale=0.2,baseline=-2]
				\coordinate (root) at (0,0);
				\node[diff] (rootnode) at (root) {};
			\end{tikzpicture}, d_{\Tw}+d_0)
			\] 
			is concentrated in homological degree zero and is isomorphic to the operad $\LA$ of Lie-admissible algebras. 
		\end{theorem}
		The main observation of the proof is to notice that $ d_{\Tw}+d_0 $ is a deformation of $ d_{\Tw} $. Indeed, the derivation $ d_0 $ decreases the number of thick edges by two. Adding   $ d_0 $ does not change the homology. It remains to compute the homology with $  d_{\Tw} $.
		One has  \[	(\PL\vee\ComMag\vee\mathbb{K}\,\begin{tikzpicture}[scale=0.2,baseline=-2]
			\coordinate (root) at (0,0);
			\node[diff] (rootnode) at (root) {};
		\end{tikzpicture}, d_{\Tw})\cong \Tw(\PL)\vee\ComMag,
		\]
		which is due to the fact that the generator of $\ComMag$ is sent to zero by the derivation $ d_{\Tw} $.
		Then, if we denote the homological group of $\Tw(\PL)$ by	$H_\bullet(\Tw(\PL))$, one has
		\begin{equation*}
			\begin{aligned}
				\ker\hat{\Phi}_{\geo} \cong	H_\bullet(\Tw(\PL)\vee\ComMag)&\cong H_\bullet(\Tw(\PL))\vee\ComMag\\& \cong \Lie\vee\ComMag\cong\LA.
			\end{aligned}
		\end{equation*}
		where one have used  $H_\bullet(\Tw(\PL))\cong \Lie$ established in \cite{dotsenko2021homotopical}. One wants to identify generating binary operations for  $\ker\hat{\Phi}_{\geo}$.
		An operadic version of the covariant derivative \eqref{covariant_derivative_semi_linear} is given by 
		\begin{equation}
			\label{operad_covariant}
			\begin{tikzpicture}[scale=0.2,baseline=-5]
				\coordinate (root) at (0,-1);
				\coordinate (center) at (0,1);
				\draw[] (root) -- (center);
				\node[var1] (rootnode) at (center) {\tiny{$ 2 $}};
				\node[var1] (rootnode) at (root) { \tiny{$  1 $}};
			\end{tikzpicture} + \frac{1}{2} \begin{tikzpicture}[scale=0.2,baseline=-5]
				\coordinate (root) at (0,-1);
				\coordinate (right) at (1,1);
				\coordinate (left) at (-1,1);
				\draw[kernels2] (root) -- (left);
				\draw[kernels2] (root) -- (right);
				\node[var1] (rootnode) at (left) {\tiny{$ 1 $}};
				\node[var1] (rootnode) at (right) {\tiny{$ 2 $}};
				\node[not] (rootnode) at (root) {};
			\end{tikzpicture}.
		\end{equation}
		Indeed, one has
		\begin{equation*}
			\nabla_{\tau_1}\tau_2 = (( \begin{tikzpicture}[scale=0.2,baseline=-5]
				\coordinate (root) at (0,-1);
				\coordinate (center) at (0,1);
				\draw[] (root) -- (center);
				\node[var1] (rootnode) at (center) {\tiny{$ 2 $}};
				\node[var1] (rootnode) at (root) { \tiny{$  1 $}};
			\end{tikzpicture} + \frac{1}{2} \begin{tikzpicture}[scale=0.2,baseline=-5]
				\coordinate (root) at (0,-1);
				\coordinate (right) at (1,1);
				\coordinate (left) at (-1,1);
				\draw[kernels2] (root) -- (left);
				\draw[kernels2] (root) -- (right);
				\node[var1] (rootnode) at (left) {\tiny{$ 1 $}};
				\node[var1] (rootnode) at (right) {\tiny{$ 2 $}};
				\node[not] (rootnode) at (root) {};
			\end{tikzpicture}  ) \circ_1 \tau_1) \circ_2 \tau_2
		\end{equation*}
		where the order of the compositions $ \circ_1, \circ_2 $ does not matter.
		We have
		\begin{equation*}
			\mathrm{d} \left( \begin{tikzpicture}[scale=0.2,baseline=-5]
				\coordinate (root) at (0,-1);
				\coordinate (center) at (0,1);
				\draw[] (root) -- (center);
				\node[var1] (rootnode) at (center) {\tiny{$ 2 $}};
				\node[var1] (rootnode) at (root) { \tiny{$  1 $}};
			\end{tikzpicture} \right) = \mathrm{d} \left(  \begin{tikzpicture}[scale=0.2,baseline=-5]
				\coordinate (root) at (0,-1);
				\coordinate (center) at (0,1);
				\draw[] (root) -- (center);
				\node[var1] (rootnode) at (center) {\tiny{$ 1 $}};
				\node[var1] (rootnode) at (root) { \tiny{$  2 $}};
			\end{tikzpicture} \right) =  \mathrm{d} \left( -\frac{1}{2} \begin{tikzpicture}[scale=0.2,baseline=-5]
				\coordinate (root) at (0,-1);
				\coordinate (right) at (1,1);
				\coordinate (left) at (-1,1);
				\draw[kernels2] (root) -- (left);
				\draw[kernels2] (root) -- (right);
				\node[var1] (rootnode) at (left) {\tiny{$ 1 $}};
				\node[var1] (rootnode) at (right) {\tiny{$ 2 $}};
				\node[not] (rootnode) at (root) {};
			\end{tikzpicture} \right) = \begin{tikzpicture}[scale=0.2,baseline=-5]
				\coordinate (root) at (0,-1);
				\coordinate (right) at (1,1);
				\coordinate (left) at (-1,1);
				\draw[] (root) -- (left);
				\draw[] (right) -- (root);
				\node[var1] (rootnode) at (left) {\tiny{$ 1 $}};
				\node[var1] (rootnode) at (right) {\tiny{$ 2 $}};
				\node[diff] (rootnode) at (root) {$  $};
			\end{tikzpicture}.
		\end{equation*} 
		so the element \eqref{operad_covariant}
		is in the kernel of $\mathrm{d}$. Since we know that the homology operad is concentrated in degree zero and is generated by binary operations, this element is precisely the generator of $\LA$. So all elements of $\LA$  are obtained as linear combinations of iterations of covariant derivatives.
		Then, one obtains the main result of \cite[Theorem 1.6]{BD24}
		\begin{theorem}\label{th:LA-operadic}
			We have an operad isomorphism $ \ker\hat{\Phi}_{\geo} \cong \LA$. Moreover, that isomorphism identifies $\ker\hat{\Phi}_{\geo}$ as the linear span of iterations of covariant derivatives. 
		\end{theorem}
		One can compute the dimension of this geometric space using the generating series in \cite[Section 4.4]{BD24} for both the Gaussian and the non-Gaussian cases. Let us mention that with this operadic and homological approach, one recovers the result on multi-indices (see \cite[Remark 4.11]{BD24}). An important application of Theorem \ref{th:LA-operadic} arises in the context of quasi-linear SPDEs where it is easy to verify on iterated covariant derivatives that some counter-terms are local (see \cite[Sections 3.3, 3.4]{BGN24})

		We finish these lecture notes by briefly mentioning an important property one wants on the renormalisation constants for the chain rule symmetry. The counter-terms of the geometric KPZ equation are given by
		\begin{equation*}
			F_{\Gamma,\sigma}(C_{\varepsilon}^g)(u_{\varepsilon}), \quad 	C_{\varepsilon}^g = 	\sum_{\tau \in \mathcal{T}_{\text{\tiny{Chr}}}^g} C_{\varepsilon}(\tau) \tau
		\end{equation*}
		where $ \mathcal{T}_{\text{\tiny{Chr}}}^g $ are the Chirstoffel trees with negative degree and with pairings over the noises. Then, one decomposes this counter-term into
		\begin{equation*}
			C_{\varepsilon}^g = C_{\varepsilon}^{\geo} + C_{\varepsilon}^{\ngeo}
		\end{equation*}
		with $C_{\varepsilon}^{\geo} \in V_{\geo}$ and $C_{\varepsilon}^{\ngeo} \in V_{\geo}^{\perp}$. Then, one wants $ C_{\varepsilon}^{\ngeo} $ to converge to a finite limit independent of the mollifier $\varrho$. One can find elements in the renormalisation group parametrised by the $ C_{\varepsilon} $ such that
		the counter-terms take values only in $V_{\geo}$. This is proved in \cite[Proposition 3.9]{BGHZ} with the following main idea. One proceeds via the absurd assuming that $ C_{\varepsilon}^{\ngeo} $ is not bounded. One can then introduce a  sequence $(r_{\varepsilon})_{\varepsilon \geq 0}$ converging to zero such that  $ r_{\varepsilon}^k C_{\varepsilon}^{\ngeo,k} $ converges to a finite limit. Here, $k$ denotes the number of noises and one has $ C_{\varepsilon}^{\ngeo} = \sum_{k \in \mathbb{N}} C_{\varepsilon}^{\ngeo,k} $
		\begin{itemize}
			\item One multiplies the space-white noise with the sequence $(r_{\varepsilon})_{\varepsilon \geq 0}$. This turns out to change $ \sigma$ into $ r_{\varepsilon} \sigma $. Then 
			\begin{equation*}
				F_{\Gamma, r_{\varepsilon} \sigma}(C_{\varepsilon}^{\ngeo}) =  \sum_{k \in \mathbb{N}} 	F_{\Gamma,  \sigma}(r_{\varepsilon}^kC_{\varepsilon}^{\ngeo,k} )
			\end{equation*}
			converges to a finite limit.
			\item One passes to the limit on $\varepsilon$ going to the deterministic dynamic which is geometric. The remaining counter-terms in $ V_{\geo}^{\perp}$ are force to be zero via an injectivity in law of the solution (see \cite[Theorem 3.5]{BGHZ}).
		\end{itemize}
		


\end{document}